\UseRawInputEncoding
\documentclass[11pt]{article}
\usepackage{amsmath}
\usepackage{amssymb,amsbsy,amsthm,dsfont}
\usepackage{graphicx}
\usepackage[dvipsnames]{xcolor}
\usepackage{caption}
\usepackage{subcaption}
\usepackage{enumerate}
\usepackage{enumitem}
\usepackage{cases}
\usepackage[margin=1in]{geometry}
\usepackage{hyperref}
\allowdisplaybreaks[1]
\numberwithin{equation}{section}
\usepackage{import}
\usepackage{xifthen}
\usepackage{pdfpages}
\usepackage{transparent}

\newcommand{\revision}[1]{#1}

\DeclareMathOperator{\R}{\mathbb{R}} 
\DeclareMathOperator{\Z}{\mathbb{Z}} 

\newcommand{\p}{\mathbb{P}} 
\newcommand{\E}{\mathbb{E}} 

\DeclareMathOperator*{\argmin}{arg\,min} 

\newcommand{\norm}[1]{\| #1 \|}

\newcommand{\val}{\mathsf{val}}

\def\cA{{\mathcal A}}
\def\cB{{\mathcal B}}

\def\cE{{\mathcal E}}

\def\cN{{\mathcal N}}

\def\cS{{\mathcal S}}
\def\cT{{\mathcal T}}

\newcommand{\cut}[1]{}

\newtheorem{theorem}{Theorem}[section]
\newtheorem{lemma}[theorem]{Lemma}
\newtheorem{proposition}[theorem]{Proposition}

\newtheorem{corollary}[theorem]{Corollary}
\newtheorem{conjecture}[theorem]{Conjecture}

\newtheorem{definition}[theorem]{Definition}
\newtheorem{assumption}[theorem]{Assumption}

\newtheorem{remark}[theorem]{Remark}

\begin{document}

\title{Quickest Inference of Network Cascades with Noisy Information\thanks{This work was supported by the United States National Science Foundation (NSF) under RAPID Grant IIS-2026982. This article was presented in part at the 54th Asilomar Conference on Signals, Systems and Computers \cite{sridhar_poor_sequential} and at the 2021 IEEE International Conference on Acoustics, Speech and Signal Processing (ICASSP) \cite{sridhar_poor_bayes}. 

A. Sridhar and H. V. Poor are with the Department of Electrical and Computer Engineering, Princeton University, Princeton, NJ 08544 (email: anirudhs@princeton.edu, poor@princeton.edu).}
}
\author{
	Anirudh Sridhar, H. Vincent Poor
}
\date{\today}

\maketitle


\begin{abstract}
\revision{We study the problem of estimating the source of a network cascade given a time series of noisy information about the spread. Initially, there is a single vertex affected by the cascade (the source) and the cascade spreads in discrete time steps across the network. Although the cascade evolution is hidden, one observes a noisy measurement of the evolution at each time step. Given this information, we aim to reliably estimate the cascade source as fast as possible. 

We investigate Bayesian and minimax formulations of the source estimation problem, and derive near-optimal estimators for simple cascade dynamics and network topologies. In the Bayesian setting, samples are taken until the error of the Bayes-optimal estimator falls below a threshold. For the minimax setting, we design a novel multi-hypothesis sequential probability ratio test. These optimal estimators require $\log \log n / \log (k - 1)$ observations for a $k$-regular tree network, and $(\log n)^{\frac{1}{\ell + 1}}$ observations for a $\ell$-dimensional lattice. We then discuss conjectures on source estimation in general topologies. Finally, we provide simulations which validate our theoretical results on trees and lattices, and illustrate the effectiveness of our methods for estimating the sources of cascades on Erd\H{o}s-R\'{e}nyi graphs.}
\end{abstract}

\section{Introduction}
Network-based interactions lie at the core of many dynamic systems, including social behavior, biological processes and wireless communications. Unfortunately, the decentralized nature of networks often make them susceptible to {\it cascading failures} in which behaviors or information originating from a small subset of nodes diffuse rapidly throughout the rest of the network. Examples include viral spread in contact networks (see e.g., \cite{brauer2012mathematical}), misinformation in social networks \cite{conspiracy, fourney2017geographic, fighting_spam,  tacchini2017some} and malware in cyber-physical networks \cite{KW91, MI17,  wang_viral_modeling, codered_propagation}. In all of these scenarios, the rapid spread of the cascade can have devastating effects. It is therefore of the utmost importance to track the cascade and contain it as fast as possible. 

A fundamental challenge in accomplishing this task is that information about the cascade is usually noisy or uncertain in real-time settings. To illustrate this point, suppose that a virus is spreading over a contact network. When the number of individuals is large, it may be infeasible to force everyone to quarantine, hence diagnostic tests may be administered to track and contain the spread. If there are not enough diagnostic tests to test the entire population at a given point in time, there is uncertainty in the status of individuals who are not tested. Moreover, diagnostic tests are typically not perfectly accurate, so even among the tested individuals there may be false positives and negatives.  

Nevertheless, by observing the results of many rounds of testing over time, it is natural to expect that one can accurately estimate the spread of the virus using the right testing and information aggregation strategies. On the other hand, if one waits too long to obtain reasonable estimates, the cascade will spread to a large subset of the population, which is undesirable. The goal of this work is to characterize inference algorithms which achieve the {\it optimal} tradeoff between the estimation error and the time until estimation. Moreover, we study how the structure of the underlying network influences the design and performance of such algorithms. 
\subsection{Summary of contributions}

\revision{For the most part, existing theoretical work on estimating the source of a network cascade takes the perspective of a \emph{reconstruction} problem: given a large, \emph{known} set of infected nodes, the goal is to identify the source among them \cite{KhimLoh15, fanti2015spy, racz_source_detection, wang2014, ying_book, shah2010detecting, shah2011rumors, shah2012finding, zhu_ying_1, zhu_ying_2}. In contrast, we study source estimation from the novel perspective of \emph{real-time inference}: by monitoring real-time signals from each node, we aim to find the source \emph{before} the number of affected nodes is large. The two paradigms of source estimation are fundamentally different, and as such, require drastically different models and methods.   
}

\revision{We mathematically formalize the task of real-time source estimation as follows.} Consider a statistical model of network cascades with noisy observations where, at discrete timesteps, each node produces a signal that is an independent sample from a pre-change distribution $Q_0$ if the node has not yet been affected by the cascade, else the signal is an independent sample from a post-change distribution $Q_1$. Initially, a single unknown vertex (the source) is affected by the cascade, and the cascade propagates to neighbors of affected vertices at each timestep. Our objective is to design algorithms that estimate the unknown source as fast as possible. \revision{We provide the first solution to this problem, to the best of our knowledge,} and derive optimal source estimators from Bayesian and minimax perspectives. 

To develop a concrete characterization of optimal source estimators, we focus on simple cascade dynamics and networks. The cascade dynamics we consider are {\it deterministic}: at each timestep, the cascade spreads to all neighbors of currently-affected nodes. We assume the network topology is either a $k$-regular tree or a $\ell$-dimensional lattice; we do so because such networks are simple to describe, they represent a diverse family of topologies, and they enjoy convenient symmetry properties which simplify our analysis considerably. We further assume that there is a known set of $n$ {\it candidate nodes} which contains the unknown source. When $n$ is large, we show that at least $\log \log n / \log (k-1)$ timesteps of noisy observations are required for reliable source estimation in $k$-regular trees, while at least $(\log n)^{\frac{1}{\ell + 1}}$ timesteps are required for $\ell$-dimensional lattices. 

We then derive {\it optimal} estimation algorithms whose performance matches the lower bounds described above. We show that the optimal algorithm in the Bayesian formulation of the source estimation problem is a simple procedure that continues to observe noisy observations of the cascade propagation until the Bayes-optimal estimator is sufficiently accurate. In the minimax formulation, we phrase source estimation as a $n$-ary hypothesis testing problem among the $n$ candidate nodes and show that a natural test based on likelihood ratios -- called the multi-hypothesis sequential probability ratio test (MSPRT) -- is optimal. Interestingly, the design of the MSPRT which matches the lower bounds can be viewed as a {\it multi-scale search procedure}: it simultaneously identifies the general area of the source while also performing a local, fine-grained analysis to obtain more precise estimates. 

Admittedly, our setting of deterministic cascade dynamics on regular trees or lattices is simplistic compared to more realistic cascade and network models \cite{cascades_review, configuration_model, random_geometric_graphs, BA99, Mah92,BA99,BRST01}. However, we find that our setting leads to a mathematically rich problem and serves as an important starting point for understanding the source estimation problem for more complex propagation dynamics and networks. On a more technical level, we present a mostly unified treatment of optimal source estimation algorithms on regular trees and lattices, with only minor differences between the two. This suggests that our methods could be generalized to describe optimal source estimators for {\it arbitrary topologies}, though this requires significantly more effort so we leave it to future work. We discuss in detail the potential extensions of our work to arbitrary topologies, providing conjectures on the structure and performance of optimal algorithms. \revision{Finally, we assess the performance of the estimators we develop through simulations. In addition to validating our theory for tree and lattice topologies, we show that our estimators perform well on natural models of random networks (Erd\H{o}s-R\'{e}nyi graphs). Strikingly, even when there is a moderate amount of noise in vertex signals, our estimators can reliably locate the cascade source in Erd\H{o}s-R\'{e}nyi graphs before 40 vertices are infected for networks with up to 2000 vertices. This provides strong evidence that our methods may be applicable quite broadly.}

\subsection{Related work}

\noindent {\bf Source estimation from a noiseless snapshot.} Perhaps the most well-known work on estimating the source of a cascade is by Shah and Zaman \cite{shah2010detecting, shah2011rumors, shah2012finding}. In their formulation of the problem, the cascade spreads {randomly} via the Susceptible-Infected process, and a single snapshot of the set of infected vertices is observed at a later point in time. They derive an expression for the maximum likelihood estimate of the source in trees and study properties of the estimator. Many authors have expanded on these ideas and methods in subsequent work, studying for instance the effect of multiple observations, \revision{multiple sources}, confidence sets for the source, different network models, and different cascade models \cite{KhimLoh15, fanti2015spy, racz_source_detection, wang2014, ying_book, zhu_ying_1, zhu_ying_2}. \revision{We emphasize that while this literature on source estimation is similar in spirit to the problem we consider in this paper, it is fundamentally different from modeling and algorithmic perspectives. For instance, the literature cited above is of a {\it static} nature, where we have a single (or a fixed number of) perfect-information snapshot(s) of a large cascade. On the other hand, we consider {\it dynamic} settings where we obtain noisy and incomplete measurements of a small but growing cascade. Moreover, the methods developed in the literature cited above (e.g., rumor centrality, Jordan centrality) have no obvious counterpart in our setting, since they are computed based on known infections. However, in the model of noisy, real-time measurements considered in this paper, it is impossible to know exactly which vertices are infected.}  \\

\noindent {\bf Cascade inference from a noisy time series.} A growing body of literature uses the data model \eqref{eq:public_signals} to perform inference of cascades, including detecting the presence of a cascade \cite{zou2020qd, ZouVeeravalli2018, ZVLT2018, Rovatsos2020quickest, Rovatsos2020heterogeneous, halme2021bayesian, qcd_empirical}, estimating the source \cite{sridhar_poor_sequential, sridhar_poor_bayes} and controlling its spread \cite{hoffman_thesis, hoffman_cost, active_screening}. The closest work to ours in terms of methods and analysis is by Zou, Veeravalli, Li and Towsley \cite{zou2020qd}, who studied the following quickest {\it detection} problem: a cascade spreads via unknown dynamics, and the goal is to stop sampling once the cascade affects a given number of vertices. Their test, which can be viewed as an adaptation of the CUSUM procedure, is agnostic to the spreading dynamics of the cascade and is optimal in the regime where samples are taken much frequently than the growth of the cascade. By contrast, we consider the regime of large networks and where samples are taken at a comparative rate to the growth of the cascade. Moreover, our results reveal the effect of the network topology on the performance of inference procedures, which is not the case in \cite{zou2020qd}. 

Finally, we remark that compared to our prior conference submissions on the source estimation problem \cite{sridhar_poor_sequential, sridhar_poor_bayes}, the current paper provides a unified and substantially more general solution. In particular, \cite{sridhar_poor_bayes} only provided a Bayesian solution for lattices and \cite{sridhar_poor_sequential} established results for the minimax setting under a somewhat unnatural, but mathematically simpler, constraint on the stopping time and estimator. 

\subsection{Notation}
\label{sec:notation}

Let $\R$ and $\Z$ denote the set of reals and integers, respectively. For a graph $G = (V, E)$, let $V$ denote the set of vertices and let $E$ denote the set of edges. For $u,v \in V$, $d(u,v)$ represents the shortest path distance between $u$ and $v$ in $G$. For $v \in V$ and a non-negative integer $s$, $\cN_v(s)$ is the $s$-hop neighborhood of $v$; that is, $\cN_v(s) : = \{ u \in V: d(u,v) \le s \}$. 

We utilize standard asymptotic notation throughout. In particular, for two functions $g(n)$ and $h(n)$, we say $g(n) \lesssim h(n)$ if there is $c > 0$ such that for $n$ sufficiently large, $g(n) \le c h(n)$. We say $g(n) \asymp h(n)$ (in words, $g(n)$ and $h(n)$ are orderwise equivalent) if and only if there are $c_1, c_2 > 0$ such that $c_1 h(n) \le g(n) \le c_2 h(n)$ for $n$ sufficiently large. We say $g(n) \sim h(n)$ (in words, $g(n)$ is equal to $h(n)$ up to first-order terms) if $\lim_{n \to \infty} g(n) / h(n) = 1$. 

\subsection{Organization}
The rest of the paper is organized as follows. In Section \ref{sec:model}, we formally define our model of cascade evolution with noisy observations, as well as the Bayesian and minimax optimality criteria. In Section \ref{sec:results}, we provide a description and overview of our results on optimal estimation in regular trees and lattices, as well as a discussion on how one might extend our techniques to general topologies. \revision{In Section \ref{sec:simulations}, we provide numerical results on the performance of optimal estimators from simulations on trees, lattices and Erd\H{o}s-R\'{e}nyi graphs.}The remaining sections are devoted to the proofs of our main results. The proofs of main results on the Bayesian setting are in Section \ref{sec:bayes}, and proofs for the main results in the minimax setting are in Section \ref{sec:minimax}. Sections \ref{sec:estimation_error} and \ref{sec:msprt} contain supporting results for the proofs in Sections \ref{sec:bayes} and \ref{sec:minimax}. We conclude in Section \ref{sec:conclusion}. Additional combinatorial results concerning the topology of regular trees and lattices can be found in Appendix \ref{sec:size_of_neighborhoods} and \ref{sec:geodesics}.

\section{Problem formulation}
\label{sec:model}

We begin by describing the most general formulation of the source estimation problem. Let $G$ be a graph with vertex set and edge set given by $V$ and $E$, respectively. Initially, a single vertex $v^* \in V$ is affected by the cascade; we call this vertex the {\it cascade source}. From $v^*$, the cascade spreads over time via the edges of the graph according to a known random or deterministic discrete-time process. Examples of cascade dynamics include variants of the susceptible-infected (SI) process, the independent cascade model and the linear threshold model (see \cite{cascades_review} and references therein).

For any $v \in V$ and any time index $t \ge 0$, let $x_v(t) \in \{0,1\}$ denote the {\it private state} of $v$, where $x_v(t) = 1$ if $v$ is affected by the cascade at time $t$, otherwise $x_v(t) = 0$. The private states are not observable, but the system instead monitors the {\it public signals} $\{y_u(t) \}_{u \in V}$, defined as
\begin{equation}
\label{eq:public_signals}
y_u(t) \sim \begin{cases}
Q_0 & x_u(t) = 0; \\
Q_1 & x_u(t) = 1, \\
\end{cases}
\end{equation}
where $Q_0$ and $Q_1$ are two mutually absolutely continuous probability measures. We can think of $y_u(t) \sim Q_0$ being typical behavior and $y_u(t) \sim Q_1$ as anomalous behavior caused by the cascade. As a shorthand, we denote $y(t) : = \{ y_u(t) \}_{u \in V}$ to be the collection of all public states at time $t$. See Figure \ref{fig:data_model} for an illustration of this data model. We remark that this data model has been studied in recent literature in the context of cascade source estimation \cite{sridhar_poor_sequential,sridhar_poor_bayes}, quickest detection of cascades \cite{zou2020qd, ZouVeeravalli2018, ZVLT2018, Rovatsos2020quickest, Rovatsos2020heterogeneous, halme2021bayesian, qcd_empirical}, and control of cascades \cite{hoffman_cost, active_screening, hoffman_thesis}. 

\begin{figure}[ht]
	\centering
	
	\def \svgwidth{0.5\columnwidth}
\begingroup%
  \makeatletter%
  \providecommand\color[2][]{%
    \errmessage{(Inkscape) Color is used for the text in Inkscape, but the package 'color.sty' is not loaded}%
    \renewcommand\color[2][]{}%
  }%
  \providecommand\transparent[1]{%
    \errmessage{(Inkscape) Transparency is used (non-zero) for the text in Inkscape, but the package 'transparent.sty' is not loaded}%
    \renewcommand\transparent[1]{}%
  }%
  \providecommand\rotatebox[2]{#2}%
  \newcommand*\fsize{\dimexpr\f@size pt\relax}%
  \newcommand*\lineheight[1]{\fontsize{\fsize}{#1\fsize}\selectfont}%
  \ifx\svgwidth\undefined%
    \setlength{\unitlength}{841.88976378bp}%
    \ifx\svgscale\undefined%
      \relax%
    \else%
      \setlength{\unitlength}{\unitlength * \real{\svgscale}}%
    \fi%
  \else%
    \setlength{\unitlength}{\svgwidth}%
  \fi%
  \global\let\svgwidth\undefined%
  \global\let\svgscale\undefined%
  \makeatother%
  \begin{picture}(1,0.70707071)%
    \lineheight{1}%
    \setlength\tabcolsep{0pt}%
    \put(0,0){\includegraphics[width=\unitlength,page=1]{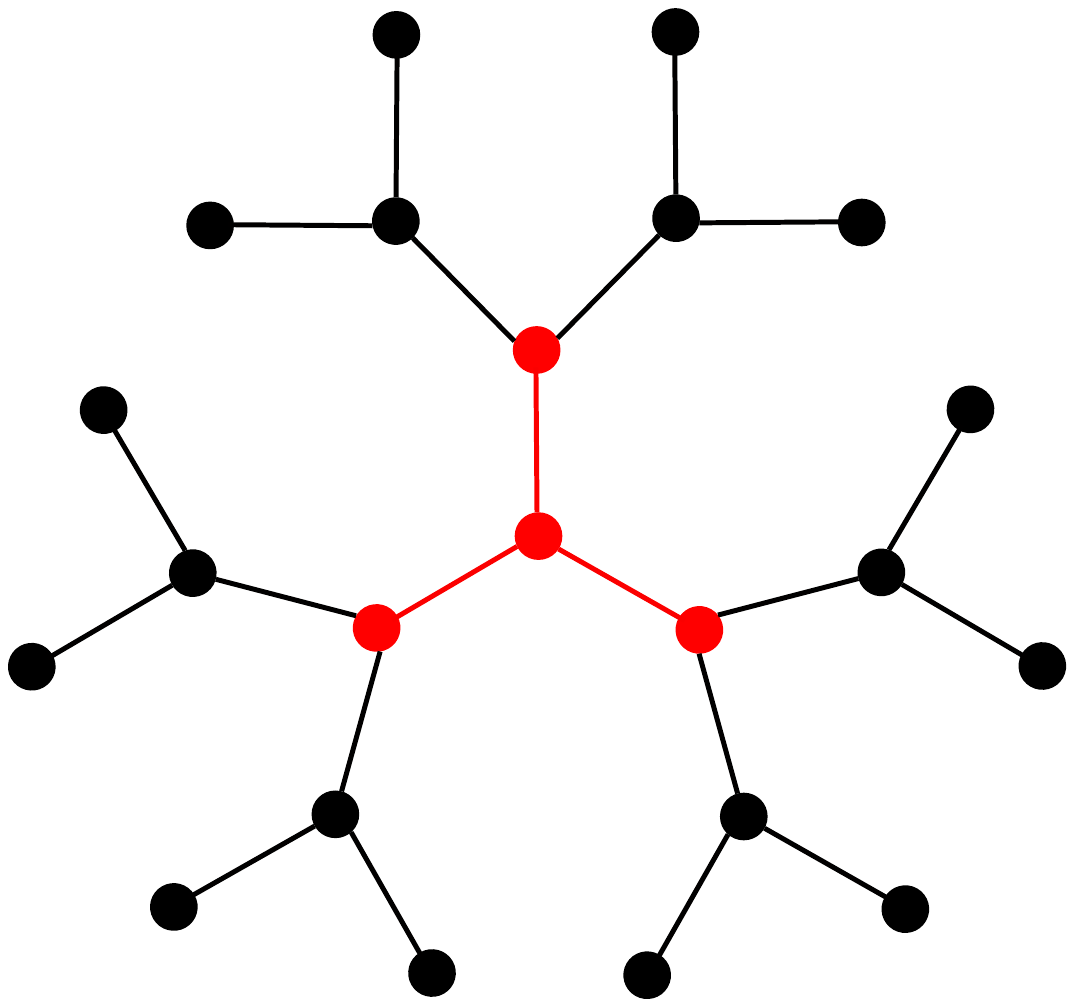}}%
    \put(0.55360562,0.39676643){\color[rgb]{0,0,0}\makebox(0,0)[lt]{\lineheight{1.25}\smash{\begin{tabular}[t]{l}$u$\end{tabular}}}}%
    \put(0.6053958,0.52164258){\color[rgb]{0,0,0}\makebox(0,0)[lt]{\lineheight{1.25}\smash{\begin{tabular}[t]{l}$v$\end{tabular}}}}%
    \put(0.10812687,0.04639006){\color[rgb]{0,0,0}\makebox(0,0)[lt]{\lineheight{1.25}\smash{\begin{tabular}[t]{l}\LARGE \color{red} $y_u(t) \sim Q_1$\end{tabular}}}}%
    \put(0.62798666,0.04638996){\color[rgb]{0,0,0}\makebox(0,0)[lt]{\lineheight{1.25}\smash{\begin{tabular}[t]{l}\LARGE $y_v(t) \sim Q_0$\end{tabular}}}}%
  \end{picture}%
\endgroup%

	\caption{Illustration of the data model at a given time $t$. The nodes in red have been affected by the cascade, and the black nodes are unaffected, though this information is hidden from the observer. The public signals of the red nodes are sampled from $Q_1$, while the public signals of the black nodes are sampled from $Q_0$.}
	\label{fig:data_model}
\end{figure}

\begin{remark}
\label{remark:data_model}
The data model in \eqref{eq:public_signals} can capture a variety of realistic scenarios. In the context of viral spread for instance, a common symptom of sickness is a fever. If the public signals correspond to the body temperature of individuals in the population, one may expect that $y_v(t)$ will be close to the typical body temperature of the individual represented by vertex $v$ if they do not carry the virus, else $y_v(t)$ is expected to be significantly higher if the individual does carry the virus. 

Another practical example of \eqref{eq:public_signals} is {\it diagnostic testing with errors}, which is used for malware detection in computer networks \cite{malware_detection} and tracking the spread of infectious diseases \cite{viral_test}. Suppose that at a given point in time, each vertex is given a diagnostic test with probability $p$, independently over all vertices. If a test is taken, the output is either 0 (the vertex is not affected) or 1 (the vertex is affected). With probability $\epsilon$, the result of the test will be incorrect. To formulate this in terms of \eqref{eq:public_signals}, let the support of $Q_0$ and $Q_1$ be $\{0, 1, \times\}$, where 0 indicates a test result of 0, 1 indicates a test result of 1, and $\times$ indicates that a test was not taken. The distributions $Q_0 : = (q_0(0), q_0(1), q_0(\times))$ and $Q_1 : = (q_1(0), q_1(1), q_1(\times))$ are given by
$$
\begin{cases}
q_0(0) = p(1 - \epsilon) & \\
q_0(1) = p\epsilon & \\
q_0(\times) = 1 - p
\end{cases}
\qquad 
\text{and} 
\qquad
\begin{cases}
q_1(0) = p\epsilon & \\
q_1(1) = p(1 - \epsilon) & \\
q_1(\times) = 1 - p.
\end{cases}
$$
\end{remark}

Given the data model \eqref{eq:public_signals}, the problem of estimating the cascade source can be phrased as a sequential multi-hypothesis testing problem: given the collection of hypotheses $\{H_v \}_{v \in V}$ where $H_v$ is the hypothesis that $v$ is the source, our goal is to output a hypothesis with a small probability of error. At the same time, it is also important that we come to a decision as fast as possible in order to minimize the number of vertices affected by the cascade. This reveals a fundamental tradeoff: when more samples are taken, one can obtain more reliable estimates of the source at the cost of allowing the cascade to spread further. An {\it optimal} procedure will achieve the best possible tradeoff between the estimation error and the number of samples needed. 

We shall proceed by formalizing these ideas. Observe that any source estimator can be represented by the pair $(T, \widehat{\mathbf{v}})$, where $T$ is a stopping time indicating when to stop sampling and $\widehat{\mathbf{v}} : = \{ \widehat{v}(t) \}_{t \ge 0}$ is a sequence of source estimators so that $\widehat{v}(t)$ is the estimate of the source given the data at time $t$. The final source estimate produced by $(T, \widehat{\mathbf{v}})$ is $\widehat{v}(T)$. We shall also assume that a {\it candidate set} $U \subset V$ is known, so that the unknown source is an element of $U$. We remark that the size of $U$, denoted by $|U|$, measures in a sense the initial uncertainty around the source location. As a matter of notation, we let $\p_v$ be the probability measure corresponding to the hypothesis $H_v$ ($v$ is the source). Similarly, $\E_v$ denotes the expectation with respect to the measure $\p_v$. If the location of the source is given by a probability distribution $\pi : = \{\pi_v\}_{v \in V}$ (where $\pi_v$ is the probability that the source is $v$), we denote $\p_\pi$ and $\E_\pi$ to be the probability measure and expectation operator with respect to $\pi$, respectively. Formally, we may write 
\begin{equation}
\label{eq:prob_expectation_definition}
\p_\pi(\cdot) := \sum\limits_{v \in V} \pi_v \p_v(\cdot) \qquad \text{and} \qquad \E_\pi[\cdot] : = \sum\limits_{v \in V} \pi_v \E_v [ \cdot ].
\end{equation}
\revision{We remark that often in this paper, we will consider the operator $\E_{\pi(t)}[ \cdot ]$ where $\pi(t)$ is the posterior distribution of the source after observing the public signals $y(0), \ldots, y(t)$. In such a case, $\E_{\pi(t)}[ \cdot ]$ would be a random variable, since it is equal to the conditional expectation $\E[ \cdot \vert y(0), \ldots, y(t) ]$. 
}

We next define the performance metrics used the evaluate the effectiveness of a source estimator. For a source estimator $(T, \widehat{\mathbf{v}})$, we shall study the {\it expected number of samples}, given by $\E_v [ T]$ when $v$ is the source. The {\it estimation error} is the expected distance between $\widehat{v}(T)$ and the source, given by $\E_v [ d(v, \widehat{v}(T)) ]$ when $v$ is the source. \revision{Here, we recall from Section \ref{sec:notation} that $d(\cdot, \cdot)$ denotes the shortest-path distance between two vertices in $G$.}

We study two natural ways to capture the tradeoff between estimation error and the expected number of samples. \\

\noindent {\it A Bayesian perspective.} Denote the source vertex by $v^*$, and suppose that the prior distribution for the source is uniform over the elements of the candidate set $U$; we denote this prior by $\pi(U)$. We say that the optimal procedure solves the following optimization problem:
\begin{equation}
\label{eq:bayes_opt}
\inf\limits_{T, \widehat{\mathbf{v}}} \E_{\pi(U)} \left [ d(v^*, \widehat{v}(T) ) + T \right] \revision{= \inf_{T, \widehat{\mathbf{v}}} \frac{1}{|U|} \sum_{v \in U} \E_v \left[ d(v, \widehat{v}(T) ) + T \right]},
\end{equation}
where we recall that $\E_{\pi(U)}$ denotes the expectation operator with respect to the measure $\pi(U)$. In words, \eqref{eq:bayes_opt} is the sum of the estimation error and the expected number of samples. If only the first term in \eqref{eq:bayes_opt} was present, the optimal strategy would be to set $T = \infty$, since more samples can only help in bringing down the estimation error. On the other hand, if only the second term in \eqref{eq:bayes_opt} was present, the optimal strategy would be to set $T = 0$. The estimator that solves \eqref{eq:bayes_opt} therefore achieves the best tradeoff between the two extremes. We remark that it is standard in Bayesian formulations of sequential testing problems to minimize the sum of the error and expected number of samples \cite{poor_hadjiliadis_2008, poor_detection_estimation}. \revision{Furthermore, we remark that one may consider other ways to quantify the tradeoff between estimation error and time -- for instance, by replacing $T$ with $h(T)$ for some increasing function $h$. While we focus on the formulation \eqref{eq:bayes_opt} for simplicity and ease of exposition, the methods we develop can also handle a large class of functions $h$.} \\

\noindent {\it A minimax perspective.} As an alternative to the Bayesian approach, one can formalize the tradeoff between the estimation error and expected number of samples via the following optimization problem: 
\begin{equation}
\label{eq:np}
\inf\limits_{T, \widehat{\mathbf{v}}} \max\limits_{v \in U} \E_v [ T] \qquad \text{subject to } \max\limits_{v \in U} \E_v [ d(v, \widehat{v}(T) ) ] \le \alpha,
\end{equation} 
where $\max_{v \in U} \E_v [ d(v, \widehat{v}(T)) ]$ is the {\it worst-case estimation error}, $\alpha$ is a specified bound on the {worst-case estimation error} and $\max_{v \in U} \E_v [ T]$ is the {\it worst-case expected runtime} of the procedure. As in the Bayesian case, we may consider two extremes. When $\alpha = \infty$, the optimal choice is $T = 0$, whereas when $\alpha = 0$ the optimal choice is $T = \infty$.\footnote{More precisely, if there exists an estimator $\widehat{\mathbf{v}}$ such that $\widehat{v}(t) \to v^*$ as $t \to \infty$, then the stopping time $T = \infty$ is optimal. If such an estimator does not exist, there is no feasible solution to \eqref{eq:np} when $\alpha = 0$.} For intermediate values of $\alpha$, the optimal algorithm indeed achieves a tradeoff between the estimation error and the worst-case expected runtime.

\section{Results}
\label{sec:results}

The goal of our work is to characterize optimal estimators based on the formulations in  \eqref{eq:bayes_opt} and \eqref{eq:np}. We are particularly interested in how the structure and performance of optimal estimators depend on the network topology. In order to provide a tractable theoretical analysis, we focus on simple networks and cascade dynamics. The cascade dynamics we consider is outlined in the following assumption.  

\begin{assumption}[Cascade dynamics]
\label{as:cascade}
Initially, a single vertex $v^*$ (the source) is affected by the cascade. The cascade then spreads deterministically in discrete time steps, so that vertex $v$ is affected by the cascade at time $t$ if and only if $d(v, v^*) \le t$. 
\end{assumption}

\noindent We consider two classes of networks -- regular trees and lattices -- which are defined formally below. 

\begin{definition}[Infinite $k$-regular tree]
\label{defn:tree}
Let $v_r$ be a designed root vertex, and let $\cT_k(1)$ be the tree with $k$ leaves attached to $v_r$. Given $\cT_k(m)$, we construct $\cT_k(m + 1)$ by attaching $k-1$ leaves to each leaf in $\cT_k(m)$. The infinite $k$-regular tree $\cT_k$ is the limiting graph obtained when $m \to \infty$; that is, $(u,v)$ is an edge in $\cT_k$ if and only if $(u,v)$ is an edge in $\cT_k(m)$ for some positive integer $m$. 
\end{definition}

\begin{definition}[Infinite $\ell$-dimensional lattice]
\label{defn:lattice}
Label elements of the vertex set by $\Z^\ell$. There is an edge between vertices $u,v$ in the infinite $\ell$-dimensional lattice if and only if $\sum_{i= 1}^\ell |u_i - v_i| = 1$.\footnote{The 2-regular tree $\cT_2$ is the same as the 1-dimensional lattice. Henceforth, we shall identify $\cT_2$ as the 1-dimensional lattice and always consider $k$-regular trees with $k \ge 3$. Indeed, from our analysis, it can be seen that the relevant properties of $\cT_2$ make the graph most naturally associated with the class of lattices.  }
\end{definition}

\begin{figure}[t]
\centering
\begin{subfigure}{0.3 \textwidth}

	\def \svgwidth{1.0\columnwidth}
\begingroup%
  \makeatletter%
  \providecommand\color[2][]{%
    \errmessage{(Inkscape) Color is used for the text in Inkscape, but the package 'color.sty' is not loaded}%
    \renewcommand\color[2][]{}%
  }%
  \providecommand\transparent[1]{%
    \errmessage{(Inkscape) Transparency is used (non-zero) for the text in Inkscape, but the package 'transparent.sty' is not loaded}%
    \renewcommand\transparent[1]{}%
  }%
  \providecommand\rotatebox[2]{#2}%
  \newcommand*\fsize{\dimexpr\f@size pt\relax}%
  \newcommand*\lineheight[1]{\fontsize{\fsize}{#1\fsize}\selectfont}%
  \ifx\svgwidth\undefined%
    \setlength{\unitlength}{841.88976378bp}%
    \ifx\svgscale\undefined%
      \relax%
    \else%
      \setlength{\unitlength}{\unitlength * \real{\svgscale}}%
    \fi%
  \else%
    \setlength{\unitlength}{\svgwidth}%
  \fi%
  \global\let\svgwidth\undefined%
  \global\let\svgscale\undefined%
  \makeatother%
  \begin{picture}(1,0.70707071)%
    \lineheight{1}%
    \setlength\tabcolsep{0pt}%
    \put(0,0){\includegraphics[width=\unitlength,page=1]{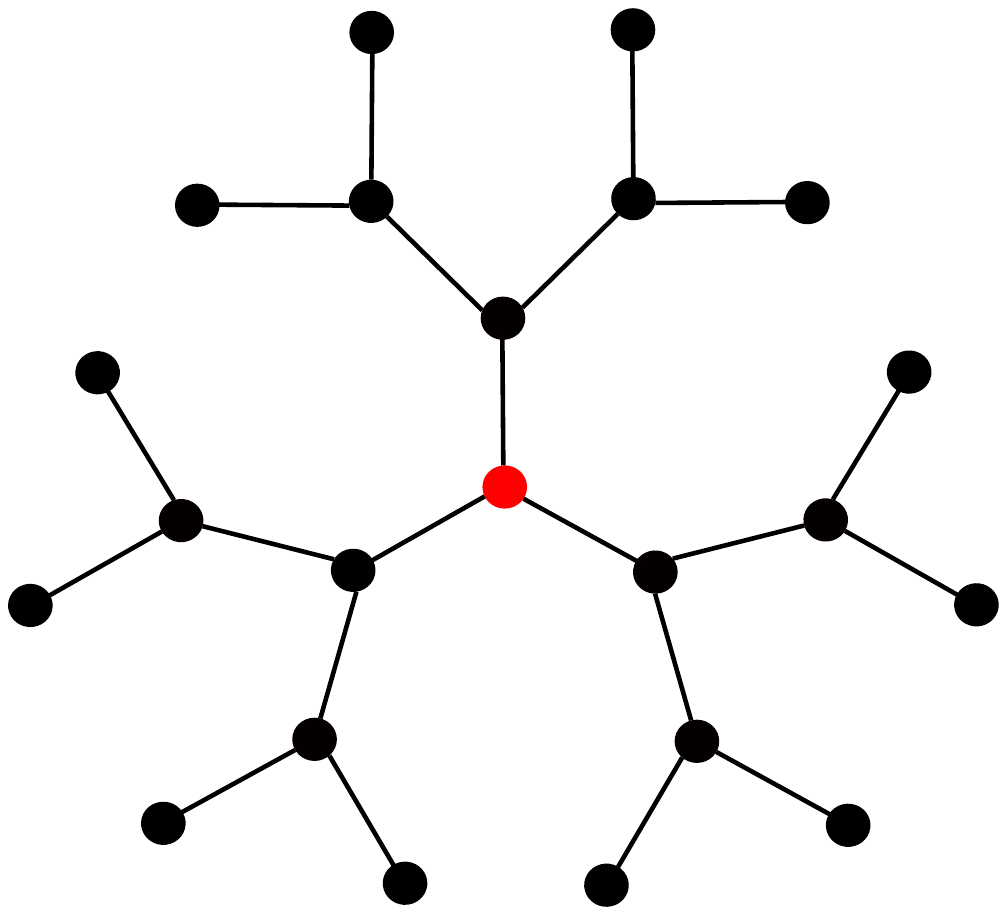}}%
    \put(0.53006259,0.32828801){\color[rgb]{0,0,0}\makebox(0,0)[lt]{\lineheight{1.25}\smash{\begin{tabular}[t]{l}$v^*$\end{tabular}}}}%
    \put(0.58852657,0.42470626){\color[rgb]{0,0,0}\makebox(0,0)[lt]{\lineheight{1.25}\smash{\begin{tabular}[t]{l}$v$\end{tabular}}}}%
    \put(0.04485808,0.00006766){\color[rgb]{0,0,0}\makebox(0,0)[lt]{\lineheight{1.25}\smash{\begin{tabular}[t]{l}\large \color{red} $y_{v^*}(0) \sim Q_1$\end{tabular}}}}%
    \put(0.51623854,0.0000673){\color[rgb]{0,0,0}\makebox(0,0)[lt]{\lineheight{1.25}\smash{\begin{tabular}[t]{l}\large $y_v(0) \sim Q_0$\end{tabular}}}}%
  \end{picture}%
\endgroup%

\caption{$t = 0$}
\end{subfigure}
\begin{subfigure}{0.3 \textwidth}

	\def \svgwidth{1.0\columnwidth}
\begingroup%
  \makeatletter%
  \providecommand\color[2][]{%
    \errmessage{(Inkscape) Color is used for the text in Inkscape, but the package 'color.sty' is not loaded}%
    \renewcommand\color[2][]{}%
  }%
  \providecommand\transparent[1]{%
    \errmessage{(Inkscape) Transparency is used (non-zero) for the text in Inkscape, but the package 'transparent.sty' is not loaded}%
    \renewcommand\transparent[1]{}%
  }%
  \providecommand\rotatebox[2]{#2}%
  \newcommand*\fsize{\dimexpr\f@size pt\relax}%
  \newcommand*\lineheight[1]{\fontsize{\fsize}{#1\fsize}\selectfont}%
  \ifx\svgwidth\undefined%
    \setlength{\unitlength}{841.88976378bp}%
    \ifx\svgscale\undefined%
      \relax%
    \else%
      \setlength{\unitlength}{\unitlength * \real{\svgscale}}%
    \fi%
  \else%
    \setlength{\unitlength}{\svgwidth}%
  \fi%
  \global\let\svgwidth\undefined%
  \global\let\svgscale\undefined%
  \makeatother%
  \begin{picture}(1,0.70707071)%
    \lineheight{1}%
    \setlength\tabcolsep{0pt}%
    \put(0,0){\includegraphics[width=\unitlength,page=1]{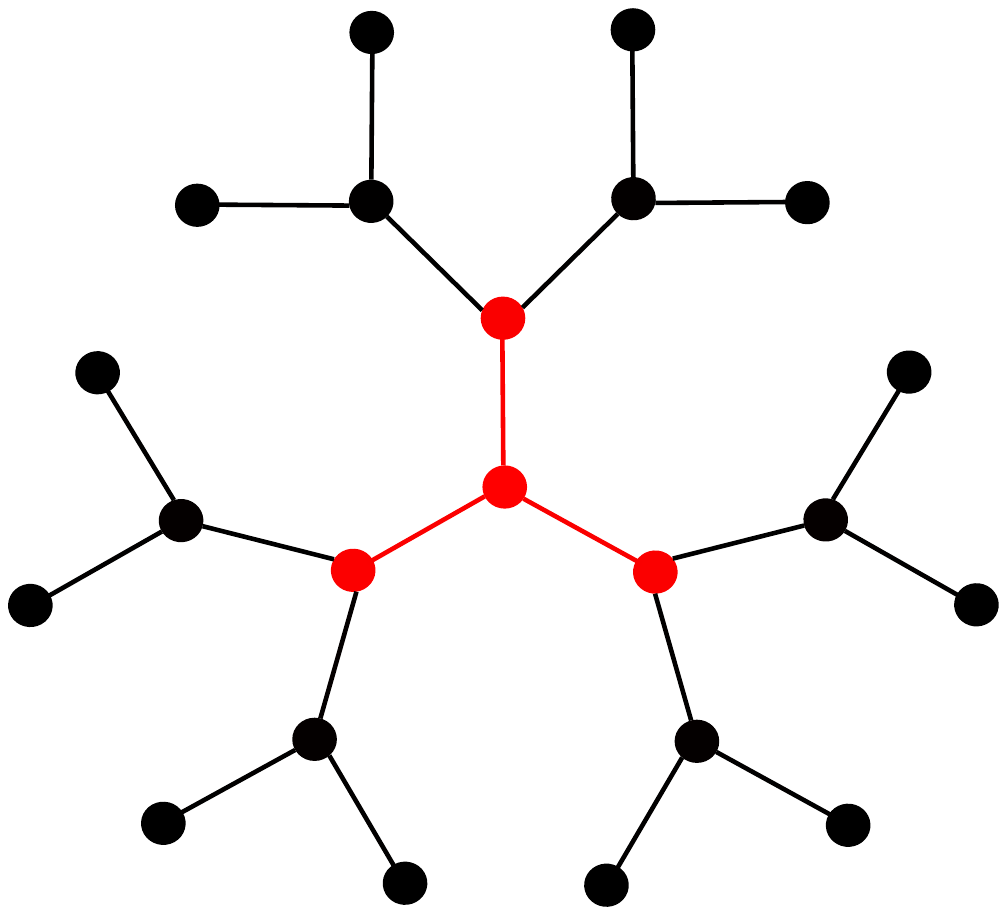}}%
    \put(0.53006259,0.32828801){\color[rgb]{0,0,0}\makebox(0,0)[lt]{\lineheight{1.25}\smash{\begin{tabular}[t]{l}$v^*$\end{tabular}}}}%
    \put(0.58852657,0.42470626){\color[rgb]{0,0,0}\makebox(0,0)[lt]{\lineheight{1.25}\smash{\begin{tabular}[t]{l}$v$\end{tabular}}}}%
    \put(0.04485808,0.00006766){\color[rgb]{0,0,0}\makebox(0,0)[lt]{\lineheight{1.25}\smash{\begin{tabular}[t]{l}\large \color{red} $y_{v^*}(1) \sim Q_1$\end{tabular}}}}%
    \put(0.51623854,0.0000673){\color[rgb]{0,0,0}\makebox(0,0)[lt]{\lineheight{1.25}\smash{\begin{tabular}[t]{l}\large $y_v(1) \sim Q_0$\end{tabular}}}}%
  \end{picture}%
\endgroup%

\caption{$t = 1$}
\end{subfigure}
\begin{subfigure}{0.3 \textwidth}

	\def \svgwidth{1.0\columnwidth}
\begingroup%
  \makeatletter%
  \providecommand\color[2][]{%
    \errmessage{(Inkscape) Color is used for the text in Inkscape, but the package 'color.sty' is not loaded}%
    \renewcommand\color[2][]{}%
  }%
  \providecommand\transparent[1]{%
    \errmessage{(Inkscape) Transparency is used (non-zero) for the text in Inkscape, but the package 'transparent.sty' is not loaded}%
    \renewcommand\transparent[1]{}%
  }%
  \providecommand\rotatebox[2]{#2}%
  \newcommand*\fsize{\dimexpr\f@size pt\relax}%
  \newcommand*\lineheight[1]{\fontsize{\fsize}{#1\fsize}\selectfont}%
  \ifx\svgwidth\undefined%
    \setlength{\unitlength}{841.88976378bp}%
    \ifx\svgscale\undefined%
      \relax%
    \else%
      \setlength{\unitlength}{\unitlength * \real{\svgscale}}%
    \fi%
  \else%
    \setlength{\unitlength}{\svgwidth}%
  \fi%
  \global\let\svgwidth\undefined%
  \global\let\svgscale\undefined%
  \makeatother%
  \begin{picture}(1,0.70707071)%
    \lineheight{1}%
    \setlength\tabcolsep{0pt}%
    \put(0,0){\includegraphics[width=\unitlength,page=1]{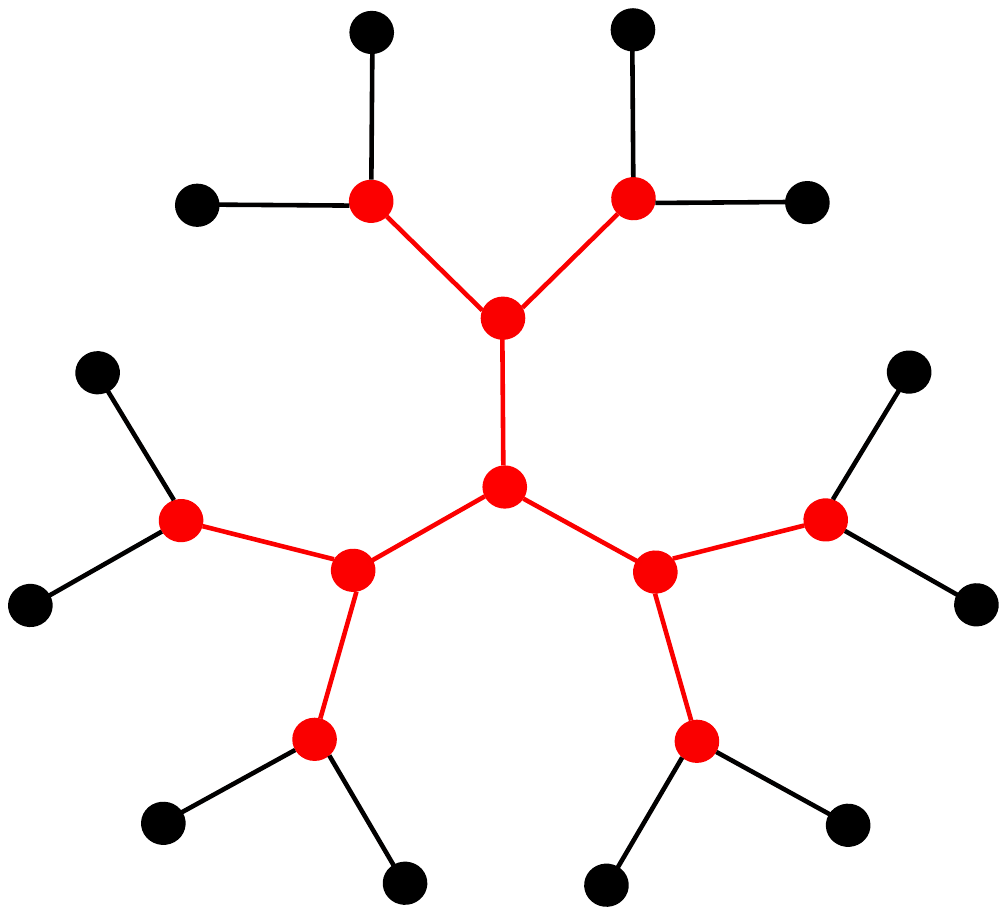}}%
    \put(0.53006259,0.32828801){\color[rgb]{0,0,0}\makebox(0,0)[lt]{\lineheight{1.25}\smash{\begin{tabular}[t]{l}$v^*$\end{tabular}}}}%
    \put(0.58852657,0.42470626){\color[rgb]{0,0,0}\makebox(0,0)[lt]{\lineheight{1.25}\smash{\begin{tabular}[t]{l}$v$\end{tabular}}}}%
    \put(0.04485808,0.00006766){\color[rgb]{0,0,0}\makebox(0,0)[lt]{\lineheight{1.25}\smash{\begin{tabular}[t]{l}\large \color{red} $y_{v^*}(2) \sim Q_1$\end{tabular}}}}%
    \put(0.51623854,0.0000673){\color[rgb]{0,0,0}\makebox(0,0)[lt]{\lineheight{1.25}\smash{\begin{tabular}[t]{l}\color{red}\large $y_v(2) \sim Q_1$\end{tabular}}}}%
  \end{picture}%
\endgroup%

\caption{$t=2$}
\end{subfigure} 
\begin{subfigure}{0.3\textwidth}

	\def \svgwidth{1.0\columnwidth}
\begingroup%
  \makeatletter%
  \providecommand\color[2][]{%
    \errmessage{(Inkscape) Color is used for the text in Inkscape, but the package 'color.sty' is not loaded}%
    \renewcommand\color[2][]{}%
  }%
  \providecommand\transparent[1]{%
    \errmessage{(Inkscape) Transparency is used (non-zero) for the text in Inkscape, but the package 'transparent.sty' is not loaded}%
    \renewcommand\transparent[1]{}%
  }%
  \providecommand\rotatebox[2]{#2}%
  \newcommand*\fsize{\dimexpr\f@size pt\relax}%
  \newcommand*\lineheight[1]{\fontsize{\fsize}{#1\fsize}\selectfont}%
  \ifx\svgwidth\undefined%
    \setlength{\unitlength}{841.88976378bp}%
    \ifx\svgscale\undefined%
      \relax%
    \else%
      \setlength{\unitlength}{\unitlength * \real{\svgscale}}%
    \fi%
  \else%
    \setlength{\unitlength}{\svgwidth}%
  \fi%
  \global\let\svgwidth\undefined%
  \global\let\svgscale\undefined%
  \makeatother%
  \begin{picture}(1,0.70707071)%
    \lineheight{1}%
    \setlength\tabcolsep{0pt}%
    \put(0,0){\includegraphics[width=\unitlength,page=1]{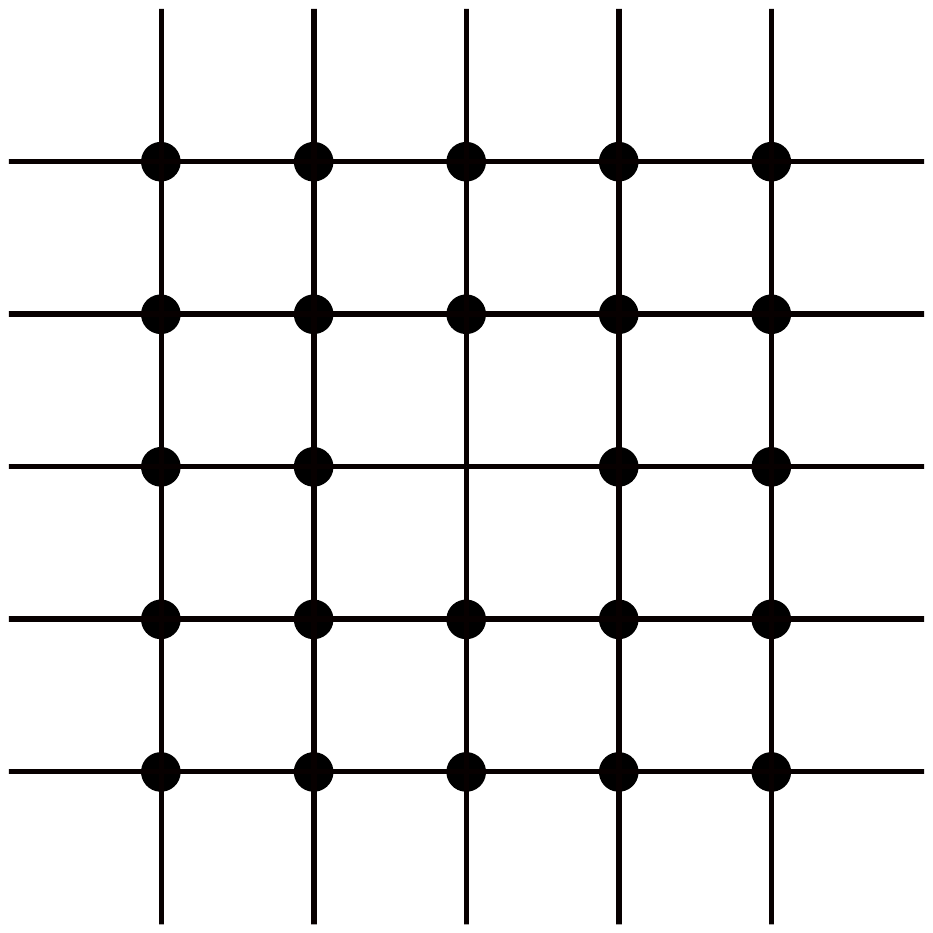}}%
    \put(0.0676858,-0.02472474){\color[rgb]{0,0,0}\makebox(0,0)[lt]{\lineheight{1.25}\smash{\begin{tabular}[t]{l}\large \color{red} $y_{v^*}(0) \sim Q_1$\end{tabular}}}}%
    \put(0.53906625,-0.0247251){\color[rgb]{0,0,0}\makebox(0,0)[lt]{\lineheight{1.25}\smash{\begin{tabular}[t]{l}\large $y_v(0) \sim Q_0$\end{tabular}}}}%
    \put(0.50799547,0.25479839){\color[rgb]{0,0,0}\makebox(0,0)[lt]{\lineheight{1.25}\smash{\begin{tabular}[t]{l}$v^*$\end{tabular}}}}%
    \put(0.60289286,0.41844012){\color[rgb]{0,0,0}\makebox(0,0)[lt]{\lineheight{1.25}\smash{\begin{tabular}[t]{l}$v$\end{tabular}}}}%
    \put(0,0){\includegraphics[width=\unitlength,page=2]{propagation_observation_lattices_t0.pdf}}%
  \end{picture}%
\endgroup%

\caption{$t=0$}
\end{subfigure}
\begin{subfigure}{0.3\textwidth}

	\def \svgwidth{1.0\columnwidth}
\begingroup%
  \makeatletter%
  \providecommand\color[2][]{%
    \errmessage{(Inkscape) Color is used for the text in Inkscape, but the package 'color.sty' is not loaded}%
    \renewcommand\color[2][]{}%
  }%
  \providecommand\transparent[1]{%
    \errmessage{(Inkscape) Transparency is used (non-zero) for the text in Inkscape, but the package 'transparent.sty' is not loaded}%
    \renewcommand\transparent[1]{}%
  }%
  \providecommand\rotatebox[2]{#2}%
  \newcommand*\fsize{\dimexpr\f@size pt\relax}%
  \newcommand*\lineheight[1]{\fontsize{\fsize}{#1\fsize}\selectfont}%
  \ifx\svgwidth\undefined%
    \setlength{\unitlength}{841.88976378bp}%
    \ifx\svgscale\undefined%
      \relax%
    \else%
      \setlength{\unitlength}{\unitlength * \real{\svgscale}}%
    \fi%
  \else%
    \setlength{\unitlength}{\svgwidth}%
  \fi%
  \global\let\svgwidth\undefined%
  \global\let\svgscale\undefined%
  \makeatother%
  \begin{picture}(1,0.70707071)%
    \lineheight{1}%
    \setlength\tabcolsep{0pt}%
    \put(0,0){\includegraphics[width=\unitlength,page=1]{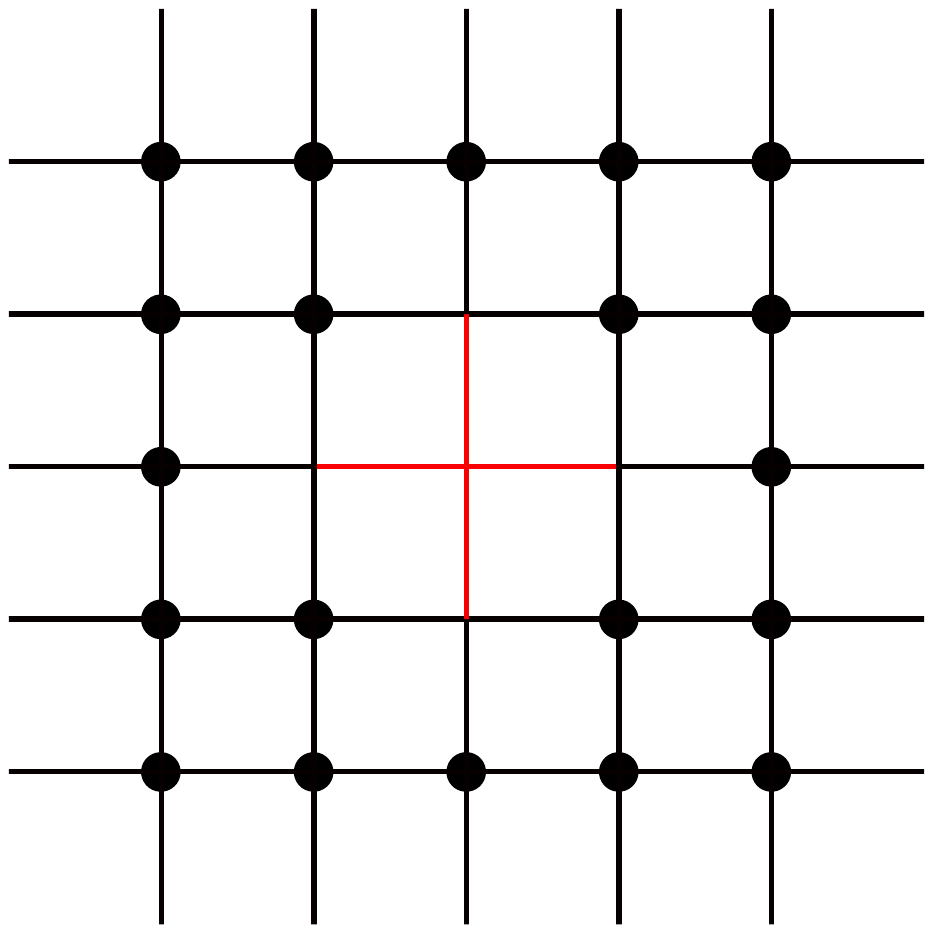}}%
    \put(0.0676858,-0.02472474){\color[rgb]{0,0,0}\makebox(0,0)[lt]{\lineheight{1.25}\smash{\begin{tabular}[t]{l}\large \color{red} $y_{v^*}(1) \sim Q_1$\end{tabular}}}}%
    \put(0.53906625,-0.0247251){\color[rgb]{0,0,0}\makebox(0,0)[lt]{\lineheight{1.25}\smash{\begin{tabular}[t]{l}\large $y_v(1) \sim Q_0$\end{tabular}}}}%
    \put(0.50799547,0.25479839){\color[rgb]{0,0,0}\makebox(0,0)[lt]{\lineheight{1.25}\smash{\begin{tabular}[t]{l}$v^*$\end{tabular}}}}%
    \put(0.60289286,0.41844012){\color[rgb]{0,0,0}\makebox(0,0)[lt]{\lineheight{1.25}\smash{\begin{tabular}[t]{l}$v$\end{tabular}}}}%
    \put(0,0){\includegraphics[width=\unitlength,page=2]{propagation_observation_lattices_t1.pdf}}%
  \end{picture}%
\endgroup%

\caption{$t=1$}
\end{subfigure}
\begin{subfigure}{0.3\textwidth}

	\def \svgwidth{1.0\columnwidth}
\begingroup%
  \makeatletter%
  \providecommand\color[2][]{%
    \errmessage{(Inkscape) Color is used for the text in Inkscape, but the package 'color.sty' is not loaded}%
    \renewcommand\color[2][]{}%
  }%
  \providecommand\transparent[1]{%
    \errmessage{(Inkscape) Transparency is used (non-zero) for the text in Inkscape, but the package 'transparent.sty' is not loaded}%
    \renewcommand\transparent[1]{}%
  }%
  \providecommand\rotatebox[2]{#2}%
  \newcommand*\fsize{\dimexpr\f@size pt\relax}%
  \newcommand*\lineheight[1]{\fontsize{\fsize}{#1\fsize}\selectfont}%
  \ifx\svgwidth\undefined%
    \setlength{\unitlength}{841.88976378bp}%
    \ifx\svgscale\undefined%
      \relax%
    \else%
      \setlength{\unitlength}{\unitlength * \real{\svgscale}}%
    \fi%
  \else%
    \setlength{\unitlength}{\svgwidth}%
  \fi%
  \global\let\svgwidth\undefined%
  \global\let\svgscale\undefined%
  \makeatother%
  \begin{picture}(1,0.70707071)%
    \lineheight{1}%
    \setlength\tabcolsep{0pt}%
    \put(0,0){\includegraphics[width=\unitlength,page=1]{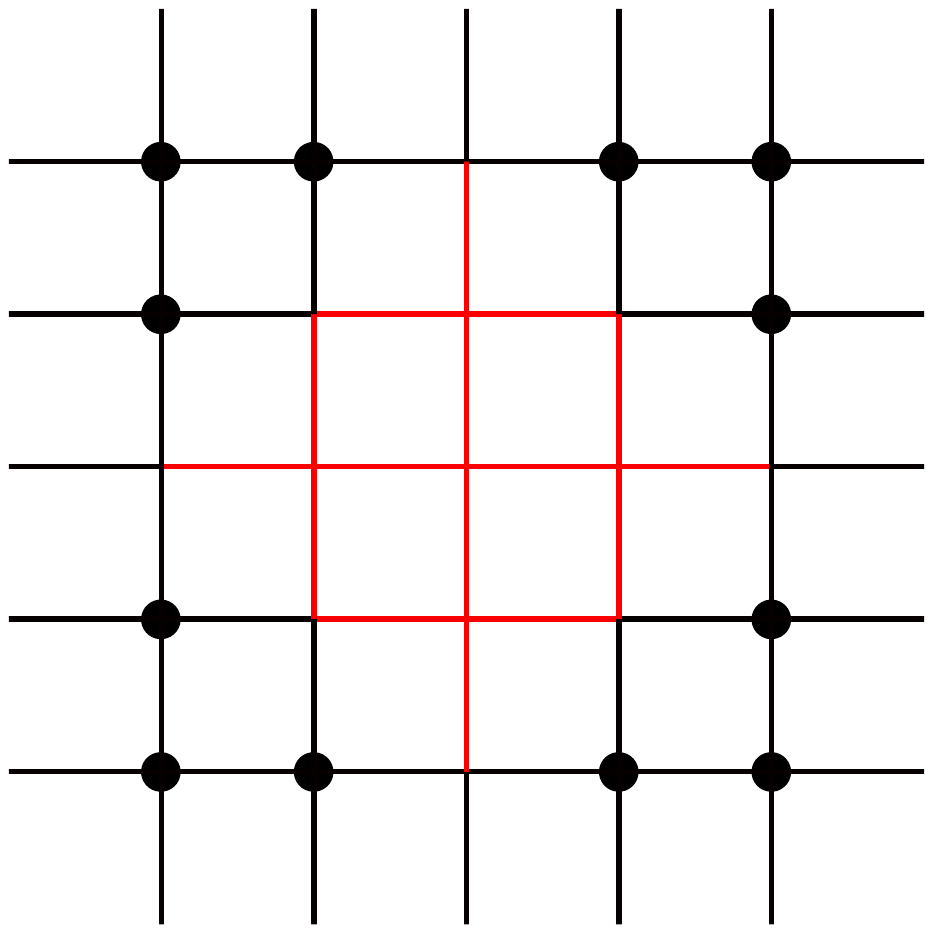}}%
    \put(0.0676858,-0.02472474){\color[rgb]{0,0,0}\makebox(0,0)[lt]{\lineheight{1.25}\smash{\begin{tabular}[t]{l}\large \color{red} $y_{v^*}(2) \sim Q_1$\end{tabular}}}}%
    \put(0.53906625,-0.0247251){\color[rgb]{0,0,0}\makebox(0,0)[lt]{\lineheight{1.25}\smash{\begin{tabular}[t]{l}\color{red}\large $y_v(2) \sim Q_1$\end{tabular}}}}%
    \put(0.50799547,0.25479839){\color[rgb]{0,0,0}\makebox(0,0)[lt]{\lineheight{1.25}\smash{\begin{tabular}[t]{l}$v^*$\end{tabular}}}}%
    \put(0.60289286,0.41844012){\color[rgb]{0,0,0}\makebox(0,0)[lt]{\lineheight{1.25}\smash{\begin{tabular}[t]{l}$v$\end{tabular}}}}%
    \put(0,0){\includegraphics[width=\unitlength,page=2]{propagation_observation_lattices_t2.pdf}}%
  \end{picture}%
\endgroup%

\caption{$t=2$}
\end{subfigure}
\caption{Illustration of cascade propagation and observations in 3-regular trees (a-c) and 2-dimensional lattices (d-f). In each image, the red nodes are those affected by the cascade, and the black are unaffected, though this information is hidden from the observer. Notice that the public signals for the cascade source $v^*$ are always sampled from $Q_1$, whereas the public signals for another vertex $v$ are initially sampled from $Q_0$ and change to samples from $Q_1$ once the cascade spreads sufficiently far.}
\label{fig:propagation_observation}
\end{figure}

See Figure \ref{fig:propagation_observation} for an illustration of the cascade dynamics and the evolution of observed signals on regular trees and lattices. We choose to study regular trees and lattices for several reasons. For one, they have strong symmetry properties (e.g., the local structure around all vertices are the same) which makes it easier to explicitly determine the performance of optimal algorithms. Second, we present a unified treatment of source estimation on regular trees and lattices (except for minor differences), even though the two families of graphs have extremely different topologies; perhaps the most obvious difference is that trees are acyclic while lattices contain many cycles of varying lengths. This indicates that it may be possible to generalize our methods to other topologies as well (see Section \ref{subsec:opt_general_graphs} for further discussion on this point). We also remark that it is a common assumption in the theoretical analysis of cascade models and inference tasks that the underlying graph has infinitely many vertices \cite{newman, sridhar_poor_bayes, sridhar_poor_sequential, huang2020contact, bhamidi2021survival, shah2011rumors, shah2012finding, KhimLoh15, fanti2015spy, racz_source_detection}. Moreover, the infinite graph setting allows us to capture scenarios where the size of the cascade is small compared to the total population without unnecessarily complicating our mathematical analysis.

There are several network and cascade models that are arguably more realistic than the ones we study in this paper; see for instance \cite{cascades_review, configuration_model, random_geometric_graphs, BA99, Mah92,BA99,BRST01}. However, even for the simple networks and cascade dynamics we consider, we expect that an exact characterization of optimal source estimators is mathematically intractable. The reason for this is that we may interpret the problem of source estimation as a {\it sequential multi-hypothesis testing problem}, where different hypotheses correspond to different potential sources. In the two-hypothesis case, the optimal hypothesis test is known to be the {\it sequential probability ratio test} (SPRT), which is a relatively simple procedure that tracks the cumulative log-likelihood ratio over time and stops when it achieves a particular threshold \cite{wald1948}. When there are more than two hypotheses, the optimal test has a complicated form and is difficult to analyze \cite{baum_veeravalli_optimal}. To carry out a tractable analysis, we therefore characterize optimal source estimators in {\it asymptotic regimes}, where the number of possible source vertices -- in other words, the size of the candidate set -- tends to infinity. Formally, we consider a {sequence} of candidate sets and study asymptotic properties of optimal estimators when the size of the candidate set grows large. \revision{As we shall see, the analysis of optimal estimators depends not only on the \emph{size} of the candidate set, but also its topology. For instance, if two vertices in the candidate set are adjacent, there is a lot of overlap in the set of potential infections caused by each vertex. Hence it takes more effort and information to decide between the two vertices. On the other extreme, if two vertices in the candidate set are very far apart, it takes comparatively less effort to distinguish between them. One can therefore imagine that a \emph{worst-case} candidate set is one where all vertices are as close to each other as possible. This idea is formally captured in the following assumption.}

\begin{assumption}
\label{as:candidate_set}
We assume the sequence of candidate sets $\{V_n \}_n$ satisfies the following: 
\begin{enumerate}
\item For all positive integers $n$, $V_n \subset V$ and $|V_n | = n$;
\item There is a designated vertex $v_0 \in V$ and a sequence of integers $\{r_n \}_n$ such that for all positive integers $n$, 
$$
\cN_{v_0}(r_n) \subseteq V_n \subset \cN_{v_0}(r_n + 1).
$$
\end{enumerate}
\end{assumption}

\revision{Above, we recall that $\cN_{v_0}(r)$ is the $r$-hop neighborhood of $v_0$. The second condition in Assumption \ref{as:candidate_set}, which states that $V_n$ is approximately a neighborhood of some arbitrary vertex, correctly captures the notion of a worst-case candidate set in the sense that it maximizes the number of vertex pairs that are close to each other. On a more technical note, by assuming a \emph{specific} topological structure for the candidate set, we have enough detail to carry out a precise mathematical analysis of optimal estimators.}

\revision{The value of $r_n$ used in Assumption \ref{as:candidate_set} can be made explicit.} For the graphs of interest to us, we can employ straightforward combinatorial arguments to show
\begin{equation}
\label{eq:rn_asymptotics}
r_n \sim \begin{cases}
\frac{\log n}{\log (k-1)} & \text{ $G$ is a $k$-regular tree, $k \ge 3$;} \\
\left( \frac{\ell!}{2^\ell} n \right)^{1/\ell} & \text{ $G$ is a $\ell$-dimensional lattice, $\ell \ge 1$.}
\end{cases}
\end{equation}
For details, see Corollaries \ref{cor:trees_rn} and \ref{cor:lattice_rn} in Appendix \ref{sec:size_of_neighborhoods}. 

\revision{As a final remark on the candidate set, we emphasize that Assumption \ref{as:candidate_set} is made only for the purposes of studying the performance of optimal estimators; it need not be satisfied to apply our estimators to more realistic, finite networks. See the part of Section \ref{sec:simulations} concerning cascade source estimation on Erd\H{o}s-R\'{e}nyi random graphs for more details on this point. }

\subsection{Results on Bayesian estimation}
\label{subsec:bayes_results}

For a stopping time $T$, sequence of estimators $\widehat{\mathbf{v}}$ and a candidate set $V_n$, define the quantities
\begin{align*}
\val_B(T, \widehat{\mathbf{v}}) & :  =\E_{\pi(V_n)} [ d(v^*, \widehat{v}(T)) + T ] \\
 \val_B^*(V_n) & : = \inf\limits_{T, \widehat{\mathbf{v}}} \val_B(T, \widehat{\mathbf{v}} ).
\end{align*}
Note in particular that $\val_B^*(V_n)$ is the optimal value of \eqref{eq:bayes_opt} when the candidate set is $V_n$. Our main result on the Bayesian formulation is the following theorem.

\begin{theorem}
\label{thm:bayes}
When $G$ is a $k$-regular tree with $k \ge 3$,
\begin{equation}
\label{eq:T_bayes_tree}
\val_B^*(V_n) \sim \frac{\log \log n}{\log (k - 1)}. 
\end{equation}
On the other hand, when $G$ is a $\ell$-dimensional lattice, there exist constants $a,b$ depending only on $\ell, Q_0, Q_1$ such that for $n$ sufficiently large, 
\begin{equation}
\label{eq:T_bayes_lattice}
a (\log n)^{\frac{1}{\ell + 1}} \le \val_B^*(V_n) \le b (\log n)^{\frac{1}{\ell + 1}}.
\end{equation}
\end{theorem}

In words, \eqref{eq:T_bayes_tree} pins down the exact first-order asymptotic behavior of $\val_B^*(V_n)$ when $n$ is large in the case of regular trees. For lattices, \eqref{eq:T_bayes_lattice} captures the {\it orderwise} behavior of $\val_B^*(V_n)$ when $n$ is large. While Theorem \ref{thm:bayes} focuses on how $\val_B^*(V_n)$ scales with $n$, we remark that the (appropriately defined) distance between $Q_0$ and $Q_1$ plays a role in the performance of optimal estimators. In the case of $k$-regular trees, it appears in the {\it second-order} expansion of $\val_B^*(V_n)$. For $\ell$-dimensional lattices, the distance between $Q_0$ and $Q_1$ influences the constants $a$ and $b$, with both blowing up to infinity as the distance between $Q_0$ and $Q_1$ becomes small. \\

\noindent {\it Proof summary.} For any vertex $v \in V$ and integer $s \ge 0$, recall that $\cN_v(s)$ is the set of vertices within distance $s$ of $v$. From the cascade dynamics defined in Assumption \ref{as:candidate_set}, $\cN_{v^*}(s)$ is precisely the set of vertices which have public signals distributed according to $Q_1$, rather than $Q_0$. The number of {\it total} public signals distributed according to $Q_1$ in the first $t$ timesteps is therefore given by 
\begin{equation}
\label{eq:neighborhood_growth_fn}
f(t) : = \sum\limits_{s = 0}^t | \cN_{v^*}(s) |.
\end{equation}
Due to the symmetry of regular trees and lattices, $|\cN_u(s)| = |\cN_v(s)|$ for any $u,v \in V$ and $s \ge 0$. Hence $f(t)$ does not depend on $v^*$, which is why we do not include $v^*$ in the notation. The interpretation of $f(t)$ as the number of public signals distributed according to $Q_1$ implies that, in an abstract sense, $f(t)$ is a measure of the amount of information an observer has about the spread of the cascade.\footnote{We later make this more formal by showing that the Kullback-Liebler (KL) divergence between the measures $\p_u$ and $\p_v$ pertaining to the variables $y(0), \ldots, y(t)$ is proportional to $f(t)$ for most pairs $u,v \in V$.} On the other hand, the initial uncertainty around the location of $v^*$ can be measured by the entropy of $\pi(V_n)$, which is $\log n$. One may then expect that when $f(t) \lesssim \log n$, the information about the cascade propagation is not enough to overcome the uncertainty around the source location. It turns out that this intuition does indeed hold: for any $(T, \widehat{\mathbf{v}})$, $\E_{\pi(V_n)} [ d(v^*, \widehat{v}(T)) ]$ is large (order $\log n$ in regular trees and $n^{1/\ell}$ in $\ell$-dimensional lattices) when $f(T) \lesssim \log n$. It follows that accurate source estimation is only possible when $f(T) \gtrsim \log n$ or equivalently, when $T \gtrsim f^{-1}(\log n)$. This leads to the lower bound $\val_B^*(V_n) \gtrsim f^{-1}(\log n)$, which is $\log \log n / \log (k - 1)$ in $k$-regular trees up to first order terms and is of order $(\log n)^{\frac{1}{\ell + 1}}$ in $\ell$-dimensional lattices. 

An upper bound on $\val_B^*(V_n)$ is then derived by characterizing the performance of a given source estimator. Consider $(T_{th}, \widehat{\mathbf{v}}_B)$, given formally by
\begin{align}
\label{eq:v_B}
\widehat{v}_B(t)  & \in  \argmin\limits_{v \in V_n} \E_{\pi(t)} [ d(v^*, v) ] \\
\label{eq:T_th}
T_{th}  & : = \min \left \{t \ge 0 : \E_{\pi(t)} [ d(v^*, \widehat{v}_B(t)) ] \le 1 \right \}.
\end{align}
Above, the measure $\pi(t)$ is the {\it posterior distribution} of the source $v^*$ after observing the sequence of public signals $\mathbf{y}(0), \ldots, \mathbf{y}(t)$, \revision{hence $\E_{\pi(t)} [ d(v^*, v) ]$ can be viewed as a conditional expectation.} The interpretations of $T_{th}$ and $\widehat{\mathbf{v}}_B$ are quite intuitive. In words, $\widehat{v}_B(t)$ is a vertex which achieves the minimum estimation error, conditioned on the observed information until time $t$. The estimator $\widehat{\mathbf{v}}_B$ can therefore be thought of as the {\it Bayes-optimal source estimator}, as it minimizes the conditional estimation error.\footnote{We formalize this idea in Lemma \ref{lemma:optimal_estimator}, where we show that if the stopping time $T$ is fixed, $\widehat{\mathbf{v}}_B$ achieves $\inf_{\widehat{\mathbf{v}}} \val_B(T, \widehat{\mathbf{v}})$.} The stopping time $T_{th}$ will keep sampling until the conditional estimation error of the optimal estimator falls below the threshold 1 (the subscript $th$ references the fact that we stop once the estimation error is below a {\it threshold}). In characterizing the performance of the estimator $(T_{th}, \widehat{\mathbf{v}}_B)$, we show that $\val_B(T_{th}, \widehat{\mathbf{v}}_B) \sim \log \log n / \log (k - 1)$ in $k$-regular trees and $\val_B(T_{th}, \widehat{\mathbf{v}}_B) \asymp (\log n)^{\frac{1}{\ell + 1}}$ in $\ell$-dimensional lattices. Remarkably, these match the the lower bounds previously established for $\val_B^*(V_n)$, leading to \eqref{eq:T_bayes_tree} and \eqref{eq:T_bayes_lattice}. Moreover, our analysis shows that the estimator $(T_{th}, \widehat{\mathbf{v}}_B)$ enjoys near-optimal performance when $n$ is large. 

\revision{As a final remark, we note that our proof methods are quite general and can also handle the case where $T$ is replaced with $h(T)$ in \eqref{eq:bayes_opt} for any $h$ that increases slower than an exponential function. Moreover, the \emph{same} estimator $(\widehat{\mathbf{v}}_B, T_{th})$ also enjoys near-optimal performance in this case. For details, see Remark \ref{remark:general_temporal_cost}. }

\subsection{Results on minimax estimation}
\label{subsec:minimax_results}

We begin by defining some notation. Let $\alpha > 0$ be fixed, and suppose $V_n$ is the candidate set. Define the class of estimators 
$$
\Delta(V_n, \alpha) : = \left \{ (T, \widehat{\mathbf{v}}) : \max\limits_{v \in V_n} \E_v [ d(v, \widehat{v}(T) ) ] \le \alpha \right \}.
$$
In words, $\Delta(V_n, \alpha)$ is the class of source estimators which have a worst-case estimation error of at most $\alpha$. In particular, $\Delta(V_n, \alpha)$ is the set of feasible estimators in the minimax formulation \eqref{eq:np}. The optimal value of the minimax formulation is denoted by 
$$
\val_M^*(V_n, \alpha) := \inf\limits_{(T, \widehat{\mathbf{v}}) \in \Delta(V_n, \alpha)} \max\limits_{v \in V_n} \E_v [ T ].
$$
The results we obtain for $\val_M^*(V_n, \alpha)$ are essentially the same as in the Bayesian formulation. Specifically, we prove the following theorem. 

\begin{theorem}
\label{thm:minimax}
Let $\alpha$ be fixed. When $G$ is a $k$-regular tree, 
\begin{equation}
\label{eq:T_minimax_tree}
\val_M^*(V_n, \alpha) \sim \frac{\log \log n}{\log (k - 1)}.
\end{equation}
On the other hand, when $G$ is a $\ell$-dimensional lattice, there exist constants $a', b'$ depending only on $\ell, Q_0, Q_1$ such that for $n$ sufficiently large, 
\begin{equation}
\label{eq:T_minimax_lattice}
a' (\log n)^{\frac{1}{\ell + 1}} \le \val_M^*(V_n, \alpha) \le b' (\log n)^{\frac{1}{\ell + 1}}.
\end{equation}
\end{theorem}

As in Theorem \ref{thm:bayes}, \eqref{eq:T_minimax_tree} provides an exact first-order characterization of $\val_M^*(V_n, \alpha)$ when $n$ is large, and \eqref{eq:T_minimax_lattice} describes the {\it orderwise} behavior of $\val_M^*(V_n, \alpha)$. We remark that the constants $a'$ and $b'$ used in \eqref{eq:T_minimax_lattice} are potentially distinct from the constants $a,b$ used in the Bayesian analogue \eqref{eq:T_bayes_lattice}. However, we make no attempt to optimize the constants, instead focusing on the orderwise behavior as $n$ grows large. As in Theorem \ref{thm:bayes}, the (appropriately defined) distance between $Q_0$ and $Q_1$ plays a role in the second order terms of $\val_M^*(V_n, \alpha)$ in regular trees. In lattices, the constants $a'$ and $b'$ blow up to infinity when the distance between $Q_0$ and $Q_1$ is small.\\ 

\begin{figure*}[t]
\centering
\begin{subfigure}{0.4 \textwidth}
\centering

	\def \svgwidth{1.0\columnwidth}
	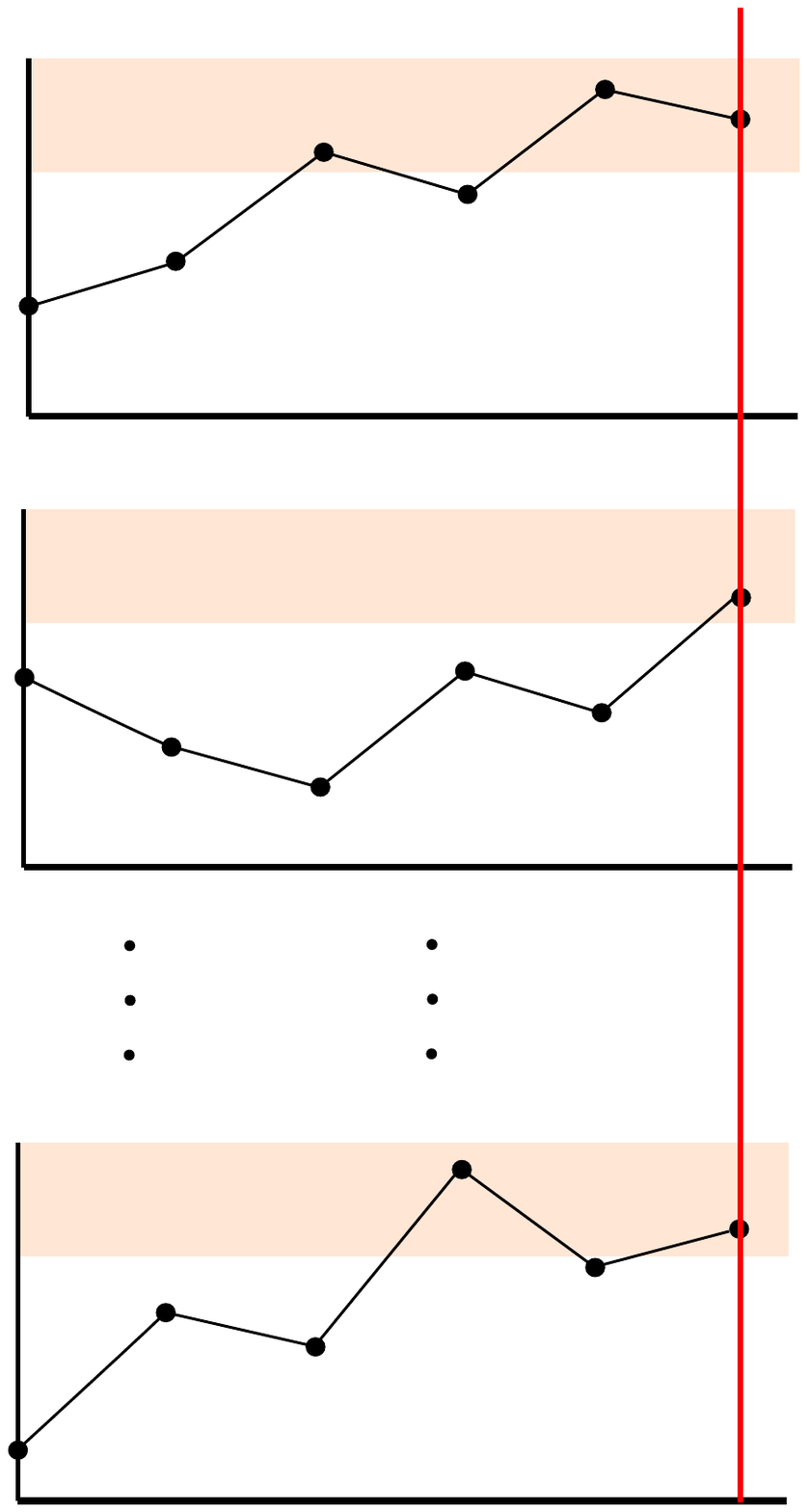

\end{subfigure}%
\begin{subfigure}{0.6 \textwidth}
\centering
\includegraphics[width=0.7\textwidth]{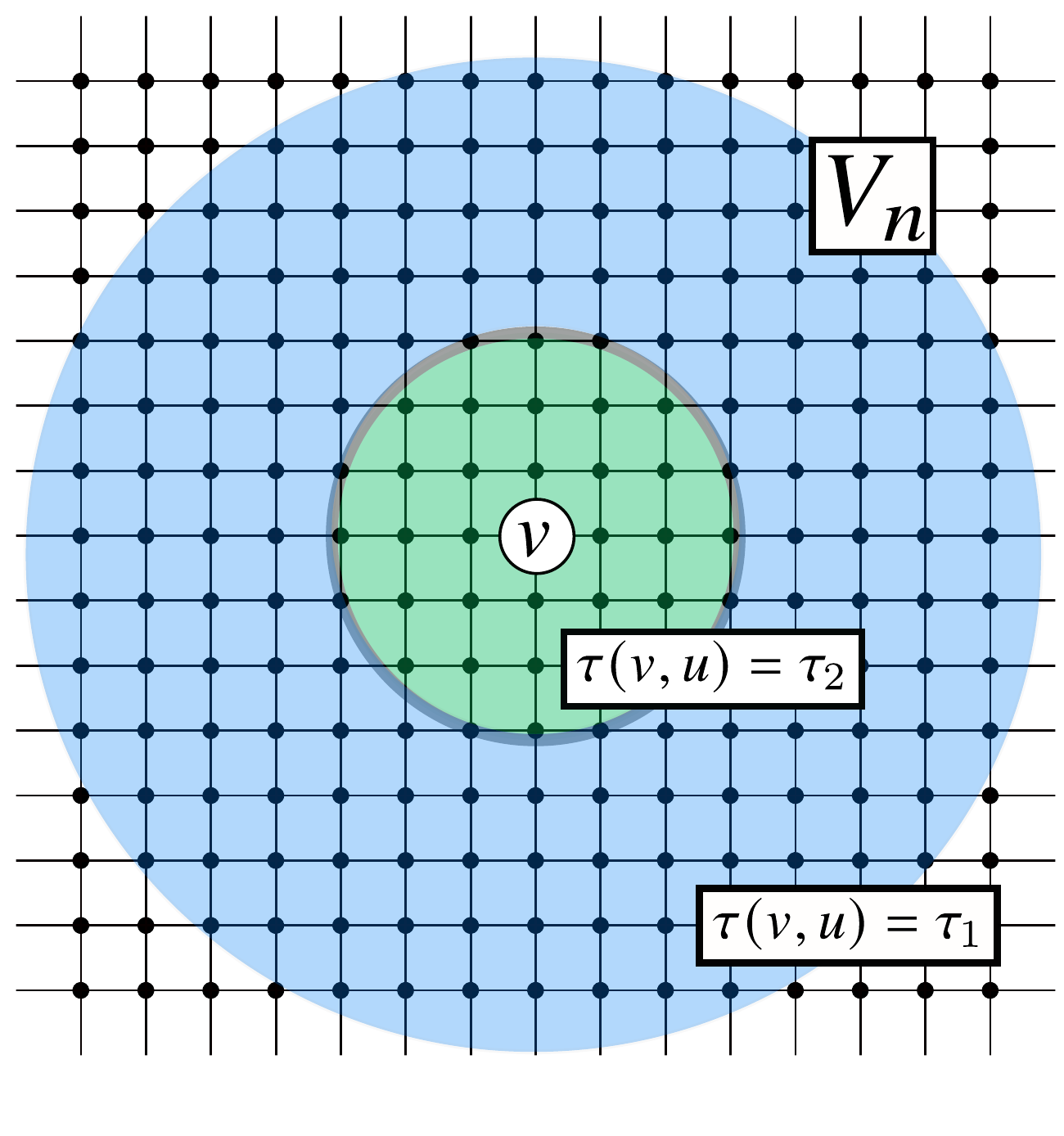}
\end{subfigure}
\caption{Schematic representation of the MSPRT we design to achieve near-optimal source estimation. {\it Left:} Plots of the likelihood ratios with respect to vertex $v$, where we have enumerated $V_n \setminus \{v\}$ as $u_1, \ldots, u_{n-1}$. We halt when $d\p_v / d\p_{u_i}$ crosses the threshold $\tau(v, u_i)$ (shown in orange) for all $1 \le i \le n-1$. This stopping time, labeled $T(v)$, is shown in red. The source estimator is the vertex $v$ which achieves the smallest value of $T(v)$. {\it Right:} Design of the thresholds $\tau(v, u)$, visualized here for the 2-dimensional lattice. For vertices $u,v \in V_n$ that are far (blue region), we set $\tau(v,u) = \tau_1$ and for vertices $u,v \in V_n$ that are close (green region), we set $\tau(v,u) = \tau_2$. 
}
\label{fig:msprt}
\end{figure*}

\noindent {\it Proof summary.} As in the Bayesian case, we focus on establishing lower bounds for $\val_M^*(V_n, \alpha)$ and derive matching upper bounds by studying the performance of a carefully designed estimator which lies within the feasible set $\Delta(V_n, \alpha)$. 

To derive lower bounds for $\val_M^*(V_n, \alpha)$, we observe that the Bayesian objective value -- which measures {\it average-case} behaviors of source estimators -- is less strict than the minimax objective, which measures {\it worst-case} behaviors of source estimators. Hence the lower bound for $\val_B^*(V_n)$ also holds for $\val_M^*(V_n, \alpha)$, provided $\alpha$ is constant with respect to $n$. 

Next, we establish an upper bound for $\val_M^*(V_n, \alpha)$ by characterizing the performance of a {\it specific} stopping time and estimator. Unfortunately, we cannot use $(T_{th}, \widehat{\mathbf{v}}_B)$ (used in the Bayesian setting) for this task since it is unclear whether it is an element of the class $\Delta(V_n, \alpha)$. We therefore take a different approach in designing an appropriate estimator within $\Delta(V_n, \alpha)$ which matches the lower bounds for $\val_M^*(V_n, \alpha)$. 

To this end, recall that the problem of source estimation can be viewed as a sequential multi-hypothesis testing problem, where each hypothesis corresponds to the possibility of a particular vertex being the source. Motivated by the optimality of SPRTs for the two-hypothesis setting, we consider a natural extension to the multiple hypothesis setting known as the {\it multi-hypothesis sequential probability ratio test} (MSPRT), described below: \begin{itemize}

\item For each pair of distinct $u,v \in V_n$, specify a threshold $\tau(v,u)$, which is a positive real number. 

\item Let $T(v)$ be the stopping time that halts as soon as 
$$
\frac{d \p_v }{d \p_u} (y(0), \ldots, y(t)) \ge \tau(v,u) , \qquad \forall u \in V_n \setminus \{v \}.
$$
\revision{Here, we recall that $\p_v$ is the probability measure conditioned on $v^* = v$.}

\item The final source estimator is $\argmin_{v \in V_n} T(v)$; that is, the vertex whose stopping time halts first. 

\end{itemize}

For general multi-hypothesis testing problems, it is known that MSPRTs enjoy near-optimal performance when the number of hypotheses is fixed and the bound on the estimation error, $\alpha$, is small \cite{baum_veeravalli_optimal, Tartakovsky1998, Dragalin_1987, lorden, msprt_opt_pt1, msprt_opt_pt2}. Although our setting is different, since $\alpha$ is fixed and the number of hypotheses are large, it is natural to expect that MSPRTs still have good performance. Indeed, we provide a novel way to design MSPRTs with worst-case expected runtime that match the lower bounds for $\val_M^*(V_n, \alpha)$: for $u,v \in V_n$ that are ``far", we set $\tau(v,u) =  \tau_1$ and for $u,v \in V_n$ that are ``close" we set $\tau(v,u) = \tau_2$ where $\tau_1, \tau_2$ are pre-determined parameters depending on the graph structure and $n$. This design can be interpreted as a {\it multi-scale search strategy}: an analysis of the likelihood ratios $d\p_v / d\p_u$ for $u,v$ far apart determine the general location of the source, and an analysis of the likelihood ratios $d\p_v / d\p_u$ for $u,v$ close give us a more fine-grained estimate. We show that with the right definition of ``closeness" as well as an appropriate choice of $\tau_1, \tau_2$, an MSPRT designed in this way achieves the upper bounds for $\val_M^*(V_n, \alpha)$ described in Theorem \ref{thm:minimax}. A diagram illustrating the key ideas of the MSPRT we have described can be found in Figure \ref{fig:msprt}. 

\subsection{Conjectures on optimal estimators in general graphs}
\label{subsec:opt_general_graphs}

In this work, we primarily focus on regular trees and lattices for a few key reasons. For one, they enjoy strong symmetry properties. In particular, the local structure of all vertex neighborhoods are isomorphic, leading to conceptually simpler proofs and near-exact computations of $\val_B^*(V_n)$ as well as $\val_M^*(V_n, \alpha)$. Second, regular trees and lattices have drastically different topological structure, yet most of our proofs work equally well for both topologies, with just minor differences. This suggests that our methods can also be used to describe optimal source estimators for {\it general topologies}. Below, we discuss how Theorems \ref{thm:bayes} and \ref{thm:minimax} may change for general topologies. 

We start by defining relevant notation. Let $G$ be a graph with (countably) infinite vertices that is locally finite (i.e., all vertex degrees are finite). As discussed earlier, we study infinite graphs since it allows us to consider scenarios where the cascade is small relative to the network size without complicating our analysis. Define the {\it vertex-dependent} neighborhood growth function
$$
f_v(t) : = \sum\limits_{s = 0}^t | \cN_v(s) |.
$$
We then have the following conjecture concerning optimal source estimation in general topologies. 

\begin{conjecture}
\label{conj:opt_val}
Suppose that $G$ is a graph with countably infinite vertices that is locally finite. Let the sequence of candidate sets $\{V_n\}_{n \ge 1}$ satisfy Assumption \ref{as:candidate_set}. Then 
$$
\val_B^*(V_n) \asymp \frac{1}{n} \sum\limits_{v \in V_n} f_v^{-1}(\log n).
$$
Additionally, for any $\alpha > 0$ that is constant with respect to $n$, 
$$
\val_M^*(V_n, \alpha) \asymp \max_{v \in V_n} f_v^{-1}(\log n).
$$
\end{conjecture}

We expect that Conjecture \ref{conj:opt_val} can be proved by a straightforward generalization of our techniques. Following analogous arguments as the proof summary for Theorem \ref{thm:bayes}, if $v^* = v$ then $f_v(t)$ is the total number of public signals distributed according to $Q_1$ until timestep $t$. We therefore expect that the uncertainty in the source location is too large to reasonably estimate the source in the regime $f_v(t) \lesssim \log n$ (equivalently, $t \lesssim f_v^{-1}(\log n)$), since the entropy of the prior $\pi(V_n)$ is $\log n$. Hence any algorithm must observe for at least $f_v^{-1}(\log n)$ timesteps to reliably estimate the source. Averaging over $v \in V_n$ leads to the lower bound $\val_B^*(V_n) \gtrsim \frac{1}{n}\sum_{v \in V_n} f_v^{-1}(\log n)$. On the other hand, the minimax setting captures the {\it worst-case} expected number of samples as opposed to the average-case number of samples, hence $\val_M^*(V_n, \alpha) \gtrsim \max_{v \in V_n } f_v^{-1}(\log n)$. 

Establishing upper bounds for $\val_B^*(V_n)$ and $\val_M^*(V_n, \alpha)$ that are orderwise tight requires an analysis of specific source estimators. Since our analysis of $(T_{th}, \widehat{\mathbf{v}}_B)$ is quite similar for both regular trees and lattices (with only minor differences), we expect that it should achieve optimal performance in general as well. We also believe that a properly designed MSPRT can achieve optimal performance in the minimax setting as well; we provide further details on this point in Section \ref{subsec:minimax_upper_bound} (see Remark \ref{remark:general_msprt}). 

\revision{

\section{Simulations}
\label{sec:simulations}

In this section, we complement our theoretical results through simulations which reveal the \emph{non-asymptotic} performance of our source estimators. Specifically, we study the performance of two estimators: the Bayes estimator described in \eqref{eq:v_B} and \eqref{eq:T_th}, and the MSPRT used to prove the achievability results in Theorem \ref{thm:minimax}. At a high level, our simulations show that even in non-asymptotic regimes, our estimators are able to locate the source while ensuring that only a small number of individuals are infected, thus validating our theoretical results on trees and lattices. We further apply our estimators to cascades spreading on natural models of random graphs (the Erd\H{o}s-R\'{e}nyi model), showing that our estimators can be successfully applied to broader scenarios.

\begin{figure*}[t]
\centering
\begin{subfigure}{0.45 \textwidth}
\centering
\includegraphics[width= \textwidth]{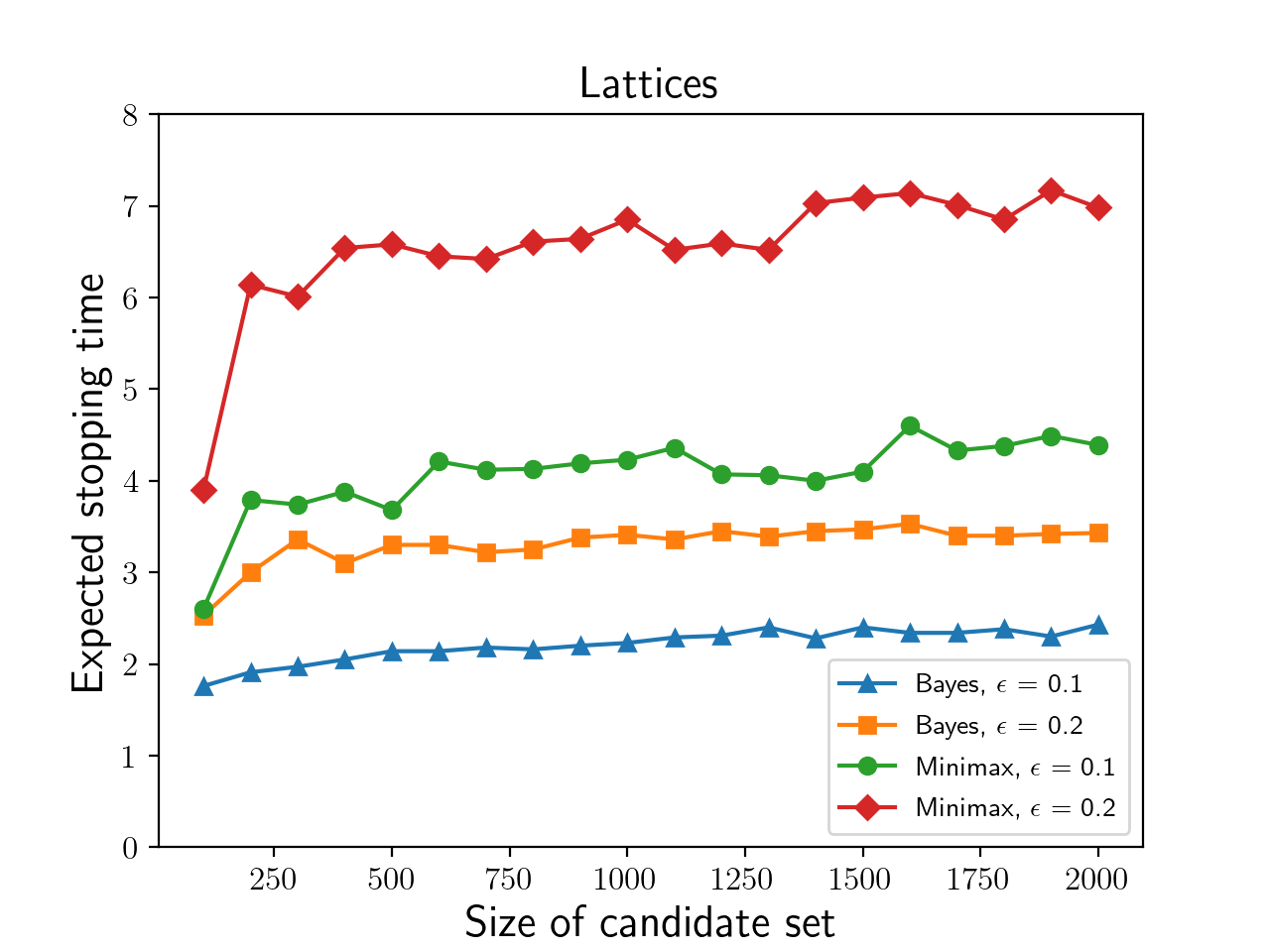}
\caption{}
\end{subfigure}%
\begin{subfigure}{0.45 \textwidth}
\centering
\includegraphics[width=\textwidth]{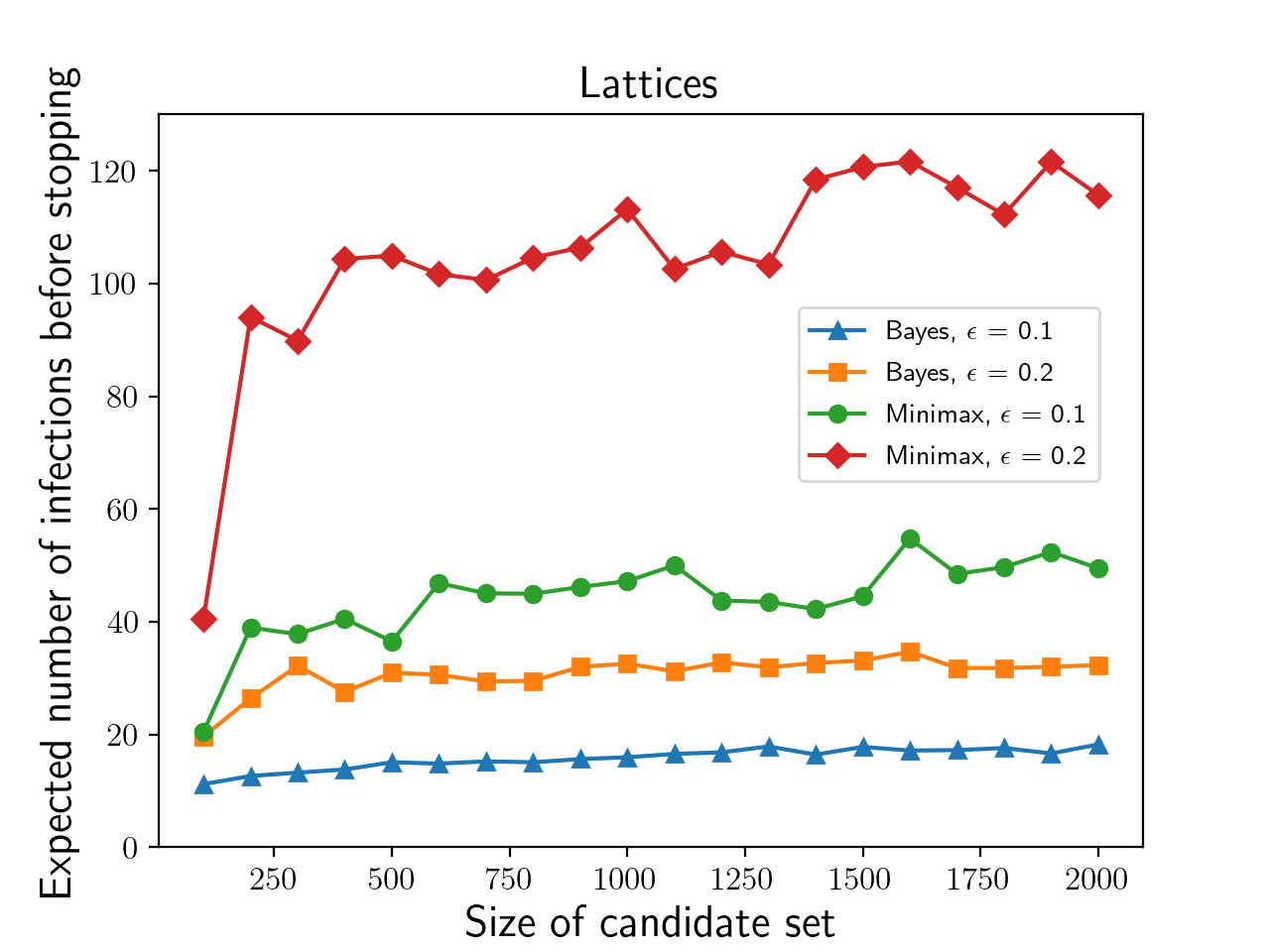}
\caption{}
\end{subfigure}
\caption{Plots of the expected stopping time $(a)$ and the expected number of infections $(b)$ as a function of $n$, the size of the candidate set. Pictured here are the performances of the Bayes and minimax-optimal source estimators in 2-dimensional lattices. 
}
\label{fig:lattice_plots}
\end{figure*}

\paragraph{Signal distributions.} We consider the case of noisy and incomplete testing, described in Remark \ref{remark:data_model}; we briefly recap the model here. Interpret the network cascade as an infection, and assume that at every timestep, each individual tests for infection with probability $p$. The test outputs the correct result (i.e., positive or negative) with probability $1 - \epsilon$. In our simulations, we let $p = 0.5$ and $\epsilon \in \{0.1, 0.2 \}$. The distributions $Q_0, Q_1$ derived from this scenario are formally described in Remark \ref{remark:data_model}. 

\paragraph{Lattices.} To make our simulations as close to our theoretical setup as possible, our base graph $G$ was taken to be a 2-dimensional 100 x 100 lattice (10,000 vertices). If the size of the candidate set is $n$, in accordance with Assumption \ref{as:candidate_set}, we chose the candidate set to be the $n$ closest vertices to the center of the lattice. We emphasize that choosing the candidate set in this way captures the notion of a \emph{worst-case candidate set} (see the discussion surrounding Assumption \ref{as:candidate_set}), and we choose vertices close to the lattice center only to avoid boundary effects (i.e., to ensure that the cascade will evolve similarly from all potential source vertices, given that $G$ is finite). In our simulations, the cascade begins at the lattice center and spreads via the deterministic dynamics described in Assumption \ref{as:cascade}, producing random observed vertex-level signals according to \eqref{eq:public_signals}. Although we could, in principle, choose any source vertex in the candidate set, we consistently choose the lattice center in order to reduce the variance of the estimators' performance across independent simulations. For each choice of $n$ (from 100 to 2000, collected at regular intervals of 100) and $\epsilon \{ 0.1, 0.2 \}$, we carried out 100 independent simulations of the cascade. We averaged over the stopping time and the number of total infections until the stopping time to generate the plots in Figure \ref{fig:lattice_plots}. The design of the MSPRT weights for the minimax estimator follows Figure \ref{fig:msprt}; for the specific threshold values, see Theorem \ref{thm:lattices} in Section \ref{subsec:minimax_upper_bound}. 

Figure \ref{fig:lattice_plots} highlights important \emph{finite-size} behaviors of the Bayes and minimax estimators. Notice that each curve is quite flat: the expected stopping time as well as the number of infections changes little with respect to $n$. This weak dependence extends to the asymptotic regime $n \to \infty$ as well; Theorems \ref{thm:bayes} and \ref{thm:minimax} show that the expected stopping time scales as $(\log n)^{1/3}$. Moreover, the Bayes estimator has strictly better performance than the minimax estimator in all cases -- notably, nearly 100 infections are prevented in the case $\epsilon = 0.2$ for large $n$ when comparing the Bayes and minimax estimators.

\begin{figure*}[t]
\centering
\begin{subfigure}{0.45 \textwidth}
\centering
\includegraphics[width= \textwidth]{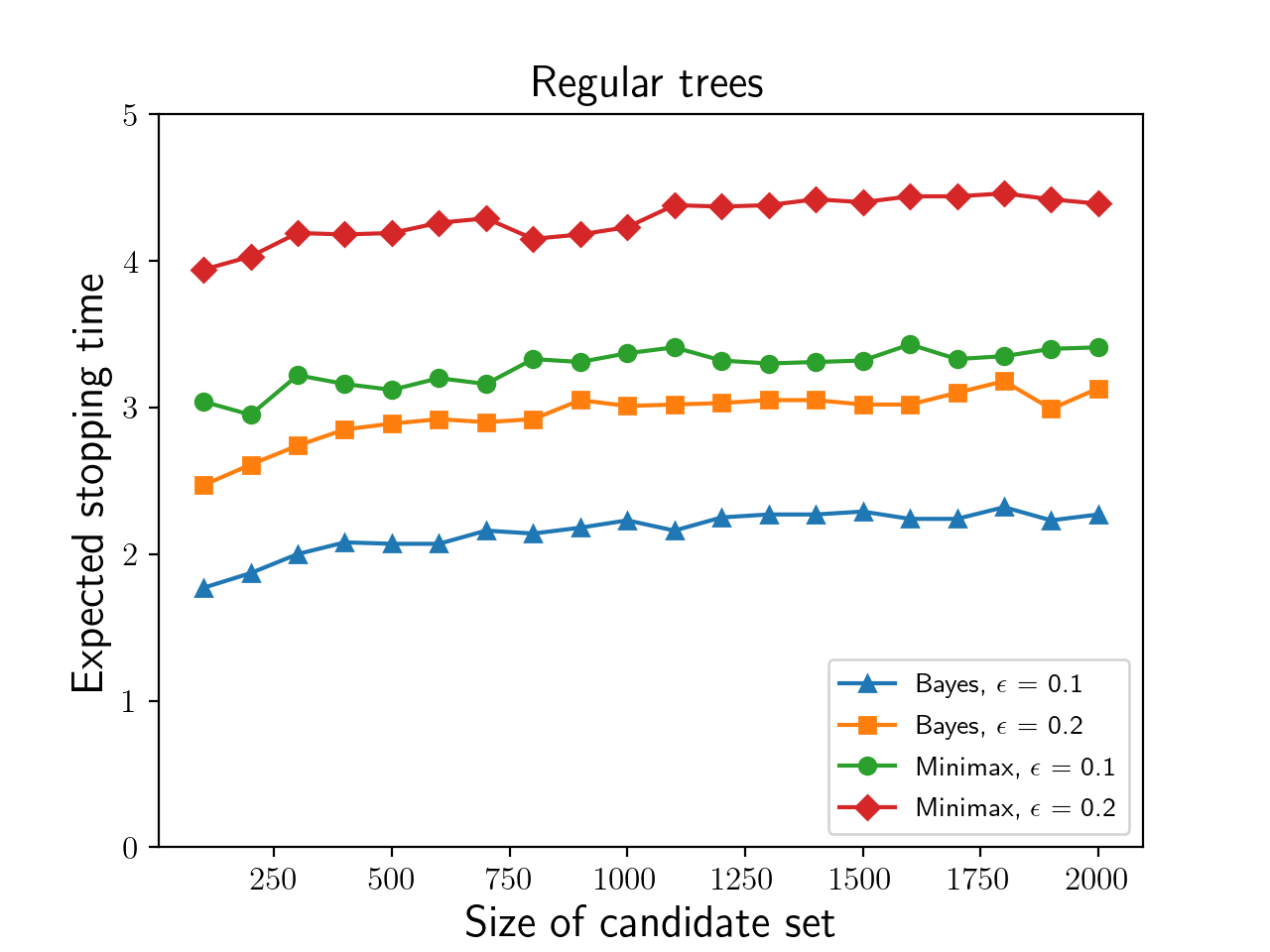}
\caption{}
\end{subfigure}%
\begin{subfigure}{0.45 \textwidth}
\centering
\includegraphics[width=\textwidth]{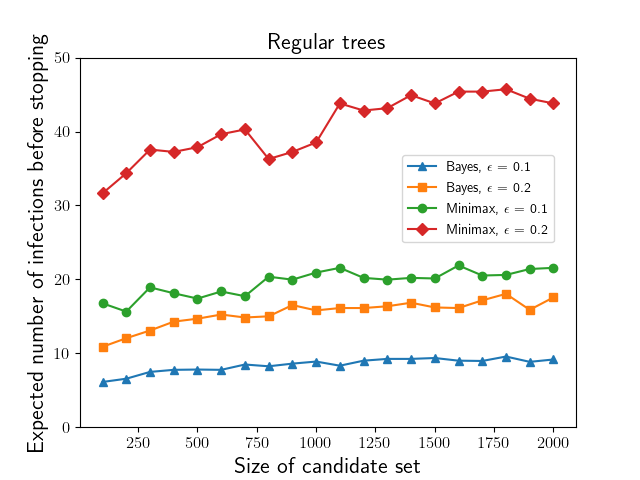}
\caption{}
\end{subfigure}
\caption{Plots of the expected stopping time $(a)$ and the expected number of infections $(b)$ as a function of $n$, the size of the candidate set. Pictured here are the performances of the Bayes and minimax-optimal source estimators in 3-regular trees. 
}
\label{fig:tree_plots}
\end{figure*}

\paragraph{Trees.} Our base graph $G$ was taken to be a 3-regular balanced tree with 16,383 vertices. If the size of the candidate set is $n$, we choose the candidate set to be the $n$ closest vertices to the root of balanced tree $G$ for similar reasons as in the case of lattices. In our simulations, the cascade begins at the root of $G$ and spreads via the deterministic dynamics described in Assumption \ref{as:cascade}, producing random observed vertex-level signals according to \eqref{eq:public_signals}. For each choice of $n$ and $\epsilon$, we carried out 100 independent simulations of the cascade and average over the stopping time as well as the number of total infections to generate the plots in Figure \ref{fig:tree_plots}. For the minimax estimator, we use an MSPRT with \emph{constant} weights, which is proved to be asymptotically optimal; see Theorem \ref{thm:trees} in Section \ref{subsec:minimax_upper_bound} for details. 

Similar conclusions as in the case of lattices can be drawn for trees based on Figure \ref{fig:tree_plots}. Interestingly, although the Bayes and minimax estimators take $\log \log n$ samples in light of Theorems \ref{thm:bayes} and \ref{thm:minimax}, Figure \ref{fig:tree_plots} shows that the Bayes-optimal estimator stops earlier in finite regimes; we believe this is due to the provable optimality of the estimator (see Lemma \ref{lemma:optimal_estimator}). Furthermore, although we have proved that the distance between $Q_0$ and $Q_1$ does not affect the first-order asymptotics of the expected stopping time in both Bayes and minimax settings, it appears to play a significant role in finite regimes. Notably, the time it takes the minimax-optimal estimator to stop is more than doubled when $\epsilon = 0.2$ compared to $\epsilon = 0.1$ for many values of $n$.

\begin{figure*}[t]
\centering
\begin{subfigure}{0.45 \textwidth}
\centering
\includegraphics[width= \textwidth]{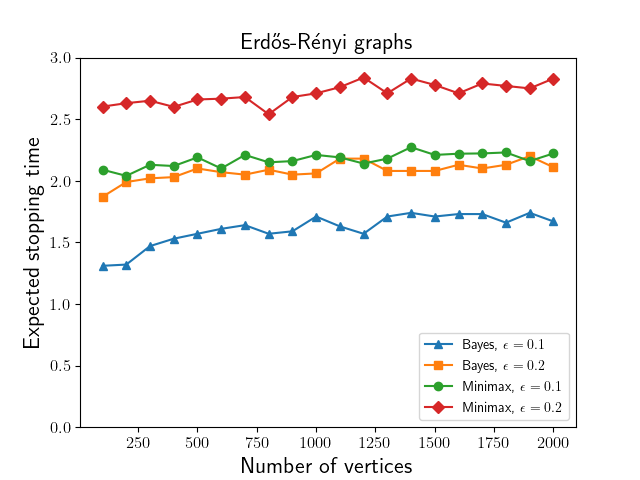}
\caption{}
\end{subfigure}%
\begin{subfigure}{0.45 \textwidth}
\centering
\includegraphics[width=\textwidth]{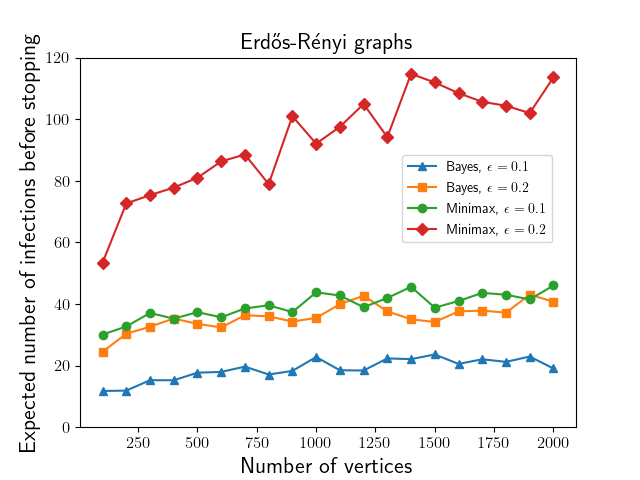}
\caption{}
\end{subfigure}
\caption{Plots of the expected stopping time $(a)$ and the expected number of infections $(b)$ as a function of $n$, the size of the candidate set. Pictured here are the performances of the Bayes and minimax-optimal source estimators in the Erd\H{o}s-R\'{e}nyi graph $G(n, 5/n)$. 
}
\label{fig:er_plots}
\end{figure*}

\paragraph{Erd\H{o}s-R\'{e}nyi random graphs.} Recall that for a positive integer $n$ and $q \in [0,1]$, an Erd\H{o}s-R\'{e}nyi random graph $G(n,q)$ is generated as follows. Let the vertex set $V$ be a set of $n$ labeled vertices, and for each pair of distinct vertices an edge is added between them with probability $q$, independently across all vertex pairs. Since our work is primarily concerned with \emph{sparse} graphs (i.e., vertex degrees are not too large), we chose $q = 5 / n$ to ensure that the average degree of the graph is 5. This choice of $q$ ensures that the largest connected component of $G(n,q)$ (also known as the \emph{giant component}) is most of the graph, while also keeping the average degree relatively small. For a given realization of $G(n,q)$, our candidate set was taken to be the vertices in the giant component.\footnote{It is known that with high probability, components other than the giant component are of order $\log n$ \cite[Chapter 11.9]{alon_spencer}. As a result, the infection will never reach most of the graph even after a long time passes. Our assumption that the candidate set is the vertex set of the giant component avoids such simple edge cases.} Figure \ref{fig:er_plots} shows the performance of the Bayes and minimax estimators on Erd\H{o}s-R\'{e}nyi graphs, both of which exhibit similar performance to that noted for trees and lattices. For the minimax estimator, since sparse Erd\H{o}s-R\'{e}nyi graphs are known to be locally tree-like \cite{alon_spencer}, we use the MSPRT with uniform weights (see Section \ref{subsec:minimax_upper_bound}) which is optimal in regular trees. To summarize, Figure \ref{fig:er_plots} shows that the estimators we develop are \emph{robust} and apply to a broader class of graphs than the ones we analyze theoretically. 

To generate the numerical results in Figure \ref{fig:er_plots}, as before we ran 100 independent simulations for each $n$ and $\epsilon$ considered. For each simulation, an independent Erd\H{o}s-R\'{e}nyi graph was generated. The data in Figure \ref{fig:er_plots} were computed by averaging the stopping times and number of infected vertices, \emph{conditioned} on the event that the cascade did not spread to all vertices by the time the algorithm stopped. The reason for this is that if the cascade affects all vertices, there is effectively no new information to be learned, and the stopping time would be extremely large with high probability. In almost all cases, however, at most one out of the 100 trials would fall into this category. The only exception was the case of $n = 100$ and $\epsilon = 0.2$ for the minimax estimator, which had 12 trials fall into this category. We expect that this is because the cascade spreads too quickly to detect it on a graph of this small size. Finally, we remark that some of the curves in Figure \ref{fig:er_plots} may appear noisier than the ones in Figures \ref{fig:tree_plots} and \ref{fig:lattice_plots}; this is likely due to the randomness of the base graph in the Erd\H{o}s-R\'{e}nyi case, compared to the deterministic nature of the other topologies considered.

}

\section{Analysis of the Bayesian setting}
\label{sec:bayes}

\subsection{Behavior of the Bayes-optimal estimator}

We begin with a discussion of the estimator $\widehat{\mathbf{v}}_B = \{ \widehat{v}_B(t) \}_{t \ge 0}$, which is defined formally by 
$$
\widehat{v}_B(t) \in \argmin\limits_{v \in V_n} \E_{\pi(t)} [ d(v^*, v) ],
$$
where $V_n$ is the candidate set under consideration. It is straightforward to show that $\widehat{\mathbf{v}}_B$ is {\it optimal}, in the sense that it minimizes the error of the final source estimator for any choice of stopping time. This is explained more formally in the following result. 

\begin{lemma}
\label{lemma:optimal_estimator}
Let the candidate set be $V_n$. Fix any stopping time $T$ and let $\widehat{\mathbf{u}} = \{ u(t) \}_{t \ge 0}$ be any source estimator so that $u(t)$ is measurable with respect to $y(0), \ldots, y(t)$. Then 
$$
\E_{\pi(V_n)} [ d(v^*, \widehat{v}_B(T)) ] \le \E_{\pi(V_n)} [ d(v^* ,\widehat{u}(T) ) ].
$$
In particular, 
$$
\inf\limits_{(T, \widehat{\mathbf{v}})} \val_B(T, \widehat{\mathbf{v}}) = \inf\limits_{T} \val_B(T, \widehat{\mathbf{v}}_B).
$$
\end{lemma}

\begin{proof}
For any given time index $t \ge 0$, it follows from the definition of $\widehat{v}_B(t)$ that 
\begin{equation}
\label{eq:estimator_inequality}
\E_{\pi(t)} [ d(v^*, \widehat{v}_B(t)) ] \le \E_{\pi(t)} [ d(v^*, \widehat{u}(t)) ].
\end{equation}
We can then write
\begin{align*}
\E_{\pi(V_n)} [ d(v^*, \widehat{v}_B(T)) ] & \stackrel{(a)}{=} \E_{\pi(V_n)} \left[ \E_{\pi(T)} [ d(v^*, \widehat{v}_B(T)) ] \right] \\
& = \E_{\pi(V_n) } \left[ \sum\limits_{t = 0}^\infty \E_{\pi(t)} [ d(v^*, \widehat{v}_B(t)) ] \mathbf{1}(T = t) \right] \\
& \stackrel{(b)}{\le} \E_{\pi(V_n) } \left[ \sum\limits_{t = 0}^\infty \E_{\pi(t)} [ d(v^*, \widehat{u}(t) ) ] \mathbf{1}(T = t) \right] \\
& = \E_{\pi(V_n)} [ d(v^*, \widehat{u}(T)) ].
\end{align*}
Above, $(a)$ is due to the tower rule and $(b)$ is a consequence of \eqref{eq:estimator_inequality}. Taking an infimum over all source estimators $(T, \widehat{\mathbf{u}})$ on both sides yields the second statement of the lemma.
\end{proof}

With the optimal estimator explicitly derived, we focus on characterizing the estimation error of $\widehat{\mathbf{v}}_B$, which will in turn aid us in characterizing optimal stopping times. Indeed, the proof of Theorem \ref{thm:bayes} depends on the following two lemmas which characterize the estimation error of $\widehat{\mathbf{v}}_B$. Before stating them, we review some basic properties of the neighborhood growth function $f(t)$ (defined in \eqref{eq:neighborhood_growth_fn}). First, it can be shown through straightforward combinatorial arguments that 

\begin{equation}
\label{eq:f_asymptotics}
f(t) \sim \begin{cases}
\frac{k(k-1)}{(k-2)^2} (k-1)^t & \hspace{-0.1cm}\text{$G$ is a $k$-regular tree, $k \ge 3$;} \\
 \frac{2^\ell}{(\ell + 1)!} t^{\ell + 1} & \hspace{-0.1cm}\text{$G$ is a $\ell$-dimensional lattice, $\ell \ge 1$.}
\end{cases}
\end{equation}
A proof of \eqref{eq:f_asymptotics} can be found in Appendix \ref{sec:size_of_neighborhoods} -- see in particular Lemmas \ref{lemma:trees_exact} and \ref{lemma:lattices_exact}. Importantly, \eqref{eq:f_asymptotics} can be used to study the asymptotics of the {\it inverse function} $F = f^{-1}$. Indeed, it follows that 

\begin{equation}
\label{eq:F_asymptotics}
F(z) \sim \begin{cases}
\frac{\log z}{\log (k-1)} & \hspace{-0.1cm}\text{$G$ is a $k$-regular tree, $k \ge 3$;} \\
\left( \frac{ (\ell + 1)!}{2^\ell} z \right)^{\frac{1}{\ell + 1}} & \hspace{-0.1cm}\text{$G$ is a $\ell$-dimensional lattice, $\ell \ge 1$.}
\end{cases}
\end{equation}
We are now ready to state our results on the estimation error. The first establishes a lower bound for the estimation error when $t$ is not too large. 

\begin{lemma}
\label{lemma:estimation_error_lower_bound}
Suppose $G$ is a $k$-regular tree with $k \ge 3$ and that the sequence of candidate sets $\{V_n \}_{n \ge 1}$ satisfies Assumption \ref{as:candidate_set}. There are constants $a_1 = a_1(k, Q_0, Q_1)$ and $b_1 = b_1(k)$ such that
$$
\lim\limits_{n \to \infty} \max\limits_{v \in V_n} \p_v \left( \min_{0 \le t \le F(a_1 \log n) }\E_{\pi(t)} [ d(v^*, \widehat{v}_B(t) )  ] \le b_1 \log n \right) = 0.
$$
Next suppose that $G$ is a $\ell$-dimensional lattice and $\{V_n \}_{n \ge 1}$ satisfies Assumption \ref{as:candidate_set}. There are constants $a_1' = a_1'(Q_0, Q_1)$ and $b_1' = b_1'(\ell)$ such that 
$$
\lim\limits_{n \to \infty} \max\limits_{v \in V_n} \p_v \left(  \min_{0 \le t \le F(a_1' \log n)} \E_{\pi(t)} [ d(v^*, \widehat{v}_B(t) )] \le b_1' n^{1/\ell} \right) = 0.
$$
\end{lemma}

It can be shown through straightforward combinatorial arguments (see Lemmas \ref{lemma:tree_distance_bounds} and \ref{lemma:geodesic_lattice} in Appendix \ref{sec:geodesics}) that the initial estimation error satisfies 
\begin{align}
\label{eq:initial_error}
\E_{\pi(0)}[ d(v^*, \widehat{v}_B(0)) ] & \asymp \begin{cases}
\log n & \text{ $G$ is a $k$-regular tree;} \\
n^{1/\ell} & \text{ $G$ is a $\ell$-dimensional lattice.}
\end{cases}
\end{align}
In light of \eqref{eq:initial_error}, Lemma \ref{lemma:estimation_error_lower_bound} states that with high probability, the estimation error does not significantly decrease for $t \lesssim F(\log n)$. At a high level, this is because the information from the public signals corresponding to the true spread of the cascade is not enough to offset the uncertainty in the source location. These ideas are formalized by computing the mean and variance of $\E_{\pi(t)} [ d(v^*, \widehat{v}(t)) ]$ and applying Chebyshev's inequality. We also remark that as a consequence of our proofs, the constant $a_1$ in Lemma \ref{lemma:estimation_error_lower_bound} depends on the average degree for regular trees, while the constant $a_1'$ does not depend on the average degree in lattices. For details, see Section \ref{subsec:lower_bound}.

The next result establishes an upper bound on the estimation error once $t$ is sufficiently large. 

\begin{lemma}
\label{lemma:estimation_error_upper_bound}
Suppose that $G$ is a $k$-regular tree or a $\ell$-dimensional lattice and that the sequence of candidate sets $\{V_n \}_{n \ge 1}$ satisfies Assumption \ref{as:candidate_set}. There are constants $a_2 = a_2(Q_0, Q_1)$ and $b_2 = b_2 (Q_0, Q_1)$ such that if $t \ge F(a_2 \log n)$, 
\begin{equation}
\label{eq:estimation_error_upper_bound}
\max\limits_{v \in V_n} \p_{v} \left( \E_{\pi(t)} [ d(v^*, \widehat{v}_B(t) ) ] \ge e^{ - b_2 t} \right) \le e^{- b_2 t}.
\end{equation}
\end{lemma}
\noindent The proof relies on large-deviations bounds which show that $\pi_{u}(t) / \pi_{v^*}(t)$ tends to 0 at an exponential rate for any $u \neq v^*$. For details, see Section \ref{subsec:upper_bound}.

\subsection{Putting everything together: Proof of Theorem \ref{thm:bayes}}

Combined, Lemmas \ref{lemma:estimation_error_lower_bound} and \ref{lemma:estimation_error_upper_bound} show that the estimation error exhibits a sharp transition: it is large for $t \lesssim F(\log n)$ and it is exponentially decaying to zero for $t \gtrsim F( \log n)$. The optimality and performance of the estimator $(T_{th}, \widehat{\mathbf{v}}_B)$ is essentially obtained from this observation. 

\begin{proof}[Proof of Theorem \ref{thm:bayes}]
Define the pair of constants $a,A$ so that $(a, A) = (a_1, b_1 \log n)$ if $G$ is a $k$-regular tree and $(a,A) = (a_1', b_1' n^{1/\ell})$ if $G$ is a $\ell$-dimensional lattice (see Lemma \ref{lemma:estimation_error_lower_bound} for definitions of these constants). For every $t \ge 0$, we also define the event
$$
\cE: = \left \{ \min\limits_{0 \le t \le F(a \log n)} \E_{\pi(t)} [ d(v^* , \widehat{v}_B(t)) ] \ge A \right \}.
$$
For any stopping time $T$, we have, for any $v \in V_n$, 
\begin{align}
\E_v \left[ \E_{\pi(T)} [ d(v^*, \widehat{v}_B(T)) ] + T \right] & = \E_v \left[ \sum\limits_{t = 0}^{\infty} \E_{\pi(t)} [ d(v^*, \widehat{v}_B(t)) ] \mathbf{1} (T = t) + T \right] \nonumber \\
& \ge \E_v \left[ \sum\limits_{t = 0}^{F(a \log n)} \E_{\pi(t)} [ d(v^*, \widehat{v}_B(t)) ] \mathbf{1}(\{T = t\} \cap \cE )  + T \mathbf{1}(T > F(a \log n) ) \vphantom{ \sum\limits_{t = 0}^1} \right] \nonumber \\
& \stackrel{(a)}{\ge} A \sum\limits_{t = 0}^{F (a\log n)} \p_v ( \{T = t \} \cap \cE) + F(a \log n) \cdot \p_v ( T > F(a \log n)) \nonumber \\
& \ge A \cdot \p_v ( \{T \le F(a \log n) \} \cap \cE)  + F(a \log n) \cdot \p_v ( \{T > F(a \log n) \} \cap \cE ) \nonumber \\
& \ge \min \{ A, F(a \log n) \} \cdot \p_v ( \cE) \nonumber \\
\label{eq:valB_lower_bound}
& \stackrel{(b)}{=} F(a \log n) \cdot \p_v ( \cE).
\end{align}
Above, $(a)$ follows from the definition of $\cE$ and by lower bounding $T$ by $F(a \log n)$ on the event $\{T > F(a \log n)\}$, and $(b)$ uses $F(a \log n) \le A$ for $n$ sufficiently large, which follows from the asymptotic behavior of $F$ (see \eqref{eq:F_asymptotics}). An important consequence of \eqref{eq:valB_lower_bound} is that the value associated with the pair $(T, \widehat{\mathbf{v}})$ can be lower bounded as 
\begin{align}
\val_B(T, \widehat{\mathbf{v}}_B) & = \E_{\pi(V_n)} [ d(v^*, \widehat{v}_B(T))  + T ] \nonumber \\
& \stackrel{(c)}{=} \E_{\pi(V_n)} \left[ \E_{\pi(T)} [ d(v^*, \widehat{v}_B(T)) ] + T  \right] \nonumber \\
& \stackrel{(d)}{=} \frac{1}{n} \sum\limits_{v \in V_n} \E_v \left[ \E_{\pi(T)} [ d(v^*, \widehat{v}_B(T)) ] + T \right] \nonumber \\
\label{eq:valB_lower_bound_2}
& \stackrel{(e)}{\ge} F(a \log n) \left( \frac{1}{n} \sum\limits_{v \in V_n} \p_v (\cE) \right).
\end{align}
Above, $(c)$ is due to the tower rule, $(d)$ follows since $\pi(V_n)$ is a uniform distribution over elements of $V_n$, and $(e)$ is a consequence of \eqref{eq:valB_lower_bound}. Moreover, since \eqref{eq:valB_lower_bound_2} holds for {\it any} stopping time $T$, we have
\begin{align*}
\val_B^*(V_n) & = \inf\limits_{T, \widehat{\mathbf{v}}} \val_B(T, \widehat{\mathbf{v}}) \\
& = \inf\limits_T \val_B(T, \widehat{\mathbf{v}}_B)\\
& \ge F(a \log n) \left( \frac{1}{n} \sum\limits_{v \in V_n} \p_v(\cE) \right),
\end{align*}
where the second equality follows from the optimality of the estimator $\widehat{\mathbf{v}}_B$, proved in Lemma \ref{lemma:optimal_estimator}. Rearranging terms and sending $n \to \infty$, we arrive at 
\begin{align}
\liminf\limits_{n \to \infty} \frac{ \val_B^*(V_n) }{ F(a \log n) } & \ge \liminf\limits_{n \to \infty} \left( \frac{1}{n} \sum\limits_{v \in V_n} \p_v ( \cE) \right) \nonumber \\
\label{eq:valB_asymptotic_lower_bound}
& \ge \liminf\limits_{n \to \infty} \min\limits_{v \in V_n} \p_v (\cE) = 1,
\end{align}
where the final equality above is a direct consequence of Lemma \ref{lemma:estimation_error_lower_bound}. 

Next, to establish an asymptotic upper bound for the optimal value, we consider the stopping time 
$$
T_{th} : = \min \{t \ge 0: \E_{\pi(t)} [ d(v^*, \widehat{v}_B(t)) ] \le 1 \},
$$
which stops once the estimation error falls below a threshold. We can then bound
\begin{align*}
\val_B  (T_{th},  \widehat{\mathbf{v}}_B) & \stackrel{(f)}{\le} \E_{\pi(V_n)} [ 1 + T_{th} ] = 1 + \sum\limits_{t = 0}^\infty \p_{\pi(V_n)} ( T_{th} > t) \\
& \le 1 + F(a_2 \log n) + \sum\limits_{t = F(a_2 \log n) + 1}^\infty \p_{\pi(V_n)} ( T_{th} > t )\\ 
& \stackrel{(g)}{\le} 1 + F(a_2 \log n) + \sum\limits_{t = F(a_2 \log n) + 1}^\infty \max\limits_{v \in V_n} \p_v(T_{th} > t) \\
& \stackrel{(h)}{\le} 1 + F(a_2 \log n) + \sum\limits_{t = F(a_2 \log n) + 1}^\infty  e^{- b_2 t} \\
& \le 1 + F(a_2 \log n) + \frac{1}{b_2} e^{- b_2 F(a_2 \log n)}.
\end{align*}
Above, the inequality $(f)$ is due to the definition of $T_{th}$, $(g)$ follows since $\p_{\pi(V_n)} (\cdot) : = \frac{1}{n} \sum_{v \in V_n} \p_v ( \cdot) \le \max_{v \in V_n} \p_v(\cdot)$, and $(h)$ follows from noting that $T_{th} > t$ implies that $\E_{\pi(t)} [ d(v^*, \widehat{v}_B(t)) ] > 1 \ge e^{-b_2 t}$ and applying Lemma \ref{lemma:estimation_error_upper_bound} to bound the latter event. Dividing both sides of the final inequality by $F(a_2 \log n)$ and letting $n \to \infty$ shows that
\begin{equation}
\label{eq:valB_asymptotic_upper_bound}
\limsup\limits_{n \to \infty} \frac{ \val_B(T_{th}, \widehat{\mathbf{v}}_B ) }{ F(a_2 \log n)} \le 1.
\end{equation}
The desired result follows from \eqref{eq:valB_asymptotic_lower_bound} and \eqref{eq:valB_asymptotic_upper_bound} by considering the asymptotic behavior of $F$ in trees and lattices (see \eqref{eq:F_asymptotics}). 
\end{proof}

\revision{
\begin{remark}[General temporal cost functions]
\label{remark:general_temporal_cost}
It is also interesting to consider the case where the cost of the stopping time in the Bayesian objective \eqref{eq:bayes_opt} is given by $h(T)$ instead of $T$, where $h$ is some increasing function. It turns out that a slight modification of the proof of Theorem \ref{thm:bayes} shows that the estimator $(T_{th}, \widehat{\mathbf{v}}_B)$ is still asymptotically near-optimal as long as $h$ grows slower than any exponential function. Indeed, if we follow the derivation of the bound in \eqref{eq:valB_lower_bound} and \eqref{eq:valB_lower_bound_2}, we obtain that
$$
\E_{\pi(V_n)} [ d(v^*, \widehat{v}_B(T) + h(T) ] \ge \min \{ A, h ( F(a \log n)) \} \left( \frac{1}{n} \sum_{v \in V_n} \p_v ( \cE ) \right),
$$
where $A, a, \cE$ are the same as in the proof of Theorem \ref{thm:bayes}. Following \eqref{eq:valB_asymptotic_lower_bound}, we obtain that
$$
\liminf_{n \to \infty} \frac{ \inf_{T}  \E_{\pi(V_n)} [ d(v^*, \widehat{v}_B(T)) + h(T) ] }{ \min \{ A, h( F(a \log n) )\} } \ge 1.
$$
In particular, if $h$ increases slower than any exponential function, it can be seen that, for the choices of $A$ and $F$ used for regular trees and lattices, $h(F(a \log n))$ is an asymptotic lower bound on the performance of any estimator. 

The proof of the upper bound on the performance of the optimal estimator can be similarly derived. For the same estimator $( \widehat{ \mathbf{v}}_B, T_{th})$, we may follow the proof of Theorem \ref{thm:bayes} to show that
\begin{align*}
\inf_{(\widehat{\mathbf{v}}, T_{th} ) } \E_{\pi(V_n)} [ d( \widehat{v}(T), v^* ) + h(T) ] & \le \E_{\pi(V_n)} [ 1 + h(T_{th}) ] \\
& = 1 + \sum_{t = 0}^\infty \p_{\pi(V_n)} ( h(T_{th} ) > t ) \\
& = 1 + \sum_{t = 0}^\infty \p_{\pi(V_n)} ( T_{th} > h^{-1} (t) ) \\
& \le 1 + h( F(a_2 \log n) ) + \sum_{t = h( F(a_2 \log n) ) + 1}^\infty e^{ - b_2 h^{-1}(t) }.
\end{align*}
We can reach the same conclusion as in Theorem \ref{thm:bayes}; that is, that 
$$
\limsup_{n \to \infty} \frac{ \inf_{(\widehat{\mathbf{v}}, T)} \E_{\pi(V_n)} [ d(\widehat{v}(t), v^*) + h(T) ] }{ h(F(a_2 \log n)) } \le 1,
$$
provided $\sum_{t \ge 0} e^{- b_2 h^{-1}(t) } < \infty$. This is the case provided $h^{-1}(t)$ increases faster than $\log t$ or equivalently, if $h(t)$ increases slower than any exponential function. Together, this shows that 
$$
h( F(a_1 \log n) ) \lesssim  \inf_{(\widehat{\mathbf{v}}, T)} \E_{\pi(V_n)} [ d(\widehat{v}(t), v^*) + h(T) ] \lesssim h(F(a_2 \log n)),
$$
provided $h$ increases slower than any exponential function. 
\end{remark}
}

\section{Analysis of the minimax setting}
\label{sec:minimax}

In this section we prove Theorem \ref{thm:minimax}. To do so, we prove lower and upper bounds for $\val_M^*(V_n, \alpha)$ in separate theorems. In Section \ref{subsec:minimax_lower_bound}, we prove Theorem \ref{thm:minimax_lower_bound}, which establishes a lower bound for $\val_M^*(V_n, \alpha)$. The upper bounds for $\val_M^*(V_n, \alpha)$ are achieved by MSPRTs of a particular design: this is proved in Theorem \ref{thm:trees} (regular trees) and in Theorem \ref{thm:lattices}, which can be found in Section \ref{subsec:minimax_upper_bound}. Remarkably, a simple design in which we set all the thresholds in the MSPRT to be the same value achieves the lower bound for regular trees. However, the same estimator adapted to lattices fails to achieve the lower bound due to key differences in the topology of lattices. To fix this issue, we consider a novel MSPRT design we term {\it $K$-level thresholds}, where we assign different thresholds to pairs of vertices $u,v$ satisfying $d(u,v) \le K$ and those satisfying $d(u,v) > K$. The performance of the resulting estimator does indeed achieve (up to a constant factor) the lower bound for $\val_M^*(V_n, \alpha)$. 

\subsection{Lower bounding the optimal value}
\label{subsec:minimax_lower_bound}

\begin{theorem}[Lower bound part of Theorem \ref{thm:minimax}]
\label{thm:minimax_lower_bound}
Suppose that the sequence of candidate sets $\{V_n \}_{n \ge 1}$ satisfies Assumption \ref{as:candidate_set}. If $G$ is a $k$-regular tree, 
\begin{equation}
\label{eq:minimax_lb_trees}
\liminf\limits_{n \to \infty} \frac{ \val_M^*(V_n, \alpha)}{\frac{ \log \log n}{\log (k - 1)}} \ge 1.
\end{equation}
On the other hand, if $G$ is a $\ell$-dimensional lattice then there is a constant $a_4$ depending on $\ell, Q_0, Q_1$ such that 
\begin{equation}
\label{eq:minimax_lb_lattices}
\liminf\limits_{n \to \infty} \frac{ \val_M^* (V_n, \alpha)}{ \left( \log n \right)^{\frac{1}{\ell + 1}}} \ge a_4.
\end{equation}
\end{theorem}

\begin{proof}
Let $(T, \widehat{\mathbf{v}})$ be an estimator in $\Delta(V_n, \alpha)$. Given the prior $\pi(V_n)$, recall from Lemma \ref{lemma:optimal_estimator} that the estimator
$$
\widehat{v}_B(t) \in \argmin_{v \in V_n} \E_{\pi(t)} [ d(v^*, v) ]
$$
minimizes $\E_{\pi(V_n)} [ d(v^*, \widehat{v}(T) ) ]$ over all estimators $\widehat{\mathbf{v}}$, for {\it any} stopping time $T$. 

For given constants $a, A$, define the events 
\begin{align*}
\cE_1 & : = \{ T \le F(a \log n) \} \\
\cE_2 & : = \left \{ \min_{0 \le t \le F(a \log n) } \E_{\pi(t)} [ d(v^*, \widehat{v}_B(t) ) ] \ge A \right \}.
\end{align*}
We can then lower bound the worst-case estimation error as 
\begin{align}
 \max\limits_{v \in V_n} \E_v [ d(v, \widehat{v}(T) ) ] & \stackrel{(a)}{\ge}  \E_{\pi(V_n)} [ d(v^*, \widehat{v}(T) ] \nonumber \\
& \stackrel{(b)}{\ge} \E_{\pi(V_n)} \left[  d(v^*, \widehat{v}_B(T)) \right] \nonumber \\
& \stackrel{(c)}{\ge} \E_{\pi(V_n)} \left[ \E_{\pi(T)} [ d(v^*, \widehat{v}_B(T)) ] \mathbf{1}(\cE_1 \cap \cE_2 ) \right] \nonumber \\
\label{eq:minimax_estimation_lower_bound}
& \stackrel{(d)}{\ge} A \cdot \p_{\pi(V_n)} ( \cE_1 \cap \cE_2).
\end{align}
Above, $(a)$ follows since $\E_{\pi(V_n)}$ is an average over the collection of operators $\{\E_v \}_{v \in V_n}$ and the maximum is greater than the average; $(b)$ is due to the optimality of $\widehat{\mathbf{v}}_B$, proved in Lemma \ref{lemma:optimal_estimator}; inequality $(c)$ is due to the tower rule and inequality $(d)$ follows since $\E_{\pi(T)} [ d(v^*, \widehat{v}_B(T)) ] \ge A$ on the event $\cE_1 \cap \cE_2$. 

Noting that $\max_{v \in V_n} \E_{\pi(V_n)} [ d(v^*, \widehat{v}(T)) ] \le \alpha$ for $(T, \widehat{\mathbf{v}}) \in \Delta(V_n, \alpha)$, \eqref{eq:minimax_estimation_lower_bound} implies
$$
\p_{\pi(V_n)} (\cE_1 \cap \cE_2) \le \frac{\alpha}{A}, 
$$
which in turn implies that 
\begin{align}
\p_{\pi(V_n)} ( \cE_1^c) & \ge \p_{\pi(V_n)} ( \cE_1^c \cup \cE_2^c) - \p_{\pi(V_n)} ( \cE_2^c) \nonumber \\
\label{eq:E1c_lower_bound}
& \ge 1 - \frac{\alpha}{A} - \p_{\pi(V_n)} ( \cE_2^c).
\end{align}
By Lemma \ref{lemma:estimation_error_lower_bound}, if we set $(a, A) = (a_1, b_1 \log n)$ if $G$ is a $k$-regular tree and $(a,A) = (a_1', b_1' n^{1/\ell})$ if $G$ is a $\ell$-dimensional lattice (see Lemma \ref{lemma:estimation_error_lower_bound} for definitions of these constants) then the right hand side of \eqref{eq:E1c_lower_bound} tends to 1 as $n \to \infty$. We can then lower bound the worst-case expected runtime as 
\begin{align}
\max_{v \in V_n} \E_v [ T ] & \ge \E_{\pi(V_n)} [ T ] \ge \E_{\pi(V_n)} [ T \mathbf{1}(\cE_1^c) ] \nonumber \\
\label{eq:general_lower_bound}
& \ge F(a \log n) \p_{\pi(V_n)} ( \cE_1^c). 
\end{align}
Since \eqref{eq:general_lower_bound} holds for {\it any} element of $\Delta(V_n, \alpha)$, we have
$$
\val_M^*(V_n,\alpha) \ge F(a \log n) \p_{\pi(V_n)} (\cE_1^c). 
$$
Dividing by $F( a\log n)$ and sending $n \to \infty$, we arrive at
$$
\liminf\limits_{n \to \infty} \frac{ \val_M^*(V_n, \alpha)}{F(a \log n)} \ge 1.
$$
The desired result follows from the asymptotic behavior of $F$ (see \eqref{eq:F_asymptotics}). 
\end{proof}

\subsection{Achieving the lower bound with MSPRTs}
\label{subsec:minimax_upper_bound}

An important question is whether the lower bound in Theorem \ref{thm:minimax_lower_bound} is achievable. To answer this question, we focus on a class of sequential estimation procedures called {\it multi-hypothesis sequential probability ratio tests} (MSPRTs). For convenience, we define them formally below. 

Given vertices $u,v \in V$, define the log-likelihood ratio 
$$
Z_{vu}(t) : = \log \frac{d\p_v}{d \p_u} ({y}(0), \ldots, {y}(t)).
$$
In our model, the likelihood of observing the sequence of public signals $\mathbf{y}(0), \ldots, \mathbf{y}(t)$ under the measure $\p_v$ has the form
\begin{equation}
\label{eq:public_signal_likelihood}
d\p_v ( {y}(0), \ldots, {y}(t)) = \prod\limits_{s = 0}^t \left( \left[ \prod\limits_{w \in \cN_v(s)} dQ_1( y_w(s)) \right] \left[ \prod\limits_{w \notin  \cN_v(s)} dQ_0 ( y_w(s)) \right] \right).
\end{equation}
Above, the product containing terms of the form $dQ_1(y_w(s))$ computes the likelihood of nodes that have been affected by the cascade, and the product containing the terms of the form $dQ_0(y_w(s))$ computes the likelihood of nodes that have not yet been affected by the cascade. In light of \eqref{eq:public_signal_likelihood}, $Z_{vu}(t)$ can be written as

\begin{align*}
Z_{vu}(t) & =  \sum\limits_{s = 0}^t  \log \frac{ \left( \prod\limits_{w \in \cN_v(s)} dQ_1(y_w(s)) \right) \left( \prod\limits_{w \notin \cN_v(s)} dQ_0(y_w(s)) \right) }{ \left( \prod\limits_{w \in \cN_u(s)} dQ_1(y_w(s)) \right) \left( \prod\limits_{w \notin \cN_u(s)} dQ_0(y_w(s)) \right) } \\
& = \sum\limits_{s = 0}^t \left( \sum\limits_{w \in \cN_v(s) \setminus \cN_u(s)} \log \frac{dQ_1}{dQ_0}(y_w(s)) - \sum\limits_{w \in \cN_u(s) \setminus \cN_v(s)} \log \frac{dQ_1}{dQ_0}(y_w(s)) \right).
\end{align*}

In words, the log-likelihood ratio $\log dQ_1 / dQ_0$ measures how likely it is that a certain public signal came from the distribution $Q_1$ as opposed to the distribution $Q_0$. It follows that summations of the form $\sum_{w \in S} \log \frac{dQ_1}{dQ_0}(y_w(s))$ measures the {\it net likelihood} that the public signals in $S$ are caused by the cascade. Hence if it is more likely that the public signals in $\cN_v(s) \setminus \cN_u(s)$ are caused by the cascade rather than the public signals in $\cN_u(s) \setminus \cN_v(s)$, the log-likelihood ratio $Z_{vu}(t)$ positive; else it is negative. 

We are now ready to define the MSPRT.

\begin{definition}[Multi-hypothesis sequential probability ratio test]
Fix a positive integer $n$ and specify a threshold function $\tau: V_n \times V_n \to (0,\infty)$. Consider the stopping time
$$
T(v) : = \min \{t \ge 0: Z_{vu}(t) \ge  \log \tau(v,u), \forall u \in V_n \setminus \{v\} \}.
$$
The corresponding MSPRT $(T, \widehat{\mathbf{v}})$ is defined via 
$$
T : = \min_{v \in V_n} T(v) \qquad  and \qquad \widehat{v}(T) : = \argmin_{v \in V_n} T(v).
$$
\end{definition}
In words, the output of the MSPRT is the first vertex $v$ for which all log-likelihood ratios $Z_{vu}(t)$ pass the thresholds $\log \tau(v,u)$ for all $u \in V_n \setminus \{v \}$. If $(T, \widehat{\mathbf{v}})$ is a MSPRT with threshold function $\tau$, we have the following useful relation, which is due to Tartakovsky \cite[Theorem 3.1]{Tartakovsky1998}.

\begin{lemma}
\label{lemma:msprt_prob_bound}
For any distinct $u,v \in V_n$, 
$$
\p_v ( \widehat{v}(T) = u) \le \frac{1}{\tau(u,v)}
$$
\end{lemma}

Since the proof is short, we provide it here for completeness. 

\begin{proof}[Proof of Lemma \ref{lemma:msprt_prob_bound}]
\begin{align*}
\p_v ( \widehat{v}(T) = u) & = \E_v \left[ \mathbf{1}( \widehat{v}(T) = u) \right] \\
& = \E_u \left[ \mathbf{1}(\widehat{v}(T) = u) e^{- Z_{uv}(T)} \right] \\
& \le e^{ - \log \tau(u,v) } \E_u \left[ \mathbf{1}(\widehat{v}(T) = u) \right] \\
& \le \frac{1}{\tau(u,v)}.
\end{align*}
Above, the equality in the second line follows from $e^{-Z_{uv}(T)} = \frac{d\p_v}{d\p_u}( y(0), \ldots, y(T))$. The inequality in the third line follows since $Z_{uv}(T) \ge \tau(u,v)$ on the event $\{\widehat{v}(T) = u\}$ by the definition of a MSPRT. 
\end{proof}

Using Lemma \ref{lemma:msprt_prob_bound}, we can bound the worst-case estimation error of the MSPRT $(T, \widehat{\mathbf{v}})$ as 
\begin{align*}
\max\limits_{v \in V_n} \E_v [ d(v,\widehat{v}(T))] & = \max\limits_{v \in V_n} \sum\limits_{v \in V_n} d(u,v) \p_v ( \widehat{v}(T) = u) \\
& \le \max\limits_{v \in V_n} \sum\limits_{v \in V_n} \frac{d(u,v)}{\tau(u,v)}.
\end{align*}
To ensure that $(T, \widehat{\mathbf{v}}) \in \Delta(V_n,\alpha)$, it suffices to check that 
\begin{equation}
\label{eq:msprt_error}
\max\limits_{v \in V_n} \sum_{u \in V_n} \frac{d(u,v)}{\tau(u,v)} \le \alpha.
\end{equation}
Perhaps the simplest weight design which satisfies \eqref{eq:msprt_error} is a {\it uniform weights design}, in which the $\tau(v,u)$'s all take on the same value. 

\begin{definition}[Uniform weights design]
The MSPRT $(T_{n,\alpha}, \widehat{\mathbf{v}}_{n,\alpha})$ is designed with uniform weights if
\begin{equation}
\label{eq:uniform_weights}
\tau(v,u) := \frac{n^2}{\alpha}, \qquad \forall u,v \in V_n : u \neq v.
\end{equation}
\end{definition}
With the uniform weights design, we have 
$$
 \max\limits_{v \in V_n} \sum\limits_{u \in V_n} d(u,v) \frac{\alpha}{n^2} \le \max\limits_{v \in V_n} \sum\limits_{u \in V_n} \frac{\alpha}{n} = \alpha.
$$ 
Hence $(T_{n,\alpha}, \widehat{\mathbf{v}}_{n,\alpha}) \in \Delta(V_n, \alpha)$. Note that in the first equality above, we have used the (loose) bound $\max_{u,v \in V_n} d(v,u) \le n$. 

When $G$ is a regular tree, the following result shows that the performance of the MSPRT with uniform weights matches the lower bound in Theorem \ref{thm:minimax_lower_bound}. 

\begin{theorem}
\label{thm:trees}
Let $G$ be a $k$-regular tree with $k \ge 3$ and fix $\alpha > 0$. If $(T_{n,\alpha}, \widehat{\mathbf{v}}_{n,\alpha})$ is the MSPRT with uniform weights,
$$
\limsup\limits_{n \to \infty} \frac{ \max_{v \in V_n} \E_v [ T_{n,\alpha} ] }{ \frac{\log \log n}{\log (k - 1)} } \le 1.
$$
\end{theorem}

Combined with Theorem \ref{thm:minimax_lower_bound}, Theorem \ref{thm:minimax} for regular trees follows as an immediate consequence. 

\begin{proof}[Proof of Theorem \ref{thm:minimax} for regular trees]
We have the series of inequalities
\begin{align*}
1 & \stackrel{(a)}{\le} \liminf\limits_{n \to \infty} \frac{\val_M^*(V_n, \alpha)}{ \frac{\log \log n}{\log (k - 1)}} \\
& \le \limsup\limits_{n \to \infty} \frac{\val_M^*(V_n, \alpha)}{\frac{\log \log n}{\log (k - 1)}} \\ & \stackrel{(b)}{\le} \limsup\limits_{n \to \infty} \frac{\max_{v \in V_n} \E_v [ T_{n, \alpha}] }{ \frac{\log \log n}{\log (k - 1)}} \\
& \stackrel{(c)}{\le} 1.
\end{align*}
Above, $(a)$ is due to Theorem \ref{thm:minimax_lower_bound}, $(b)$ follows since $(T_{n,\alpha}, \widehat{\mathbf{v}}_{n,\alpha}) \in \Delta(V_n, \alpha)$ and $(c)$ is due to Theorem \ref{thm:trees}. Since both the start and the end of the chain of inequalities is 1, the inequalities are all {\it equality}. Hence the following limit is well-defined:
$$
\lim\limits_{n \to \infty} \frac{\val_M^*(V_n, \alpha)}{\frac{\log \log n}{\log (k -1)}} = 1,
$$
which proves the desired result. 
\end{proof}

We provide a brief sketch of the proof of Theorem \ref{thm:trees}, and defer the details to Section \ref{subsec:proof_msprt_trees}. Suppose that $v \in V_n$ is the true source. When $t$ is sufficiently large, we show that it holds for all $u \in V_n \setminus \{v \}$ that 
$$
Z_{vu}(t) \sim \E_v [ Z_{vu}(t) ] \asymp \sum_{s = 0}^t  | \cN_v(s) \setminus \cN_u(s) | \asymp (k- 1)^t.
$$
This in particular implies that once $t \gtrsim \frac{\log \log (n^2/\alpha)}{\log (k- 1)}$, {\it all} the log-likelihood ratios $\{Z_{vu}(t) \}_{u \in V_n \setminus \{u \}}$ will cross the threshold $\log n^2/\alpha$. Since $\log \log (n^2/\alpha) \sim \log \log n$ when $n$ is large, Theorem \ref{thm:trees} follows. The key technical ingredient of this proof is a large-deviations-type inequality for $Z_{vu}(t)$. 

Unfortunately, the MSPRT with a uniform weights design is {\it not} optimal in lattices. Suppose that the dimension of the lattice is $\ell$. For $u,v \in V_n$ which are far apart, we have $Z_{vu}(t) \asymp t^{\ell + 1}$, but for $u,v$ which are relatively close together, $Z_{vu}(t) \asymp t^\ell$. As the log-likelihood $Z_{vu}(t)$ grows at a slower rate in this latter case, this ends up being the primary contributor to the behavior of the stopping time. As a result, we obtain an upper bound for $T^*(V_n, \alpha)$ of order $(\log n)^{1/\ell}$, whereas the lower bound established in Theorem \ref{thm:minimax_lower_bound} is of order $(\log n)^{1/(\ell + 1)}$. To close this gap, we shall consider a design for the MSPRT weights which places different thresholds for pairs of vertices that are close and pairs that are far. 

\begin{definition}[$K$-level weights]
\label{def:K_level}
Let $K$ be a non-negative integer. The MSPRT $(T^K_{n, \alpha}, \widehat{\mathbf{v}}_{n,\alpha}^K)$ is designed with $K$-level weights if 
$$
\tau(v,u) := \begin{cases}
\frac{2K | \cN(K) |}{\alpha} & 0 < d(u,v) \le K \\
\frac{2n^2}{\alpha} & \text{ else.}
\end{cases}
$$
\end{definition}

It is straightforward to show that the $K$-level weights satisfy \eqref{eq:msprt_error}:
\begin{align*}
\max\limits_{v \in V_n}  \sum\limits_{u \in V_n} \frac{d(u,v)}{\tau(u,v)}  & = \max\limits_{v \in V_n} \left( \sum\limits_{u \in V_n \cap \cN_v(K)} \frac{d(u,v)}{\tau(u,v)}  + \sum\limits_{u \in V_n \setminus \cN_v(K) } \frac{d(u,v)}{\tau(u,v)} \right) \\
& \le \max\limits_{v \in V_n} \left( \sum\limits_{u \in V_n \cap \cN_v(K)} \frac{\alpha}{2 | \cN(K)|} + \sum\limits_{u \in V_n \setminus \cN_v(K) } \frac{\alpha}{2 n} \right) \\
& \le \alpha,
\end{align*}
where to derive the inequality on the second line, we used $d(u,v) \le K$ for $u \in \cN_v(K)$ and $d(u,v) \le n$ for $u \in V_n$. Putting everything together, we have shown that $(T_{n, \alpha}^K, \widehat{\mathbf{v}}_{n,\alpha}^K) \in \Delta(V_n, \alpha)$. 

At a high level, the MSPRT designed with $K$-level weights can be thought of as a {\it multi-scale} source estimation algorithm. When the $Z_{vu}(t)$'s are large for many vertices $u$ far from $v$ (specifically, $d(u,v) > K$), this indicates that $v$ must be relatively close to the source. Assuming $K$ is not too large, there are roughly $n$ vertices far from $v$, which means that the threshold for the log-likelihood ratios {\it must} be on the order of $\log n$ in order to reliably narrow down the general location of the source. Simultaneously, the $Z_{vu}(t)$'s for $u$ close to $v$ (specifically, $d(u,v) \le K$) provide {\it fine-grained} information about the location of the source within the localized region $\cN_v(K)$. To compensate for the slower growth of $Z_{vu}(t)$ in this case, we require that they pass the much smaller threshold of $\log ( 2K | \cN(K) |/ \alpha)$, hence removing the bottleneck found in the uniform weights design. 

In the following result, we show that for the right value of $K$, the MSPRT designed with $K$-level weights is {\it orderwise optimal} in lattices. 

\begin{theorem}
\label{thm:lattices}
Let $G$ be an $\ell$-dimensional lattice and set $K = (\log n)^{1/\ell}$. Then there is a constant $b_4$ depending only on $\ell, Q_0, Q_1$ such that 
$$
 \limsup\limits_{n \to \infty} \frac{ \max_{v \in V_n} \E_v [ T^K_{n, \alpha} ] }{ \left(  \log n \right)^{\frac{1}{\ell + 1}} } \le b_4,
$$
where $a_4$ is the constant used in \eqref{eq:minimax_lb_lattices}. 
\end{theorem}

Combined with Theorem \ref{thm:minimax_lower_bound}, Theorem \ref{thm:minimax} for lattices immediately follows. 

\begin{proof}[Proof of Theorem \ref{thm:minimax} for lattices]
Let $G$ be the $\ell$-dimensional lattice. We have the series of inequalities 
\begin{align*}
a_4 & \stackrel{(a)}{\le} \liminf\limits_{n \to \infty} \frac{ \val_M^*(V_n, \alpha)}{ (\log n)^{\frac{1}{\ell + 1}}} \\
& \le \limsup\limits_{n \to \infty} \frac{\val_M^*(V_n, \alpha)}{ (\log n)^{\frac{1}{\ell + 1}} } \\ & \stackrel{(b)}{\le} \limsup\limits_{n \to \infty} \frac{\max_{v \in V_n} \E_v [ T_{n, \alpha}^K ] }{ (\log n)^{\frac{1}{\ell + 1}} } \\
& \stackrel{(c)}{\le} b_4.
\end{align*}
Above, $(a)$ is due to Theorem \ref{thm:minimax_lower_bound}, $(b)$ follows since $(T_{n, \alpha}^K, \widehat{\mathbf{v}}_{n, \alpha}^K ) \in \Delta(V_n, \alpha)$, and $(c)$ is due to Theorem \ref{thm:lattices}. In particular, the chain of inequalities implies that for $n$ sufficiently large, 
\begin{align*}
\frac{a_4}{2} (\log n)^{\frac{1}{\ell + 1}} & \le \val_M^*(V_n, \alpha) \\
& \le \max\limits_{v \in V_n}  \E_v [ T_{n,\alpha}^K ] \\
& \le 2b_4 (\log n)^{\frac{1}{\ell + 1}}.
\end{align*}
The theorem follows from setting $a' : = a_4 / 2$ and $b' : = 2 b_4$. 
\end{proof}

The proof of Theorem \ref{thm:lattices} follows similar reasoning as the proof of Theorem \ref{thm:trees}. In the case that $v$ is the true source of the cascade, we separately consider the performance of $Z_{vu}(t)$ for $u \in \cN_v(K)$ and $u \in V_n \setminus \cN_v(K)$ and use large-deviations-type results to characterize when the log-likelihood ratios cross the thresholds specified by the $K$-level weights design. For full details, see Section \ref{subsec:proof_msprt_lattices}. 

\begin{remark}[MSPRTs for general topologies]
\label{remark:general_msprt}
The $K$-level weights introduced in Definition \ref{def:K_level} can be generalized to arbitrary topologies if we set 
$$
\tau(v,u) : = \begin{cases}
\frac{2K | \cN_v(K) | }{\alpha} & 0 < d(v,u) \le K ;\\
\frac{2n^2}{\alpha} & \text{ else}.
\end{cases}
$$
Above, we write $| \cN_v(K) |$ instead of $|\cN(K)|$ since the size of the $K$-hop neighborhood of a vertex $v$ can depend strongly on $v$ in general topologies. Since the proof of Theorem \ref{thm:lattices} is quite generic for the most part, we expect that a MSPRT with $K$-level weights can achieve the lower bound for $\val_M^*(V_n, \alpha)$ in general. The choice of $K$, however, will depend on the topology of interest.
\end{remark}

\section{Analysis of the Bayesian estimation error: Proof of Lemmas \ref{lemma:estimation_error_lower_bound} and \ref{lemma:estimation_error_upper_bound}}
\label{sec:estimation_error}
\subsection{Preliminary results: Properties of the posterior distribution}

Before proving Lemmas \ref{lemma:estimation_error_lower_bound} and \ref{lemma:estimation_error_upper_bound}, we introduce some simple supporting results. Recall that the posterior distribution is given by $\pi(t) = \{ \pi_v(t) \}_{v \in V_n}$, where 
$$
\pi_v(t) : = \p_{\pi(V_n)} \left( v^* = v \mid y(0),\ldots, y(t) \right).
$$
From Bayes' formula, it holds for any distinct $u,v \in V_n$ that 
\begin{equation}
\label{eq:pi_recursion}
\frac{\pi_v(t)}{\pi_u(t)} = \frac{\pi_v(t - 1)}{\pi_u(t - 1)} \cdot \frac{d\p_v(y(t))}{d\p_u(y(t))}. 
\end{equation}
Recall that under $\p_v$, $y_w(t) \sim Q_1$ if $w \in \cN_v(t)$, else $y_w(t) \sim Q_0$. Since the public signals $y(0), \ldots, y(t)$ are independent conditioned on the source, the likelihood ratio $d\p_v(y(t))/ d\p_u(y(t))$ can be written as 
\begin{align}
\frac{d\p_v(y(t))}{d \p_u(y(t))} & =  \frac{ \left( \prod\limits_{w \in \cN_v(t)} dQ_1(y_w(t)) \right) \left( \prod\limits_{w \notin \cN_v(t)} dQ_0(y_w(t)) \right)}{ \left( \prod\limits_{w \in \cN_u(t)} dQ_1(y_w(t)) \right) \left( \prod\limits_{w \notin \cN_u(t)} dQ_0 (y_w(t)) \right) } \nonumber \\
\label{eq:lr_decomposition}
& = \frac{ \prod_{w \in \cN_v(t) \setminus \cN_u(t)} \frac{dQ_1}{dQ_0}(y_w(t)) }{ \prod_{w \in \cN_u(t) \setminus \cN_v(t)} \frac{dQ_1}{dQ_0}(y_w(t)) }
\end{align}
Combining \eqref{eq:pi_recursion} and \eqref{eq:lr_decomposition}, we have
\begin{align}
\label{eq:ratio_pi_v1}
\frac{\pi_v(t)}{\pi_u(t)} & = \frac{ \prod_{s = 0}^t \prod_{w \in \cN_v(s) \setminus \cN_u(s)} \frac{dQ_1}{dQ_0}(y_w(s)) }{ \prod_{s = 0}^t  \prod_{w \in \cN_u(s) \setminus \cN_v(s)} \frac{dQ_1}{dQ_0}(y_w(s)) } \\
\label{eq:ratio_pi_v2}
& = \frac{ \prod_{s = 0}^t \prod_{w \in \cN_v(s) } \frac{dQ_1}{dQ_0}(y_w(s)) }{ \prod_{s = 0}^t  \prod_{w \in \cN_u(s) } \frac{dQ_1}{dQ_0}(y_w(s)) }.
\end{align}
Equation \eqref{eq:ratio_pi_v1} follows directly from \eqref{eq:pi_recursion} and \eqref{eq:lr_decomposition}. The difference between \eqref{eq:ratio_pi_v1} and \eqref{eq:ratio_pi_v2} is that we take a product over $w \in \cN_u(s) \setminus \cN_v(s)$ and $w \in \cN_v(s) \setminus \cN_u(s)$ in the former, and $w \in \cN_u(s)$ and $w \in \cN_v(s)$ in the latter. The expressions are equivalent since the vertices in $\cN_u(s) \cap \cN_v(s)$ are cancelled out. We display both equations, as each will be useful in different contexts. 

Equation \eqref{eq:ratio_pi_v2} implies that the posterior probabilities can be written as 
$$
\pi_v(t) = \frac{1}{Y(t)} \prod\limits_{s = 0}^t \prod\limits_{w \in \cN_v(s)} \frac{dQ_1}{dQ_0}(y_w(s)),
$$
where the normalizing constant, $Y(t)$, is given by 
$$
Y(t) : = \sum\limits_{v \in V_n} \prod\limits_{s = 0}^t \prod\limits_{w \in \cN_v(s)} \frac{dQ_1}{dQ_0} (y_w(s)).
$$
It will be convenient to use the notation $\pi_v(t) = X_v(t) / Y(t)$, where $X_v(t)$ is explicitly given by 
$$
X_v(t) : = \prod\limits_{s = 0}^t \prod\limits_{w \in \cN_v(s)} \frac{dQ_1}{dQ_0} (y_w(s)).
$$
With this notation, $Y(t) = \sum_{u \in V_n} X_u(t)$. The following lemma establishes some basic properties of the collection $\{X_u(t) \}_{u \in V_n}$. 

\begin{lemma}
\label{lemma:xu_mean}
Denote 
$$
\beta : = \E_{\cA \sim Q_1} \left [ \frac{dQ_1}{dQ_0}(\cA) \right].
$$
Then $\beta > 1$ and for any $u,v \in V_n$ and $t \ge 0$, 
$$
\E_v [ X_u(t) ] = \beta^{ \sum_{s = 0}^t | \cN_v(s) \cap \cN_u(s) |}.
$$
\end{lemma} 

\begin{proof}
We start by proving $\beta > 1$. We can write
\begin{align*}
\E_{\cA \sim Q_1} \left[ \frac{dQ_1}{dQ_0}(\cA) \right] & = \E_{\cB \sim Q_0} \left[ \left( \frac{dQ_1}{dQ_0}(\cB) \right)^2 \right] \\
& \ge \E_{\cB \sim Q_0} \left[ \frac{dQ_1}{dQ_0}(\cB) \right]^2 = 1.
\end{align*}
Above, the first equality is due to a change of measure (a valid operation since $Q_0, Q_1$ are mutually absolutely continuous), and the inequality is due to Jensen's inequality. Above, the inequality is strict since the equality case only occurs if $\frac{dQ_1}{dQ_0}(\cB)$ is a constant (equivalently, $Q_0 = Q_1$). 

Since $\{y_w(s) \}_{w \in V, 0 \le s \le t}$ is a collection of independent random variables conditioned on $v^* = v$, we have
\begin{align}
\label{eq:xu_expectation_expansion}
\E_v [ X_u(t) ] & = \E_v \left [ \prod\limits_{s = 0}^t \prod\limits_{w \in \cN_u(s)} \frac{dQ_1}{dQ_0} (y_w(s)) \right] \\
& = \prod\limits_{s = 0}^t \prod\limits_{w \in \cN_u(s)} \E_v \left [ \frac{dQ_1}{dQ_0} (y_w(s)) \right].
\end{align}
For each term in the product, we have
\begin{equation}
\label{eq:llr_expectation}
\E_v \left [ \frac{dQ_1}{dQ_0} (y_w(s)) \right] = \begin{cases}
\beta & d(v,w) \le s \\
1 & \text{ else,}
\end{cases}
\end{equation}
where we have used the fact that $y_w(s) \sim Q_1$ in the first case, and $y_w(s) \sim Q_0$ in the second. The desired statement follows from substituting \eqref{eq:llr_expectation} into \eqref{eq:xu_expectation_expansion}.
\end{proof}

The following lemma bounds the covariance between $X_u(t)$ and $X_v(t)$. As a matter of notation, we recall that the {\it neighborhood growth function} (defined in \eqref{eq:neighborhood_growth_fn}) is given by $f(t) : = \sum_{s = 0}^t | \cN(s)|$. 

\begin{lemma}
\label{lemma:xu_cov}
For any $u,v,w \in V_n$ and $t \ge 0$, there is a constant $\lambda = \lambda(Q_0,Q_1)$ such that $\mathrm{Cov}_v(X_u(t), X_w(t)) = 0$ if $d(u,w) > 2t$ and $\mathrm{Cov}_v(X_u(t), X_w(t)) \le \lambda^{f(t)}$ if $d(u,w) \le 2t$. 
\end{lemma}

\begin{proof}
Since $X_u(t)$ depends only on the signals in $\cN_u(t)$, it is clear that $X_u(t)$ and $X_w(t)$ are independent under $\p_v$ if $d(u,w) > 2t$. Hence $\mathrm{Cov}_v(X_u(t), X_w(t)) = 0$ in this case. To handle the case where $d(u,w) \le 2t$, we first define 
\begin{align*}
\lambda_1 & : = \mathop{\E}_{\cA \sim Q_1} \left [ \left( \frac{dQ_1}{dQ_0} (\cA) \right)^2 \right] , \lambda_0  : = \mathop{\E}_{\cB \sim Q_0} \left [ \left( \frac{dQ_1}{dQ_0}(\cB) \right)^2 \right].
\end{align*}
We have the following bound on the covariance due to the Cauchy-Schwartz inequality. 
\begin{align*}
\mathrm{Cov}_v(X_u(t), X_w(t) ) & \le \E_v [ X_u(t) X_w(t) ] \\
& \le \E_v [ X_u(t)^2 ]^{1/2} \E_v [ X_w(t)^2 ]^{1/2}.
\end{align*}
To bound $\E_v [ X_u(t)^2 ]$, we can write
\begin{align*}
\E_v [ X_u(t)^2 ] & = \prod\limits_{s = 0}^t \prod\limits_{a \in \cN_u(s)} \E_v \left [ \left( \frac{dQ_1}{dQ_0}(y_a(s)) \right)^2 \right] \\
& =  \lambda_1^{ \sum_{s = 0}^t | \cN_u(s) \cap \cN_v(s) |} \lambda_0^{\sum_{s = 0}^t | \cN_u(s) \setminus \cN_v(s) |} \\
& \le (\max\{\lambda_0, \lambda_1\})^{ \sum_{s = 0}^t | \cN_u(s) |} \\
& = (\max\{ \lambda_0, \lambda_1 \})^{f(t)}.
\end{align*}
Since the bound we have derived holds for any $u \in V_n$, it follows that $\mathrm{Cov}_v(X_u(t), X_w(t)) \le (\max\{\lambda_0, \lambda_1 \})^{f(t)}$, which proves the desired claim with $\lambda := \max\{\lambda_0, \lambda_1\}$. 
\end{proof}

The results we have established allow us to prove the following concentration result for $Y(t)$ when $t$ is not too large. 

\begin{lemma}
\label{lemma:Y}
Recall the constants $\beta = \beta(Q_0, Q_1)$ (defined in Lemma \ref{lemma:xu_mean}) and $\lambda = \lambda(Q_0, Q_1)$ (defined in \ref{lemma:xu_cov}). Furthermore let $F : = f^{-1}$ be the inverse of the neighborhood growth function. If
$$
t \le F \left( \frac{\log n}{4 \log (\max \{ \beta, \lambda \} )} \right),
$$
then for any $\epsilon > 0$, 
\begin{equation}
\label{eq:lemma_final_chebyshev_Y}
\max\limits_{v \in V_n} \p_v ( | Y(t) - n | \ge \epsilon n ) \le \frac{4}{\epsilon^2 \sqrt{n}}.
\end{equation}
\end{lemma}

\begin{proof}
Let $\beta = \beta(Q_0, Q_1)$ be the constant defined in Lemma \ref{lemma:xu_mean}. We begin by computing the expectation of $Y(t)$ with respect to $\p_v$. 
\begin{align*}
\E_v [ Y(t) ] & = \sum\limits_{u \in V_n} \E_v [ X_u(t)] \\
& = \sum\limits_{u \in V_n} \beta^{ \sum_{s = 0}^t | \cN_v(s) \cap \cN_u(s) |} \\
& = \sum\limits_{u \in \cN_v(2t)} \beta^{ \sum_{s = 0}^t | \cN_v(s) \cap \cN_u(s) |}  + | V_n \setminus \cN_v(2t) |.
\end{align*}
Since $\beta > 1$ and $\sum_{s = 0}^t | \cN_v(s) \cap \cN_u(s)| \le f(t)$, we have the bounds
$$
n \le \E_v[Y(t) ] \le n + |\cN(2t) | \beta^{f(t)}.
$$
From Lemmas \ref{lemma:trees_exact} and \ref{lemma:lattices_exact}, we have the asymptotics 
\begin{align*}
| \cN(2t) | & \asymp \begin{cases}
(k - 1)^{2t} & \text{ $G$ is a $k$-regular tree}; \\
t^\ell & \text{ $G$ is a $\ell$-dimensional lattice}
\end{cases} \\
f(t) & \asymp \begin{cases}
(k-1)^t & \text{ $G$ is a $k$-regular tree}; \\
t^{\ell + 1} & \text{ $G$ is a $\ell$-dimensional lattice}.
\end{cases} 
\end{align*}
Hence we have, for $t$ sufficiently large, the simpler upper bound of $n + \beta^{2 f(t)}$ for $\E_v[Y(t)]$. Next, suppose that $t$ satisfies $f(t) \le \frac{\log n}{4 \log \beta }$ so that $\beta^{2f(t)} \le \sqrt{n}$. Then
\begin{align}
\p_v \left( | Y(t) - n | \ge \epsilon n \right) & \le \p_v \left( | Y(t) - \E_v [ Y(t) ]| \ge \frac{\epsilon}{2} n \right) \nonumber \\
\label{eq:Y_chebyshev}
& \le \frac{4 \cdot \mathrm{Var}_v( Y(t))}{\epsilon^2 n^2},
\end{align}
where the first inequality holds if $\epsilon \ge \frac{2}{\sqrt{n}}$ and the second inequality is due to Chebyshev's inequality. Using Lemma \ref{lemma:xu_cov}, we can bound the variance of $Y(t)$ as 
\begin{align*}
\mathrm{Var}_v ( Y(t)) & = \sum\limits_{u,w \in V_n} \mathrm{Cov}_v ( X_u(t), X_w(t)) \\
& \le \sum\limits_{u \in V_n} | \cN(2t)| \lambda^{f(t)} \\
& \le 2n f(t)^2\lambda^{f(t)} \\
& \le n \lambda^{2f(t)},
\end{align*}
where we have used $|\cN(2t) | \le f(2t) \le 2 f(t)^2$ and $\lambda = \lambda(Q_0, Q_1)$ is defined in Lemma \ref{lemma:xu_cov}. Moreover, if $f(t) \le \frac{\log n}{4 \log \lambda}$, then $\lambda^{2f(t)} \le \sqrt{n}$ and $\mathrm{Var}_v(Y(t)) \le n^{3/2}$. The desired result follows. 
\end{proof}

Our final result establishes exponential lower tail bounds for the ratio $\pi_v(t) / \pi_u(t)$ using Chernoff bounds. Before stating the result, we define some notation. Let
\begin{equation}
\label{eq:symmetric_KL}
D(Q_0, Q_1) : = \E_{\cA \sim Q_1} \left [ \log \frac{dQ_1}{dQ_0}(\cA) \right ] + \E_{\cB \sim Q_0} \left [ \log \frac{dQ_0}{dQ_1} ( \cB) \right] 
\end{equation}
denote the {\it symmetrized Kullblack-Liebler divergence}, and define the rate function 
\begin{equation}
\label{eq:ldp_rate_function}
I(x)  : = \sup\limits_{\lambda \ge 0} \left \{-  \lambda ( {D}(Q_0, Q_1) - x ) \vphantom{\left [ \left( \frac{dQ_0}{dQ_1}(\cA) \right)^\lambda \left( \frac{dQ_1}{dQ_0}(\cB) \right)^\lambda \right]}- \log \E \left [ \left( \frac{dQ_0}{dQ_1}(\cA) \right)^\lambda \left( \frac{dQ_1}{dQ_0}(\cB) \right)^\lambda \right] \right \},
\end{equation}
where $\cA \sim Q_1$ and $\cB \sim Q_0$ are independent. Finally, for $u,v \in V$, define the {\it neighborhood difference function}
$$
f_{vu}(t) : = \sum\limits_{s = 0}^t | \cN_v(s) \setminus \cN_u(s) |. 
$$

\begin{lemma}
\label{lemma:llr_chernoff}
Let $u,v \in V$ be any two vertices. If $x > 0$, then $I(x) > 0$ and  
\begin{equation}
\label{eq:llr_chernoff}
\p_v \left ( \frac{\pi_v(t)}{\pi_u(t)} \le e^{( {D}(Q_0, Q_1) - x) f_{vu}(t)} \right ) \le e^{ - I(x) f_{vu}(t)}, \hspace{0.25cm} x > 0.
\end{equation}
\end{lemma}

\begin{proof}
Conditioned on $\cA, \cB$, define the function 
$$
g(\lambda) : = \left( \frac{dQ_0}{dQ_1}(\cA) \frac{dQ_1}{dQ_0}(\cB) \right)^\lambda.
$$
Note that $g(\lambda)$ is differentiable, and by the mean-value theorem, we have for any $\lambda_1 \in [0,1]$ that there exists $\eta \in [0, \lambda_1]$ so that
\begin{align}
\frac{|g(\lambda_1) - 1|}{\lambda_1}  & = |g'(\eta)|  = \eta \left| \frac{dQ_0}{dQ_1}(\cA) \frac{dQ_1}{dQ_0}(\cB) \right|^{\eta- 1} \nonumber \\
\label{eq:ldct_1}
& \le  1 + \left | \frac{dQ_0}{dQ_1}(\cA) \frac{dQ_1}{dQ_0} (\cB) \right| .
\end{align}
In the final inequality, we have used $\eta \le 1$ as well as the fact that $a^{\eta - 1} \le a$ if $a \ge 1$ else $a^{\eta - 1} \le 1$ if $0 \le a \le 1$. Furthermore, note that due to the independence of $\cA$ and $\cB$, 
\begin{equation*}
\E_{\cA \sim Q_1, \cB \sim Q_0} \left [ \frac{dQ_0}{dQ_1}(\cA) \frac{dQ_1}{dQ_0}(\cB) \right] \\
= \E_{\cA \sim Q_1} \left[ \frac{dQ_0}{dQ_1}(\cA) \right] \E_{\cB \sim Q_0} \left[ \frac{dQ_1}{dQ_0}(\cB) \right] = 1.
\end{equation*}
It then follows from the definition of the Lebesgue integral that 
\begin{equation}
\label{eq:ldct_2}
\E \left [ \left| \frac{dQ_0}{dQ_1}(\cA) \frac{dQ_1}{dQ_0}(\cB) \right| \right] < \infty.
\end{equation}
In particular, the dominating function for $|g(\lambda_1) - 1|/\lambda_1$ in \eqref{eq:ldct_1} is integrable. Together, \eqref{eq:ldct_1} and \eqref{eq:ldct_2} along with the Lebesgue Dominated Convergence Theorem imply that $\E[ g(\lambda)]$ is differentiable at $\lambda = 0$. Moreover, the derivative of $\log \E [ g(\lambda)]$ at $\lambda = 0$ is equal to $-D(Q_0, Q_1)$. Next, if we define the function 
$$
h(\lambda) : = - \lambda( D(Q_0, Q_1) - x) - \log \E [ g(\lambda)],
$$
then the results we have established thus far imply $h'(0) = x > 0$. Since $h(0) = 0$, it follows that $h(\lambda) / \lambda > 0$ for sufficiently small $\lambda$, and in particular $h(\lambda) > 0$. Since $I(x) = \sup_{\lambda \ge 0} h(\lambda)$, the claim $I(x) > 0$ follows. 

We now show how one can use the rate function $I(x)$ to obtain the inequality \eqref{eq:llr_chernoff}. Recall that under the measure $\p_v$, the variables $\{y_w(s) \}_{w \in V, 0 \le s \le t}$ are independent, with 
$$
y_w(s) \sim \begin{cases}
Q_1 & w \in \cN_v(s) \\
Q_0 & \text{ else}. 
\end{cases}
$$
Using the representation \eqref{eq:ratio_pi_v1}, the following distributional identity holds under the measure $\p_v$:
\begin{equation}
\label{eq:dist_rep_1}
\frac{ \pi_v(t) }{\pi_u(t) } \stackrel{d}{=} \prod\limits_{i = 1}^{f_{vu}(t)} W_i, 
\end{equation}
where the $W_i$'s are i.i.d. with distribution given by 
\begin{equation}
\label{eq:dist_rep_2}
W_i \stackrel{d}{=} \frac{dQ_1}{dQ_0}(\cA) \left( \frac{dQ_1}{dQ_0}(\cB) \right)^{-1},
\end{equation}
for independent $\cA \sim Q_1$ and $\cB \sim Q_0$. A Chernoff-type bound implies that 
\begin{align*}
\p_v  \left( \frac{\pi_v(t)}{\pi_u(t)} \le e^{({D}(Q_0, Q_1) - x ) f_{vu}(t)} \right) & = \inf\limits_{\lambda \ge 0} \p_v \left( \left( \frac{\pi_v(t)}{\pi_u(t)} \right)^{-\lambda} \ge e^{- \lambda ({D}(Q_0, Q_1) - x ) f_{vu}(t) } \right) \\
& \le \inf\limits_{\lambda \ge 0} \mathrm{exp} \left( \lambda ( \widetilde{D}(Q_0, Q_1) - x ) f_{vu}(t) \vphantom{\sum\limits_{s = 0}^t }  +  \log \E \left [ \left( \frac{\pi_v(t)}{\pi_u(t)} \right)^{-\lambda} \right] \right) \\
& = \mathrm{exp} \left( - I(x) f_{vu}(t) \right),
\end{align*}
where the final expression follows from the distributional representation for $\pi_v(t) / \pi_u(t)$ in \eqref{eq:dist_rep_1} and \eqref{eq:dist_rep_2}. 
\end{proof}

\subsection{Lower bounding the estimation error: Proof of Lemma \ref{lemma:estimation_error_lower_bound}}
\label{subsec:lower_bound}

At a high level, the proof strategy is to first establish a probabilistic lower bound for $\E_{\pi(t)}[ d(v^*, u) ]$ where $u \in V_n$ is {\it fixed}. Through union bounds, this will lead to a probabilistic lower bound for $\min_{u \in V_n} \E_{\pi(t)} [ d(v^*, u) ]$. We remark that the proof of Lemma \ref{lemma:estimation_error_lower_bound} makes use of some combinatorial properties of trees and lattices, the proofs of which may be found in Appendix \ref{sec:geodesics}. 

For a fixed vertex $u \in V_n$, we can write
\begin{align*}
\E_{\pi(t)} [ d(v^*, u)  ] & = \sum\limits_{w \in V_n} d(w,u) \pi_w(t) \\
& = \frac{1}{Y(t) } \sum\limits_{w \in V_n} d(w,u) X_w(t).
\end{align*}
Lemma \ref{lemma:Y} has already established a concentration inequality for $Y(t)$, so we will proceed by establishing a probabilistic lower bound for $\sum_{w \in V_n} d(u, w) X_w(t)$. For any $v \in V_n$, we have
\begin{align}
\E_v \left [ \sum\limits_{w \in V_n} d(u,w) X_w(t) \right] & = \sum\limits_{w \in V_n} d(u,w) \E_v [ X_w(t) ] \nonumber \\
\label{eq:sum_expectation_lower_bound}
& \ge \sum\limits_{w \in V_n} d(u,w),
\end{align}
where the inequality is due to $\E_v [ X_w(t) ] \ge 1$, which was proved in Lemma \ref{lemma:xu_mean}. We can also upper bound the variance as 
\begin{align}
 \mathrm{Var}_v \left( \sum\limits_{w \in V_n} d(u,w) X_w(t) \right) & = \sum\limits_{w_1 \in V_n} \sum\limits_{w_2 \in V_n} d(u, w_1) d(u, w_2) \mathrm{Cov}_v \left( X_{w_1}(t), X_{w_2}(t) \right) \nonumber  \\
& \stackrel{(a)}{\le} \lambda^{f(t)} \sum\limits_{w_1 \in V_n} \sum\limits_{w_2 \in V_n: d(w_1, w_2) \le 2t} d(u, w_1) d(u, w_2) \nonumber \\
& \stackrel{(b)}{\le} \lambda^{f(t)} | \cN(2t)| \sum\limits_{w_1 \in V_n} d(u, w_1) ( d(u, w_1) + 2t) \nonumber \\
& = \lambda^{f(t)} | \cN(2t) | \left( \sum\limits_{w \in V_n} d(u, w)^2 + 2t \sum\limits_{w \in V_n} d(u,w) \right) \nonumber \\
\label{eq:sum_var_upper_bound}
& \stackrel{(c)}{\le} \lambda^{f(t)} | \cN(2t)| (1 + 2t) \sum\limits_{w \in V_n} d(u, w)^2,
\end{align}
where $(a)$ follows from Lemma \ref{lemma:xu_cov}, $(b)$ is due to the inequality $d(u, w_2) \le d(u, w_1) + d(w_1, w_2) \le d(u, w_1) + 2t$, and $(c)$ follows from bounding $d(u,w) \le d(u,w)^2$. 

Next, an application of inequality \eqref{eq:sum_expectation_lower_bound} and Chebyshev's inequality yields 
\begin{equation}
\label{eq:lemma1_chebyshev}
\p_v \left( \sum\limits_{v \in V_n} d(u,w) X_w(t) \le \frac{1}{2} \sum\limits_{w \in V_n} d(u,w) \right) \\
 \le 4 \lambda^{f(t)} | \cN(2t) | (1 + 2t) \cdot \frac{ \sum_{w \in V_n} d(u,w)^2 }{ \left( \sum_{w \in V_n} d(u,w) \right)^2 },
\end{equation}
where the right hand side uses the variance upper bound \eqref{eq:sum_var_upper_bound}. To proceed, we bound the right hand side of \eqref{eq:lemma1_chebyshev} when $G$ is a regular tree or a lattice. Although we treat these cases separately for convenience, the methodology is the same. \\
\\
\noindent {\bf Case 1: $G$ is a $k$-regular tree.} \\
In this case, Lemma \ref{lemma:tree_distance_bounds} shows that
\begin{align*}
\sum\limits_{w \in V_n} d(u,w) & \ge \frac{n \log n}{k \log(k - 1)} \\
\sum\limits_{w \in V_n} d(u,w)^2 & \le \frac{4n \log^2 n}{\log^2 (k - 1)}. 
\end{align*}
Substituting the above bounds into \eqref{eq:lemma1_chebyshev} shows that
\begin{align}
\p_v \left( \sum\limits_{w \in V_n} d(u,w) X_w(t) \le \frac{n \log n}{2k \log (k - 1)} \right) & \le  \lambda^{f(t)} | \cN(2t)| (1 + 2t) \frac{16k^2}{n} \nonumber \\
\label{eq:lemma1_trees_bound}
& \le \frac{50 t \lambda^{f(t)}| \cN(2t)| k^2}{n}.
\end{align}
A na\"{i}ve method for deriving a probabilistic bound for $\min_{u \in V_n} \sum_{w \in V_n} d(u,w) X_w(t)$ is to take a union bound over the events pertaining to $\sum_{w \in V_n} d(u,w) X_w(t)$ for all $u \in V_n$. However, the probability bound in \eqref{eq:lemma1_trees_bound} is not small enough for a union bound to work, since $|V_n | = n$. Fortunately, as we shall see, it suffices to take a union bound over a much smaller set of vertices. To this end, define 
$$
m : = \left \lceil r_n + 1 - \frac{\log n}{6k \log(k - 1)} \right \rceil.
$$
Since $r_n \sim \log (n) / \log (k - 1)$ (see \eqref{eq:rn_asymptotics}), $m \sim \left( 1 - \frac{1}{6k} \right) r_n$. In addition, it holds for $n$ sufficiently large that 
\begin{align*}
| \cN_{v_0}(m) | & \stackrel{(d)}{=} 1 + \frac{k}{k - 2} \left( (k - 1)^m - 1 \right) \\
& \stackrel{(e)}{\le} 1 + \frac{k}{k-2} ( k - 1)^{ \left( 1 - \frac{1}{12k} \right) \frac{\log n}{\log (k - 1)} } \\
& \stackrel{(f)}{\le} 2k n^{1 - \frac{1}{12k}},
\end{align*}
where $(d)$ follows from Lemma \ref{lemma:trees_exact}, $(e)$ holds since $m \le \left( 1 - \frac{1}{12k} \right) \frac{\log n}{\log (k - 1)}$ for $n$ sufficiently large due to the asymptotics of $r_n$, and $(f)$ follows from upper bounding the coefficient on the first-order term in the second line. Combining \eqref{eq:lemma1_trees_bound} with a union bound over elements of $\cN_{v_0}(m)$ implies
\begin{equation}
\label{eq:tree_restricted_set_bd}
\p_v \left( \min\limits_{u \in \cN_{v_0}(m)} \sum\limits_{w \in V_n} d(u,w) X_w(t) \le \frac{n \log n}{2k \log (k - 1)} \right) \le 100t \lambda^{f(t)} | \cN(2t)| k^3 n^{- \frac{1}{12k}}.
\end{equation}
Next, define the event 
$$
\cE : = \left \{ \frac{n}{2} \le Y(t) \le \frac{3n}{2} \right \}.
$$
If the event $\cE$ holds, we have the series of implications
\begin{align*}
 \min\limits_{u \in \cN_{v_0}(m)} \sum\limits_{w \in V_n} d(u,w) \pi_w(t) \le \frac{\log n}{3k \log (k-1)} & \Rightarrow \min\limits_{u \in \cN_{v_0}(m)} \frac{1}{3n/2} \sum\limits_{w \in V_n} d(u,w) X_w(t) \le \frac{\log n}{3k \log (k-1)} \\
& \Rightarrow \min\limits_{u \in \cN_{v_0}(m)} \sum\limits_{w \in V_n} d(u,w) X_w(t) \le \frac{n \log n}{2k \log (k-1)}.
\end{align*}
Above, the first implication uses $\pi_w(t) = X_w(t) / Y(t)$ and the fact that $Y(t) \le 3n/2$ on $\cE$. We then have, for $t \le F \left( \frac{\log n}{4 \log \max \{\beta, \lambda \}} \right)$, 
\begin{align}
 \p_v \left( \min\limits_{u \in \cN_{v_0}(m) } \sum\limits_{w \in V_n} d(u,w) \pi_w(t) \le \frac{\log n}{3k \log (k - 1) } \right) & \le \p_v \left( \min\limits_{u \in \cN_{v_0}(m)} \sum\limits_{w \in V_n} d(u,w) X_w(t) \le \frac{n\log n}{2k \log (k - 1) } \right) \nonumber \\
& \hspace{1cm} + \p_v(\cE^c) \nonumber \\
\label{eq:lemma1_trees_final_bound}
& \le 100t \lambda^{f(t)} | \cN(2t)| k^3 n^{- \frac{1}{12k}} + \frac{16}{\sqrt{n}}.
\end{align}
The final inequality above follows from Lemma \ref{lemma:Y} as well as the bound in \eqref{eq:tree_restricted_set_bd}. If we additionally have $t \le F \left( \frac{\log n}{24k \log \lambda} \right)$, we have the bounds
\begin{align*}
t & \le (1 + o_n(1)) \frac{\log \log n}{\log (k - 1)} \\
\lambda^{f(t)} &  \le n^{\frac{1}{24k}}, \\
| \cN(2t)| & \le \log^{2 + o_n(1)} n,
\end{align*}
where $o_n(1) \to 0$ as $n \to \infty$. This shows in particular that the final expression in \eqref{eq:lemma1_trees_final_bound} can be bounded for $n$ sufficiently large by 
$$
\frac{16}{\sqrt{n}} + \frac{200k^3 (\log \log n) (\log^3 n)}{\log (k - 1)} n^{- \frac{1}{24k}} \le n^{-\frac{1}{30k}}.
$$
To put everything together, the way we have defined $m$ implies that for every $u \in V_n$ there is $u' \in \cN_{v_0}(m)$ such that $d(u, u') \le \frac{\log n}{6k \log (k - 1)}$. Moreover,
\begin{align*}
\sum\limits_{w \in V_n} d(u,w) \pi_w(t) & \ge \sum\limits_{w \in V_n} (d(u', w) - d(u, u')) \pi_w(t) \\
& \ge \sum\limits_{w \in V_n} d(u', w)\pi_w(t) - \frac{\log n}{6k \log (k - 1)}.
\end{align*}
We therefore have 
\begin{align*}
\min\limits_{u' \in \cN_{v_0}(m)} \sum\limits_{w \in V_n} d(u',w) \pi_w(t) > \frac{\log n}{3k \log (k - 1)}  & \Rightarrow \min\limits_{u \in V_n} \sum\limits_{w \in V_n} d(u,w) \pi_w(t) > \frac{\log n}{6 k \log (k- 1)}.
\end{align*}
To summarize, we have shown that if we set 
$$
a_1 := \left( \max\{ 4 \log \beta, 24k \log \lambda \} \right)^{-1},
$$
then for $t \le F(a_1 \log n)$, 
$$
\p_v \left ( \min\limits_{u \in V_n} \sum\limits_{w \in V_n} d(u,w) \pi_w(t) \le \frac{\log n}{6 k \log (k - 1)} \right) \le n^{- \frac{1}{30k}}.
$$
Taking a union bound over $0 \le t \le F (a_1 \log n)$, and recalling that $\E_{\pi(t)} [ d(v^*, u)] = \sum_{w \in V_n} d(u,w) \pi_w(t)$, we arrive at
\begin{align*}
\p_v & \left ( \max\limits_{0 \le t \le F(a_1 \log n)} \E_{\pi(t)} [ d(v^*, \widehat{v}_B(t)) ] \le \frac{\log n}{6 k \log (k - 1)} \right)  \le F(a_1 \log n) n^{- \frac{1}{30 k}}.
\end{align*}
Since $F(a_1 \log n) \sim \frac{\log \log n}{\log (k - 1)}$, the right hand side tends to 0 uniformly over $v \in V_n$ as $n \to \infty$, and the result of Lemma \ref{lemma:estimation_error_lower_bound} for trees follows.  \\
\\
{\bf Case 2: $G$ is a $\ell$-dimensional lattice.}\\
By Lemma \ref{lemma:geodesic_lattice}, there exist constants $c_3, c_4 > 0$ depending only on $\ell$ such that 
\begin{align*}
\sum\limits_{w \in V_n} d(u,w) & \ge c_3 n^{1 + \frac{1}{\ell} } \\
 \sum\limits_{w \in V_n} d(u,w)^2 & \le c_4 n^{1 + \frac{2}{\ell} }.
\end{align*}
Substituting the above bounds into \eqref{eq:lemma1_chebyshev} shows that
\begin{align*}
\p_v  \left( \sum\limits_{w \in V_n} d(u,w) X_w(t) \le \frac{c_3}{2} n^{1 + \frac{1}{\ell} } \right) & \le 4 \lambda^{f(t)} | \cN(2t)| (1 + 2t) \frac{c_4 n^{1 + 2/\ell}}{c_3^2 n^{2 + 2/\ell}} \\
& \le \left( \frac{12 c_4}{c_3} \right) \frac{ t \lambda^{f(t)} | \cN(2t)| }{n}.
\end{align*}
As in the case of regular trees, we need to take a union bound over a small set of vertices. The following combinatorial lemma guarantees the existence of such a set. 

\begin{lemma}
\label{lemma:lattice_covering}
For every $n$, there exists a set $S_n$ whose size can be bounded as a function of $\ell$ only, such that for every $u \in V_n$, there exists $u' \in S_n$ such that $d(u,u') \le \frac{c_3}{6} n^{1/\ell}$. 
\end{lemma}

Before proving the lemma, we shall show how we can use it to prove Lemma \ref{lemma:estimation_error_lower_bound} for lattices. Mirroring the steps of \eqref{eq:lemma1_trees_final_bound} in the case of lattices, if $t \le F\left( \frac{\log n}{4 \log \max \{\beta, \lambda \}} \right)$ we arrive at the probability bound
\begin{equation}
\label{eq:lemma1_lattices_final_bound}
\p_v \left( \min\limits_{u \in S_n} \sum\limits_{w \in V_n} d(u,w) \pi_w(t) \le \frac{c_3}{3} n^{1/\ell} \right) \\
\le \frac{16}{\sqrt{n}} + \frac{12c_4}{c_3^2} | S_n| \cdot \frac{t \lambda^{f(t)} | \cN(2t)|}{n}.
\end{equation}
If we additionally have $t \le F \left( \frac{\log n}{3 \log \lambda} \right)$, we have the bounds
\begin{align*}
t & \le O \left( (\log n)^{\frac{1}{\ell + 1}} \right) \\
\lambda^{f(t)} & \le n^{1/3} \\
| \cN(2t) | & \le O \left( (\log n)^{\frac{\ell}{\ell + 1}} \right).
\end{align*}
The big-$O$ bounds follow from the asymptotic behavior of $F$ (see \eqref{eq:F_asymptotics}) as well as the asymptotic behavior of $| \cN(t) |$ (see \eqref{eq:lattice_N}). For $n$ sufficiently large, we can therefore bound the right hand side in \eqref{eq:lemma1_lattices_final_bound} by 
\begin{equation}
\label{eq:lemma1_lattices_prob_bound}
\frac{16}{\sqrt{n}} + O((\log n) n^{-2/3} ) \le \frac{20}{\sqrt{n}},
\end{equation}
where the final inequality holds for $n$ sufficiently large. Putting everything together, since for every $u \in V_n$ we can find $u' \in S_n$ such that $d(u, u') \ge \frac{c_3}{6} n^{1/\ell}$, we have
\begin{align*}
& \min\limits_{u' \in S_n} \sum\limits_{w \in V_n} d(u', w) \pi_w(t)  > \frac{c_3}{3} n^{1/\ell}  \Rightarrow \min\limits_{u \in V_n} \sum\limits_{w \in V_n} d(u,w) \pi_w(t) > \frac{c_3}{6} n^{1/\ell}.
\end{align*}
Hence, if we set 
$$
a_1' := \left( \max\{ 4 \log \beta, 4 \log \lambda \} \right)^{-1},
$$
then \eqref{eq:lemma1_lattices_final_bound} and \eqref{eq:lemma1_lattices_prob_bound} imply that for $t \le F(a_1' \log n)$, 
$$
\p_v \left( \min\limits_{u \in V_n} \sum\limits_{w \in V_n} d(u,w) \pi_w(t) \le \frac{c_3}{6} n^{1/\ell} \right) \le \frac{20}{\sqrt{n}}.
$$
Taking a union bound over $0 \le t \le F(a_1' \log n)$ and recalling the definition of $\E_{\pi(t)} [ d(v^*, u)]$, we arrive at
\begin{align*}
\p_{v} & \left( \max\limits_{0 \le t \le F(a_1' \log n) } \E_{\pi(t)} [ d(v^*, \widehat{v}_B(t)) ] \le \frac{c_3}{6} n^{1/\ell} \right) \le \frac{20 F(a_1' \log n)}{\sqrt{n}}.
\end{align*}
Due to the asymptotic behavior of $F$ (see \eqref{eq:F_asymptotics}), the right hand side tends to 0 uniformly over $v \in V_n$ as $n \to \infty$, and the result of Lemma \ref{lemma:estimation_error_lower_bound} for lattices follows.

We now turn to the proof of Lemma \ref{lemma:lattice_covering}. 

\begin{proof}[Proof of Lemma \ref{lemma:lattice_covering}]
Let $m$ be a positive integer. We define an $m$-covering of $V_n$ to be $S \subseteq V_n$ such that for any $u \in V_n$, there exists $v \in S$ such that $d(u,v) \le m$. We also define an $m$-packing of $V_n$ to be $S \subseteq V_n$ such that for any $u, v \in S$, $d(u,v) > m$. We say a $S$ is a {\it maximal} $m$-packing of $V_n$ if it has the maximum possible cardinality. A fundamental result on coverings and packings is that a maximal $m$-packing is also a valid covering \cite[Lemma 5.12]{ramon_notes}, so the proof focuses on bounding the size of a maximal packing. Our proof is based on \cite[Lemma 5.13]{ramon_notes}. 

Set $m : = \lfloor \frac{c_3}{8} n^{1/\ell} \rfloor$ and let $S \subset V_n$ be a $2m$-packing. This in particular implies that $\{\cN_u(m) \}_{u \in S}$ is a collection of disjoint sets satisfying 
$$
\bigcup\limits_{u \in S} \cN_u(m) \subseteq \bigcup\limits_{u \in V_n} \cN_u(m) \subseteq \cN_{v_0}(r_n + 1 + m).
$$
This in turn implies
\begin{equation}
\label{eq:S_inequality1}
\sum\limits_{u \in S} | \cN_u(m) | = |S| \cdot | \cN(m) | \le | \cN(r_n + 1 + m) |.
\end{equation}
Since $r_n \asymp n^{1/\ell}$ (see \eqref{eq:rn_asymptotics}) and $m \asymp n^{1/\ell}$, we can find constants $C_1, C_2 > 0$ depending only on $\ell$ such that 
$$
C_1 n^{1/\ell} \le m \le r_n + 1 + m \le C_2 n^{1/\ell}.
$$
Next, recall that $|\cN(t) | \sim c_\ell t^\ell$, where $c_\ell$ is a constant depending only on $\ell$ (see \eqref{eq:lattice_N}). Hence 
$$
|S| \le \frac{ | \cN(r_n + 1 + m)|}{ | \cN(m) | } \sim \frac{ (r_n + 1 + m)^\ell}{m^\ell} \le \left( \frac{C_2}{C_1} \right)^\ell.
$$
Note that the right hand side is of constant order even as $n \to \infty$. Moreover, the bound holds for {\it all} $2m$-packings, including {\it maximal} packings that are also coverings. This guarantees the existence of a $2m$-covering of size bounded by a constant depending on $\ell$ even as $n \to \infty$. We conclude by noting that $2m \le \frac{c_3}{4} n^{1/\ell}$.
\end{proof}

\subsection{Upper bounding the estimation error: Proof of Lemma \ref{lemma:estimation_error_upper_bound}}
\label{subsec:upper_bound}

For any vertex $v \in V_n$, we can write
\begin{align*}
\E_{\pi(t)} [ d(v^*, \widehat{v}_B(t))] & \stackrel{(a)}{\le} \E_{\pi(t)} [ d(v^*, v) ] = \sum\limits_{w \in V_n} d(w,v) \pi_w(t) \\
& \stackrel{(b)}{\le} \sum\limits_{w \in V_n} d(w,v) \frac{X_w(t)}{X_v(t)},
\end{align*}
where $(a)$ follows since $\E_{\pi(t)} [ d(v^*, \widehat{v}_B(t) ) ] = \min_{u \in V_n} \E_{\pi(t)} [ d(v^*, u)  ]$ and $(b)$ follows since $\pi_w(t) = X_w(t) / Y(t)$ and $Y(t) \ge X_v(t)$. For distinct vertices $w,v \in V$, recall the notation $f_{vw}(t) : = \sum_{s = 0}^t | \cN_v(s) \setminus \cN_w(s) |$ and further recall that ${D}(Q_0, Q_1)$ is the symmetrized Kullback-Liebler divergence between $Q_0$ and $Q_1$ (see \eqref{eq:symmetric_KL}). As a shorthand, denote $\theta : = {D}(Q_0, Q_1)/2$. Next, define the event
$$
\cE_{vw} : = \left \{ \frac{X_v(t)}{X_w(t)} \ge e^{\theta f_{vw}(t)} \right \}.
$$
By Lemma \ref{lemma:llr_chernoff}, $\p_v( \cE_{vw}^c) \le e^{- I (\theta) f_{vw}(t)}$, where $I(\cdot)$ is the large-deviations rate function defined in Lemma \ref{lemma:llr_chernoff}. On the event $\cE_v : = \bigcup_{w \in V_n \setminus \{v \}} \cE_{vw}$, we have the bound
\begin{align}
\E_{\pi(t)} [ d(v^*, \widehat{v}_B(t)) ] & \le \sum\limits_{w \in V_n} d(w,v) \frac{X_w(t)}{X_v(t)} \nonumber  \\
\label{eq:risk_bound}
& \le \sum\limits_{w \in V_n} d(w,v) e^{- \theta f_{vw}(t) }.
\end{align}
To bound the final summation in \eqref{eq:risk_bound}, we split the summation into two parts: $w$ such that $d(v,w) \le 2t$ and $w$ such that $d(v,w) > 2t$. To handle the first part, it is useful to define the function 
$$
f_1(t) : = \sum\limits_{s = 0}^t | \cN_a(s) \setminus \cN_b(s) |
$$
where $a,b$ are any two neighboring vertices (since the graph is vertex-transitive, we obtain the same formula for any two neighboring $a,b$). We may now bound the summation over $w$ such that $d(v,w) \le 2t$ as 
\begin{equation}
\label{eq:lemma2_summation1}
\sum\limits_{w \in V_n : d(w,v) \le 2t} d(w,v) e^{ - \theta f_{vw}(t) } \le 2t | \cN(2t) | e^{- \theta f_1(t)}.
\end{equation}
To handle the summation over $w$ such that $d(v,w) > 2t$, first note that 
\begin{equation}
\label{eq:fvw_f_equivalence}
f_{vw}(t) = \sum\limits_{s = 0}^t | \cN_v(s) \setminus \cN_w(s) | \stackrel{(c)}{=} \sum\limits_{s = 0}^t | \cN_v(s) | = f(t),
\end{equation}
where the equality $(c)$ follows since $\cN_v(s) \cap \cN_w(s) = \emptyset$ for $0 \le s \le t$ because $d(v,w) > 2t$. Hence we can bound the second part of the summation by 
\begin{align}
\label{eq:lemma2_summation2}
\sum\limits_{w \in V_n: d(w,v) > 2t} d(w,v) e^{ - \theta f(t) }  & \le \sum\limits_{w \in V_n : d(w,v) > 2t} n e^{ - \theta f(t) } \nonumber \\
& \le  n^2 e^{- \theta f(t) },
\end{align}
where we have used the coarse bound $d(w,v) \le n$ above. Putting everything together, the total bound on the estimation error on the event $\cE_v$ is 
$$
2t | \cN(2t) | e^{ - \theta f_1(t) } + n^2 e^{- \theta f(t) }.
$$
The remaining element of the proof is to bound $\p_v(\cE_v^c)$. We can write 
\begin{align}
\p_v ( \cE_v^c) & \stackrel{(d)}{\le} \sum\limits_{w \in V_n \setminus \{v \}} \p_v ( \cE_{vw}^c ) \nonumber \\
& \stackrel{(e)}{\le} \sum\limits_{w \in V_n \setminus \{v \}} e^{- I(\theta) f_{vw}(t) } \nonumber \\
& \stackrel{(f)}{\le} \sum\limits_{w \in V_n: d(w,v) \le 2t} e^{ - I(\theta) f_1(t) } + \sum\limits_{w \in V_n : d(w,v) > 2t} e^{ - I(\theta) f(t) } \nonumber \\
\label{eq:lemma2_prob_bound}
& \le | \cN(2t) | e^{ - I(\theta) f_1(t) } + n e^{ - I(\theta) f(t) }.
\end{align}
Above, $(d)$ is due to a union bound, $(e)$ follows from Lemma \ref{lemma:llr_chernoff}, and $(f)$ uses $f_1(t) \le f_{vw}(t)$ as well as $f_{vw}(t) = f(t)$ if $d(v,w) > 2t$ (see \eqref{eq:fvw_f_equivalence}). Putting everything together, we have shown that for any $v \in V_n$, 
\begin{align*}
\p_v & \left( \E_{\pi(t)} [ d(v^*, \widehat{v}_B(t)) ] > 2t | \cN(2t) | e^{- \theta f_1(t)} + n^2 e^{-\theta f(t)} \right) \\
& \le | \cN(2t) | e^{-I(\theta) f_1(t) } + ne^{- I(\theta) f(t) }.
\end{align*}
Note in particular that the probability bound above holds uniformly over all $v \in V_n$. Focusing on the special cases of regular trees and lattices, we will simplify the bound on the estimation error as well as the probability bound. \\
\\
\noindent {\bf Case 1: $G$ is a $k$-regular tree.} Lemma \ref{lemma:trees_exact} provides the asymptotics of various combinatorial quantities related to neighborhood sizes, summarized below: 
\begin{align*}
| \cN(2t) | & \sim \frac{k}{k-2} (k - 1)^{2t} \\
f_1(t) &  \sim \frac{(k-1)^{t + 1}}{k - 2} \\
f(t) & \sim \frac{k}{(k - 2)^2} (k - 1)^{t + 1} \\
F(z) & \sim \frac{\log z}{\log (k - 1) }.
\end{align*}
The terms $2t | \cN(2t) | e^{ - \theta f_1(t) }$ and $| \cN(2t) | e^{- I(\theta) f_1(t) }$ can therefore be bounded by $e^{- O( (k - 1)^t ) }$ for $t$ sufficiently large. If $t \ge F\left( \frac{4 \log n}{\theta} \right)$, then $n^2 e^{ - \theta f(t) } \le e^{ - \frac{1}{2} \theta f(t) } = e^{ - O ( (k - 1)^t ) }$, where the hidden factors in the big $O$ do not depend on $n$. Similarly, if $t \ge F \left( \frac{2 \log n}{I(\theta) } \right)$, then $n e^{- I(\theta) f(t) } \le e^{ - O ( (k- 1)^t ) }$, where again, the hidden factors do not depend on $n$. 

Putting everything together, we have shown that for $t \ge F \left( \frac{4 \log n}{\min \{\theta, I(\theta) \}} \right)$, 
$$
\p_{v} \left( \E_{\pi(t)} [ d(v^*, \widehat{v}_B(t)) ] > e^{- O( (k - 1)^t ) } \right) \le e^{- O ( (k - 1)^t ) },
$$
which implies the desired result for $k$-regular trees. \\
\\
\noindent {\bf Case 2: $G$ is a $\ell$-dimensional lattice.} Lemma \ref{lemma:lattices_exact} proves the following asymptotic behavior of neighborhood sizes in $\ell$-dimensional lattices:
$$
| \cN(2t) | \asymp t^\ell, \hspace{1cm} f(t) \asymp t^{\ell + 1}, \hspace{1cm} F(z) \asymp z^{\frac{1}{\ell + 1}}.
$$
Additionally, we can lower bound $f_1(t)$ as follows: if $u,v$ are adjacent, then
\begin{equation}
\label{eq:f1_lower_bound}
f_1(t) = \sum\limits_{s = 0}^t | \cN_v(s) \setminus \cN_u(s) | \ge t + 1,
\end{equation}
where the inequality above uses the fact that $\cN_v(s) \setminus \cN_u(s) \neq \emptyset$ for all $s \ge 0$. We can therefore bound the terms
\begin{align*}
2t | \cN(2t) | e^{ - \theta f_1(t) } & \le O \left( t^{\ell + 1} e^{- \theta t} \right) \\
| \cN(2t) | e^{- I(\theta) f_1(t) } & \le O \left( t^\ell e^{- I(\theta) t} \right).
\end{align*}
Additionally, as in the previous case, if $t \ge F \left( \frac{4 \log n}{ \min \{\theta, I(\theta)  \}} \right)$ then $n^2 e^{- \theta f(t) }$ and $n e^{- I(\theta) f(t) }$ are bounded by $e^{- O(t^{\ell + 1} ) }$, where the hidden factors in the big $O$ do not depend on $n$. Putting everything together, we have shown that for $t \ge F\left( \frac{4 \log n}{\min \{\theta, I(\theta) \}} \right)$, 
$$
\p_{\pi(V_n)} \left( \E_{\pi(t)} [ d(v^*, \widehat{v}_B(t) ) ] > O \left( t^{\ell + 1} e^{- \theta t} \right) \right) \le O \left( t^\ell e^{- I(\theta) t} \right),
$$
which implies the desired result for $n$ sufficiently large if we set $b_2 = 0.5\cdot \min \{ \theta, I(\theta) \}$.

\section{Proof of the MSPRT upper bounds}
\label{sec:msprt}

\subsection{Useful preliminary results}

We start by stating and recalling some useful combinatorial results concerning the sizes of neighborhoods in regular trees and lattices. To begin, for vertices $u,v$ recall that
$$
f_{vu}(t) : = \sum\limits_{s = 0}^t | \cN_v(s) \setminus \cN_u(s) |.
$$
Moreover, recall that the {\it neighborhood growth function} (originally defined in \eqref{eq:neighborhood_growth_fn}) is 
$$
f(t) : = \sum\limits_{s = 0}^t | \cN(s) |.
$$
We also define, for any pair of adjacent vertices $u,v$, the function
$$
f_1(t) : = \sum\limits_{s = 0}^t | \cN_v(s) \setminus \cN_u(s) |.
$$
We also define the inverse functions $F_{vu} = f_{vu}^{-1}, F = f^{-1}, F_1 = f_1^{-1}$. These inverse functions are well-defined since $f_{vu}, f, f_1$ are strictly increasing functions. In $k$-regular trees, we have the asymptotics
\begin{equation}
\label{eq:msprt_F1_trees}
F_1(z) \sim \frac{\log z}{\log (k - 1)}.
\end{equation}
For a proof, see Lemma \ref{lemma:trees_exact}. In $\ell$-dimensional lattices, we have the {\it orderwise} asymptotics
\begin{equation}
\label{eq:msprt_asymptotics_lattices}
| \cN(t) | \asymp t^\ell, \qquad f(t) \asymp t^{\ell + 1}, \qquad F(z) \asymp z^{\frac{1}{\ell + 1}}.
\end{equation}
For a more precise statement, see Lemma \ref{lemma:lattices_exact}. Next, we prove a few simple, generic results regarding these functions. The following result provides a simple but useful bounds for $f_1$ that hold in regular trees and lattices. 

\begin{lemma}
\label{lemma:f1_linear_bound}
Suppose that $G$ is an infinite regular tree or lattice and $u,v \in V$ are two distinct vertices. If $t_2 \ge t_1 \ge 0$, $f_{vu}(t_2) - f_{vu}(t_1) \ge t_2 - t_1$. 
\end{lemma}

\begin{proof}
For any two distinct vertices $u,v$ and any non-negative integer $s$,  $|\cN_v(s) \setminus \cN_u(s) | \ge 1$. Hence 
$$
f_{vu}(t_2) - f_{vu}(t_1) = \sum\limits_{s = t_1 + 1}^{t_2} | \cN_v(s) \setminus \cN_u(s) | \ge t_2 - t_1.
$$
\end{proof}

Next, we prove a simple linear upper bound for $F_1$, since its exact expression is challenging to compute in lattices. 

\begin{lemma}
\label{lemma:F1_linear_bound}
Let $G$ be an infinite regular tree or lattice. Then $F_1(z) \le z$. 
\end{lemma}

\begin{proof}
Lemma \ref{lemma:f1_linear_bound} implies the lower bound
$$
f_1(t) = f_1(0) + f_1(t) - f_1(0) \ge t + 1 \ge t.
$$
Setting $z = f_1(t)$, we have $F_1(z) = t$ which implies the desired result. 
\end{proof}

The following lemma derives conditions under which $F_{vu} = F$. 

\begin{lemma}
\label{lemma:F_equivalence}
If $z <  f ( d(u,v)/2)$ then $F(z) = F_{vu}(z)$. 
\end{lemma}

\begin{proof}
For $s < d(u,v)/2$, $\cN_v(s)$ and $\cN_u(s)$ are disjoint so $\cN_v(s) \setminus \cN_u(s) = \cN_v(s)$. Hence for $t < d(u,v)/2$,
$$
f_{vu}(t) = \sum\limits_{s = 0}^t | \cN_v(s) \setminus \cN_u(s) | = \sum\limits_{s = 0}^t | \cN_v(s) | = f(t).
$$
It follows that $F_{vu}(z) = F(z)$ if $F(z) < d(u,v) / 2$. Equivalently, $z < f ( d(u,v) / 2)$ which proves the lemma. 
\end{proof}

Finally, at the core of the analysis is the following large-deviations-type result. It essentially follows as a corollary from the large-deviations result Lemma \ref{lemma:llr_chernoff} which was used in the analysis of the Bayesian setting. 

\begin{lemma}
\label{lemma:llr_Z}
Let $u,v \in V$ be any two vertices. For any $x > 0$, 
$$
\p_v \left( Z_{vu}(t) \le (D(Q_0, Q_1) - x) f_{vu}(t) \right) \le e^{- I(x) f_{vu}(t) }.
$$
Above, $D(Q_0, Q_1)$ is the symmetrized Kulback-Liebler divergence between $Q_0$ and $Q_1$, and $I(\cdot)$ is the rate function defined in \eqref{eq:ldp_rate_function}. Moreover, $I(x) > 0$ for $x > 0$. 
\end{lemma}

\begin{proof}
Recall from the analysis of the Bayesian setting that for any vertex $v \in V_n$, 
$$
\pi_v(t) : = \p(v^* = v \mid {y}(0), \ldots, {y}(t) ).
$$
Hence, by Bayes' rule, 
$$
\frac{\pi_v(t)}{\pi_u(t)} = \frac{d \p_v}{d \p_u}({y}(0), \ldots, {y}(t)) = e^{Z_{vu}(t)}.
$$
The desired result now follows from a direct application of Lemma \ref{lemma:llr_chernoff}. 
\end{proof}

\subsection{Performance of the MSPRT in trees: Proof of Theorem \ref{thm:trees}}
\label{subsec:proof_msprt_trees}

Our first goal is to establish a probabilistic bound for $T_n(v)$ under the measure $\p_v$. We can write 
\begin{align}
\mathbb{P}_v (T_n(v) > t) & \le \mathbb{P}_v \left( \exists u \in V_n \setminus \{v \} \text{ s.t. } Z_{vu}(t) < \log n^2/\alpha \right) \nonumber \\
\label{eq:tnv_bound_1}
& \le \sum\limits_{u \in V_n \setminus \{v \}} \mathbb{P}_v \left( Z_{vu}(t) < \log n^2/\alpha \right).
\end{align}
Next, define $\theta : = {D}(Q_0, Q_1)/2$ and suppose that $t$ is sufficiently large so that 
\begin{equation}
\label{eq:f1_bound1}
\log \frac{n^2}{\alpha} \le \theta f_1(t).
\end{equation}
Letting $I(\cdot)$ be the large-deviations rate function in Lemma \ref{lemma:llr_Z}, we can upper bound the final summation in \eqref{eq:tnv_bound_1} by 
\begin{align}
\sum\limits_{u \in V_n \setminus \{v \}} \p_v  \left( Z_{vu}(t) < \theta f_1(t) \right) & \stackrel{(a)}{\le} \sum\limits_{u \in V_n \setminus \{v\}} \p_v \left( Z_{vu}(t) < \theta  f_{vu}(t) \right) \nonumber \\
& \stackrel{(b)}{\le} \sum\limits_{u \in V_n \setminus \{v \}} \mathrm{exp} \left( - I( \theta ) f_{vu}(t) \right) \nonumber \\
\label{eq:msprt_trees_bound1}
& \stackrel{(c)}{\le} \mathrm{exp} \left( \log n - I( \theta) f_1(t) \right).
\end{align}
Above, $(a)$ is due to $f_1(t) \le f_{vu}(t)$, $(b)$ follows from Lemma \ref{lemma:llr_Z} and $(c)$ is again due to $f_1(t) \le f_{vu}(t)$ and the observation that the summation is over at most $n$ terms. We now define the quantities
$$
C(Q_0, Q_1) : = \min \{ \theta, I(\theta) \} \hspace{0.2cm} \text{and} \hspace{0.2cm} t_n : = F_1 \left( \frac{\log n^2/\alpha}{ C(Q_0, Q_1)} \right).
$$
In particular, if $t \ge t_n$ then \eqref{eq:f1_bound1} holds and $\log n \le I(\theta) f_1(t)$. The expectation of $T_n(v)$ can then be bounded as
\begin{align}
\E_v [ T_n(v) ] & = \sum\limits_{t = 0}^\infty \p_v (T_n(v) > t) \nonumber \\
& \stackrel{(d)}{\le} t_n + \sum\limits_{t = t_n }^\infty \mathrm{exp} \left( \log n - I(\theta) f_1(t) \right) \nonumber \\
& \stackrel{(e)}{\le} t_n + \mathrm{exp}\left( \log n - I(\theta) f_1(t_n) \right) \sum\limits_{s = 0}^\infty e^{ - I(\theta) s} \nonumber \\
\label{eq:expectation_bound}
& \stackrel{(f)}{\le} t_n + \frac{1}{1 - e^{- I(\theta)}}.
\end{align}
Above, $(d)$ is due to the upper bound \eqref{eq:msprt_trees_bound1} on the probabilities in the summation which holds for $t \ge t_n$, $(e)$ uses $f_1(t) - f_1(t_n) \ge t - t_n$ which was proved in Lemma \ref{lemma:f1_linear_bound} and $(f)$ follows from noting $I(\tilde{d}) f_1(t_n) \ge \log n$ and using the geometric sum formula on the summation. Noting that $T_n \le T_n(v)$, \eqref{eq:expectation_bound} implies
$$
\limsup\limits_{n \to \infty} \frac{ \max_{v \in V_n} \E_v [ T_n] }{ t_n} \le \limsup\limits_{n \to \infty} \frac{ \max_{v \in V_n} \E_v [ T_n(v)] }{t_n} \le 1.
$$
From the asymptotic behavior of $F_1$ (see \eqref{eq:msprt_F1_trees}), it follows that 
$$
t_n \sim \frac{\log \log n}{\log (k - 1)},
$$
which proves the desired result.

\subsection{Performance of the MSPRT in lattices: Proof of Theorem \ref{thm:lattices}}
\label{subsec:proof_msprt_lattices}

As in the proof of Theorem \ref{thm:trees}, we begin by bounding $\p_v(T_n(v) > t)$. A union bound yields
\begin{align}
\label{eq:multiscale_terms}
\p_v (T_n(v) > t) & \le \sum\limits_{\substack{u \in V_n : \\ 0 < d(v,u) \le K }} \p_v \left( Z_{vu}(t) < \log \frac{ 2 K | \cN(K) |}{\alpha} \right) \nonumber \\
& \hspace{1cm} + \sum\limits_{\substack{u \in V_n : \\ d(v,u) > K}} \p_v \left( Z_{vu}(t) < \log \frac{2n^2}{\alpha} \right).
\end{align}
To bound the summations above, we first recall a few quantities. Define $\theta : = \widetilde{D}(Q_0, Q_1)/2$ and $C(Q_0, Q_1) : = \min \{ \theta, I(\theta) \}$, where $I(\cdot)$ is the large-deviations rate function used in Lemma \ref{lemma:llr_Z}. Also define 
\begin{align*}
t_{n,1} & : = F_1 \left( \frac{\log 2K| \cN(K) |/\alpha }{C(Q_0, Q_1)}  \right) \\
t_{n,2}  & : = \max\limits_{u \in V_n : d(u,v) > K } F_{vu} \left( \frac{\log 2n^2/\alpha}{C(Q_0, Q_1)} \right).
\end{align*}
For $t \ge t_{n,1}$, we will make use of the following inequalities to bound the first summation in \eqref{eq:multiscale_terms}:
\begin{equation}
\label{eq:tn1_inequalities}
\log \frac{2 K| \cN(K)|}{\alpha} \le \theta f_1(t) \hspace{0.2cm} \text{and} \hspace{0.2cm} \log | \cN(K) | \le I(\theta) f_1(t).
\end{equation}
 Using the first inequality in \eqref{eq:tn1_inequalities}, we can bound the first summation in \eqref{eq:multiscale_terms} by 
\begin{align}
\sum\limits_{\substack{u \in V_n: \\ 0 < d(v,u) \le K }}  \p_v \left( Z_{vu}(t) \le \theta f_1(t) \right) & \stackrel{(a)}{\le} \sum\limits_{\substack{u \in V_n : \\ 0 < d(v,u) \le K}} \p_v \left( Z_{vu}(t) \le \theta f_{vu}(t) \right) \nonumber \\
& \stackrel{(b)}{\le} \sum\limits_{\substack{u \in V_n : \\ 0 < d(v,u) \le K }} e^{ - I(\theta) f_{vu}(t) } \nonumber \\
\label{eq:first_sum_bound}
& \stackrel{(c)}{\le} \mathrm{exp} \left( \log | \cN(K) | - I(\theta) f_1(t) \right).
\end{align}
Above, $(a)$ and $(c)$ are due to $f_1(t) \le f_{vu}(t)$, and $(b)$ follows from Lemma \ref{lemma:llr_Z}. Similarly, for $t \ge t_{n,2}$ we have, for all $u \in V_n$ satisfying $d(u,v) > K$, 
\begin{equation}
\label{eq:tn2_inequalities}
\log \frac{2n^2}{\alpha} \le \theta f_{vu}(t) \qquad \text{and} \qquad \log n \le I(\theta) f_{vu}(t).
\end{equation}
Using the same reasoning as in \eqref{eq:first_sum_bound}, we have the following bound on the second summation in \eqref{eq:multiscale_terms} for $t \ge t_{n,2}$: 
\begin{equation}
\label{eq:second_sum_bound}
\mathrm{exp} \left( \log n - I(\theta) \min\limits_{u \in V_n : d(u,v) > K} f_{vu}(t) \right).
\end{equation}
Plugging in the bounds \eqref{eq:first_sum_bound} and \eqref{eq:second_sum_bound} into \eqref{eq:multiscale_terms} shows that, for $t \ge \max\{t_{n,1}, t_{n,2} \} = :t_n$, 
\begin{equation}
\label{eq:lattices_final_prob_bound}
\p_v (T_n(v) > t) \le \mathrm{exp} \left( \log | \cN(K) | - I ( \theta) f_1(t) \right) \\
+ \mathrm{exp} \left( \log n - I (\theta) \min\limits_{u \in V_n : d(v,u) > K } f_{vu}(t) \right).
\end{equation}
Next, define the quantities
\begin{align*}
A_n & := \log | \cN(K) | - I(\theta) f_1(t_n) \\
B_n & := \log n - I(\theta) \min_{u \in V_n : d(u,v) > K} f_{vu}(t_n),
\end{align*}
and notice that $A_n, B_n \le 0$ in light of the second inequalities in \eqref{eq:tn1_inequalities} and \eqref{eq:tn2_inequalities}. Using the relation $\E_v [ T_n(v) ] = \sum_{t = 0}^\infty \p_v (T_n(v) > t) \le t_n + \sum_{t = t_n}^\infty \p_v ( T_n(v) > t)$, we have
\begin{align}
\E_v [ T_n(v) ] & \stackrel{(d)}{\le} t_n + \sum\limits_{t = t_n}^\infty e^{ \log | \cN(K) | - I(\theta) f_1(t) } + \sum\limits_{t = t_n}^\infty e^{ \log n - I(\theta) \min_{u \in V_n : d(u,v) > K} f_{vu}(t) } \nonumber \\
& \stackrel{(e)}{\le} t_n + \left( e^{ A_n } + e^{ B_n } \right) \sum\limits_{s = 0}^\infty e^{- I(\theta) s} \nonumber \\
\label{eq:lattices_expectation_bound}
& \stackrel{(f)}{\le} t_n + \frac{1}{1 - e^{- I(\theta)}}. 
\end{align}
Above, $(d)$ is a consequence of \eqref{eq:lattices_final_prob_bound}, $(e)$ is due to the inequality $f_{vu}(t') - f_{vu}(t) \ge t' - t$ which was proved in Lemma \ref{lemma:f1_linear_bound}, and $(f)$ holds since $A_n. B_n \le 0$ and by applying the geometric sum formula. Next, using the inequality $T_n \le T_n(v)$, \eqref{eq:lattices_expectation_bound} implies
$$
\limsup\limits_{n \to \infty} \frac{ \max_{v \in V_n} \E_v [ T_n] }{t_n} \le \limsup\limits_{n \to \infty } \frac{ \max_{v \in V_n} \E_v [ T_n(v)] }{t_n} \le 1. 
$$
It remains to study the asymptotics of $t_n$ as $n$ grows large. From the asymptotic behavior of $|\cN(t) |$ \eqref{eq:msprt_asymptotics_lattices}, we have $| \cN(K) | \asymp K^\ell = \log n$. Hence 
\begin{equation}
\label{eq:tn1_asymptotics}
t_{n,1} \asymp F_1 \left( \frac{\log \log n}{C(Q_0, Q_1)} \right) = O (\log \log n),
\end{equation}
where the final big-$O$ bound is due to the inequality $F_1(z) \le z$, proved in Lemma \ref{lemma:F1_linear_bound}. Next, we establish the asymptotic behavior of $t_{n,2}$. We have, for $n$ sufficiently large, 
\begin{equation}
\label{eq:f(K)}
f \left( \frac{K}{2} \right) \asymp (\log n)^{1 + \frac{1}{\ell}} \ge \frac{\log 2n^2/\alpha}{C(Q_0, Q_1)},
\end{equation}
where the asymptotic behavior of $f(K/2)$ follows from \eqref{eq:msprt_asymptotics_lattices} and the second inequality holds for $n$ sufficiently large. Equation \ref{eq:f(K)} satisfies the condition of Lemma \ref{lemma:F_equivalence}, so we have
\begin{equation}
\label{eq:tn2_asymptotics}
t_{n,2} = F \left( \frac{\log 2n^2/\alpha}{C(Q_0, Q_1)} \right) \asymp (\log n)^{\frac{1}{\ell + 1}}.
\end{equation}
Above, the asymptotic behavior of $F$ follows from \eqref{eq:msprt_asymptotics_lattices}. Hence $t_n \asymp (\log n)^{\frac{1}{\ell + 1}}$, which proves the theorem.

\section{Conclusion and future directions}
\label{sec:conclusion}

In this paper, we considered the problem of quickest estimation of a cascade source from noisy information. We studied a Bayesian and minimax formulation of this problem and derived optimal estimators in the regime of large networks under simple cascade dynamics and network topologies. Furthermore, our results exposed the interplay between the network topology and the performance of optimal estimators. 

There remain several avenues for future work. Although we examined simple networks and cascade dynamics for mathematical tractability, in important next step is to study source estimation for more realistic networks and cascade dynamics \cite{cascades_review, configuration_model, random_geometric_graphs, BA99, Mah92,BA99,BRST01}. We remark that in many cascade models, the cascade evolution is {\it non-deterministic}, hence it will be difficult to compute the the estimators proposed in this paper. We expect that tractable relaxations of the estimators we consider may be more amenable to the analysis of more complex scenarios. 

Another exciting future direction is sampling with {\it incomplete} information. In this work we assumed that all public signals at a given point in time are observable, but when the network is large this may be infeasible. A natural question of interest is to characterize optimal source estimators given that only a budget of $B$ public signals can be observed at any timestep. There are many possibilities for choosing the $B$ signals to observe: one may target potential super-spreaders (i.e., high-degree vertices) or choose vertices adaptively.

\appendix

\section{Bounds on the size of neighborhoods}
\label{sec:size_of_neighborhoods}

We begin by defining and recalling some notation. Given a graph $G$, a vertex $v$, and a non-negative integer $t$, we define 
\begin{align*}
\partial \cN_v(t)  &: = \{ u \in V: d(u,v) = t \} \\
\cN_v(t) & := \{ u \in V: d(u,v) \le t \} .
\end{align*}
Since $G$ is vertex-transitive, $| \partial \cN_v(t)|$ does not depend on $v \in V$; for brevity of notation, we will therefore write $| \partial \cN(t) |$. The same holds for $| \cN_v(t)|$, which we will often write as $| \cN(t) |$. Additionally recall the {\it neighborhood growth function} 
$$
f(t) : = \sum\limits_{s = 0}^t | \cN(s) |
$$
as well as 
$$
f_1(t) : = \sum\limits_{s = 0}^t | \cN_u(s) \setminus \cN_v(s) |,
$$
where $u,v$ are adjacent vertices. As explained earlier, due to the vertex-transitivity of the underlying graph, the formula for $f_1(t)$ is the same for {\it any} pair of adjacent vertices. We also define $F_1 = f_1^{-1}$.

The following result provides exact formulas for $| \partial \cN(t)|, | \cN(t)|$ as well as asymptotic behavior for $f(t), f_1(t), F(z)$ and $F_1(z)$ in regular trees. 

\begin{lemma}
\label{lemma:trees_exact}
Let $G$ be a $k$-regular tree with $k \ge 3$. Then 
\begin{align}
\label{eq:trees_partialN}
| \partial \cN(t) | & = \begin{cases}
1 & t = 0 \\
k(k - 1)^{t - 1} & t \ge 1;
\end{cases} \\
\label{eq:trees_N}
| \cN(t) | & = 1 + \frac{k}{k - 2} \left( (k - 1)^t - 1 \right); \\
\label{eq:trees_f1}
f_1(t) & \sim \frac{ (k - 1)^{t + 1}}{k - 2}; \\
\label{eq:trees_f}
f(t) & \sim   \frac{k }{(k - 2)^2}  (k - 1)^{t + 1}; \\
\label{eq:trees_F1}
F_1(z) & \sim \frac{\log n}{\log (k - 1)} \\
\label{eq:trees_F}
 F(z) & \sim \frac{\log z}{\log (k - 1)}.
\end{align}
\end{lemma}

\begin{proof}
Fix an arbitrary vertex $v$, and suppose we root $G$ at $v$ so that $|\partial \cN(t) |$ is the number of children at height $t$ from the root. Since $\partial \cN_v(0) = \{v\}$ and the root node $v$ is the only vertex with $k$ children while all others have $k - 1$ children, we have the formula $| \partial \cN(t) | = k (k - 1)^{t - 1}$ for $t \ge 1$. 

Next, to compute $| \cN(t)|$, we use the formula for $| \partial \cN(t)|$ and the geometric sum formula:
\begin{align*}
| \cN(t) | & = \sum\limits_{s = 0}^t | \partial \cN(s) |  = 1 + k \sum\limits_{s = 0}^{t-1} (k-1)^s \\
& = 1 + \frac{k}{k - 2} \left( (k - 1)^t - 1 \right).
\end{align*}

The same techniques can be used to derive $f(t)$: 
\begin{align*}
f(t) & = \sum\limits_{s = 0}^t | \cN(s) | \\
& = \sum\limits_{s = 0}^t \left( 1 + \frac{k}{k - 2} \left( (k - 1)^s - 1 \right) \right) \\
& = - \frac{2(t + 1)}{k - 2} + \frac{k}{k - 2} \sum\limits_{s = 0}^t (k - 1)^s \\
& = - \frac{2(t + 1)}{k - 2} + \frac{k }{(k - 2)^2} \left( (k - 1)^{t + 1} - 1 \right).
\end{align*}
To compute $f_1(t)$, we start by computing $|\cN_v(s) \setminus \cN_u(s)|$. Let $u_1, \ldots, u_k$ be the neighbors of $u$ in $G$ and let $S_1, \ldots, S_k$ be a partition of the vertices exactly distance $s$ from $u$, such that the path connecting $u$ and a vertex in $S_i$ must cross $u_i$. Simple counting arguments show that $|S_i| = (k-1)^{s-1}$ for each $i$, and that if we assume without loss of generality that $u_1 = v$, 
$$
|\cN_u(s) \setminus \cN_v(s) | = |S_2 \cup S_3 \cup \ldots \cup S_k| = (k-1)^s.
$$
Hence we have
$$
f_1(t) = \sum\limits_{s = 0}^t (k-1)^s = \frac{(k-1)^{t + 1} - 1}{k - 2}.
$$
The first-order behavior of $F_1 = f_1^{-1}$ is a direct consequence.

We now study $F(z)$, the inverse function of $f$. Substituting $t = \frac{\log z}{\log (k - 1)}$ in the formula for $f(t)$, we have
\begin{align*}
f  \left( \frac{\log z}{\log (k - 1)} \right) & = \frac{k}{(k - 2)^2} \left( (k - 1)^{\frac{\log z}{\log(k - 1)} + 1} - 1 \right)  - \frac{2}{k - 2} \left( \frac{\log z}{\log (k - 1)} + 1 \right) \\
& = \frac{k(k-1)}{(k - 2)^2} z - \frac{k}{(k - 2)^2}  - \frac{2}{k - 2} \left( \frac{\log z}{\log (k - 1)} + 1 \right) 
\end{align*}
Since $\frac{k(k - 1)}{(k - 2)^2} > 1$, we have for $z$ sufficiently large (in particular, $z$ is larger than some function of $k$ alone) that $f \left( \frac{\log z}{\log (k - 1)} \right) \ge z$, which is equivalent to $F(z) \le \frac{\log z}{\log (k - 1)}$. On the other hand, 
\begin{align*}
f \left( \frac{\log z}{\log (k - 1)} - 3 \right) & = \frac{k}{(k - 1)^2 (k - 2)^2} z - \frac{k}{(k - 2)^2}  - \frac{2}{k - 2} \left( \frac{\log z}{\log (k - 1)} - 3 \right) \\
& \le \frac{k}{(k - 1)^2 (k - 2)^2} z \\
& \le \frac{k}{(k - 1)^2} z.
\end{align*}
Since $k < (k - 1)^2$ for $k \ge 3$, we have $f \left( \frac{\log z}{\log (k - 1)} - 3 \right) \le z$, which in turn implies that $F(z) \ge \frac{\log z}{\log (k - 1)} - 3$. 
\end{proof}

A useful corollary of \eqref{eq:trees_N} is a characterization of $r_n$ in $k$-regular trees. 

\begin{corollary}
\label{cor:trees_rn}
Let $G$ be a $k$-regular tree and let $\{V_n \}_{n \ge 1}$ be a sequence of candidate sets satisfying Assumption \ref{as:candidate_set}. Then 
$$
r_n = \left \lfloor \frac{ \log \left( \frac{k-2}{k} (n-1) + 1 \right) }{ \log (k - 1) } \right \rfloor \sim \frac{ \log n}{ \log (k -1 )}.
$$
\end{corollary}

The following lemma computes the asymptotic behavior for $|\partial \cN(t)|, |\cN(t)|, f(t)$ and $F(z)$ in lattices. 

\begin{lemma}
\label{lemma:lattices_exact}
Let $G$ be a $\ell$-dimensional lattice. Then
\begin{align}
\label{eq:lattice_partialN}
| \partial \cN(t) | & \sim \frac{2^\ell}{(\ell - 1)!} t^{\ell - 1}; \\
\label{eq:lattice_N}
| \cN(t) | & \sim \frac{2^\ell}{ \ell!} t^\ell; \\
\label{eq:f_N}
f(t) & \sim \frac{2^\ell}{ (\ell + 1)!} t^{\ell + 1}; \\
\label{eq:F_N}
F(z) & \sim \left( \frac{(\ell + 1)!}{2^\ell} z \right)^{\frac{1}{\ell + 1}}.
\end{align}
\end{lemma}

\begin{proof}
Recall that $\Z$ is the set of integers. For a vector $x \in \Z$, let $\norm{x}_0$ denote the number of nonzero entries of $x$ and let $\norm{x}_1$ denote the $\ell_1$ norm of $x$. For every integer $1 \le k \le \ell$ and an integer $t \ge 0$, define the set 
$$
\cS_k(t) : = \left \{ x \in \Z^d: \norm{x}_0 = k \text{ and } \norm{x}_1 = t \right \}.
$$
Since the $\cS_k(t)$'s partition $\partial \cN_0(t)$, we have $| \partial \cN(t) | = \sum_{k = 1}^\ell | \cS_k(t)|$. We proceed by computing the size of $|\cS_k(t)|$ via combinatorial arguments. First, we choose the $k$ nonzero coordinates of a vector in $\cS_k(t)$; this can be done in ${\ell \choose k}$ ways. Next, note that the number of positive integer solutions to $y_1 + \ldots + y_k = t$ is exactly ${t - 1 \choose k - 1}$ if $t \ge k$ else it is 0; this can be seen through standard counting arguments. Now, since the number of vectors in $\cS_k(t)$ for which the absolute value of the entries are given by $y_1, \ldots, y_k$ (in that order) is $2^k$ (since each nonzero entry of $x$ can be positive or negative), we may put everything together to obtain 
$$
| \cS_k(t) | = 2^k {\ell \choose k} {t - 1 \choose k - 1} \text{ if $t \ge k$, else 0.}
$$
When $t$ is large, the first-order term of $|S_k(t)|$ is $\frac{2^k}{(k-1)!} {\ell \choose k} t^{k - 1}$. It follows that 
\begin{equation}
\label{eq:partialN_first_order}
| \partial \cN(t) | \sim |\cS_\ell(t) | \sim \frac{2^\ell}{(\ell-1)!} t^{\ell - 1}. 
\end{equation}
Next, we use \eqref{eq:partialN_first_order} to obtain the first-order behavior of $| \cN(t) |$. To this end, we first note that for any $p \ge 0$, approximating a summation by an integral gives
\begin{align*}
\frac{1}{p + 1} & \left( k_1^{p + 1} - (k_0 - 1)^{p + 1} \right) = \int_{k_0 - 1}^{k_1} s^p ds \le \sum\limits_{k = k_0}^{k_1} k^p \\
& \le \int_{k_0}^{k_1 + 1} s^p ds = \frac{1}{p + 1} \left( (k_1 + 1)^{p + 1} - k_0^{p + 1} \right).
\end{align*}
In particular, when $k_1$ is much larger than $k_0$, 
$$
\sum\limits_{k = k_0}^{k_1} k^p \sim \frac{1}{p + 1} k_1^{p + 1}.
$$
The first-order term of $|\cN(t)|$ is therefore 
$$
| \cN(t) | = \sum\limits_{s = 0}^t | \partial \cN(t) | \sim \sum\limits_{s = \ell}^t \frac{2^\ell}{(\ell-1)!} s^{\ell - 1} \sim \frac{2^\ell}{\ell!} t^\ell. 
$$
Through analogous arguments, $f(t) \sim \frac{2^\ell}{(\ell + 1)!} t^{\ell + 1}$. The first order behavior of $F$ is an immediate consequence. 
\end{proof}

A useful corollary of \eqref{eq:lattice_N} is a characterization of $r_n$ in $\ell$-dimensional lattices. 

\begin{corollary}
\label{cor:lattice_rn}
Let $G$ be a $\ell$-dimensional lattice $\{V_n \}_{n \ge 1}$ be a sequence of candidate sets satisfying Assumption \ref{as:candidate_set}. Then 
$$
r_n \sim \left( \frac{\ell!}{2^\ell} n \right)^{1/\ell}.
$$
\end{corollary}

\section{Summations of geodesics}
\label{sec:geodesics}

The goal of this section is to bound summations of the form $\sum_{w \in V_n} d(w, v)$ and $\sum_{w \in V_n} d(w, v)^2$, which are useful in studying the Bayesian formulation of the quickest source estimation problem. 

\subsection{Regular trees}

We begin by proving a few intermediate results. The following lemma will provide a useful lower bound for $\sum_{w \in V_n} d(w,v)$. 

\begin{lemma}
\label{lemma:tree_distance_comparison}
Let $G$ be a regular tree, let $v_0 \in V$ and let $r$ be a positive integer. Then for all $v \in \cN_{v_0}(r)$, 
$$
\sum\limits_{w \in \cN_{v_0}(r)} d(w, v) \ge \sum\limits_{w \in \cN_{v_0}(r)} d(w, v_0). 
$$
\end{lemma}

To prove the lemma, we will make use a notion of centrality in trees. Let $G_r$ be the finite $k$-regular tree restricted to the vertex set $\cN_{v_0}(r)$. For a given vertex $v$ of $G_r$, label the neighbors of $v$ by $v^1, \ldots, v^k$. Let $\cS_i(v)$ be the set of vertices in $\cN_{v_0}(r)$ such that the path connecting $w$ and $v$ includes $v^i$. Notice that if we root $G_r$ at $v$, the sets $\{\cS_i(v) \}_{i = 1}^k$ correspond to subtrees of the rooted tree and we have the partition
\begin{equation}
\label{eq:subtree_partition}
\cN_{v_0}(r) \setminus \{v \} = \bigcup\limits_{i= 1}^k \cS_i(v).
\end{equation}
We say that $v$ is a {\it centroid} of $G_r$ if 
\begin{equation}
\label{eq:centrality_condition}
\max\limits_{1 \le i \le k} | \cS_i(v) | \le \frac{| \cN_{v_0}(r)|}{2}.
\end{equation}
A consequence of \eqref{eq:centrality_condition} is that if $v$ is {\it not} a centroid and $| \cS_1(v) | = \max_{1 \le i \le k} | \cS_i(v) |$,  we must have $| \cS_1(v) | \ge | \cN_{v_0}(r)|/2 + 1$. Since $\sum_{i = 1}^k | \cS_i(v) | = | \cN_{v_0}(r) | - 1$, it follows that $\sum_{i = 2}^k | \cS_i(v) | \le | \cN_{v_0}(r) |/2 - 2$ which in turn implies
\begin{equation}
\label{eq:not_centroid_condition}
| \cS_1(v) | \ge \sum\limits_{i = 2}^k | \cS_i(v) | + 3.
\end{equation}
In general, a tree may have at most two centroids, in which case the centroids are neighbors \cite[Lemma 2.1]{jog_loh}. This leads us to the following result. 

\begin{proposition}
\label{prop:unique_centroid}
The unique centroid of $G_r$ is $v_0$. 
\end{proposition}

\begin{proof}
We first show that $v_0$ is indeed a centroid of $G_r$. Notice that if we root $G_r$ at $v_0$, the rooted tree is balanced and in particular, $| \cS_i(v_0) | = (| \cN_{v_0}(r) |-1)/k$ for all $i \in \{1, \ldots, k \}$. Since $k \ge 2$, \eqref{eq:centrality_condition} is satisfied. 

Next, suppose by contradiction that $v_0$ is {\it not} the unique centroid. Without loss of generality, assume that $v_0^1$ is also a centroid. However, since $G_r$ rooted at $v_0$ is balanced, all neighbors of $v_0$ are isomorphic\footnote{More precisely, for each pair of neighbors of $v_0$, we can find a graph homomorphism mapping one neighbor to the other.} so all vertices in the collection $\{v_0^i \}_{i = 1}^k$ must also be centroids. But since $k \ge 2$, this implies that there are at least 3 centroids, which is a contradiction. 
\end{proof}

We are now ready to prove Lemma \ref{lemma:tree_distance_comparison}. 

\begin{proof}[Proof of Lemma \ref{lemma:tree_distance_comparison}]
Without loss of generality, we shall assume that $| \cS_1(v) | = \max_{1 \le i \le k} | \cS_i(v) |$. We can then write
\begin{align}
 \sum\limits_{w \in \cN(v_0, r)}  d(w,v) & \stackrel{(a)}{=} \sum\limits_{w \in \cS_1(v)} d(w, v) + \sum\limits_{i = 2}^k \sum\limits_{w \in \cS_i(v)} d(w,v) \nonumber \\
& \stackrel{(b)}{=} \sum\limits_{w \in \cS_1(v)} ( d(w, v^1) + 1) + \sum\limits_{i = 2}^k \sum\limits_{w \in \cS_i(v)} (d(w, v^1) - 1) \nonumber \\
& = \sum\limits_{i = 1}^k \sum\limits_{w \in \cS_i(v)} d(w, v^1) + |\cS_1(v) | - \sum\limits_{i = 2}^k | \cS_i(v) | \nonumber \\
\label{eq:neighbor_centrality_equivalence}
& \stackrel{(c)}{=} \sum\limits_{w \in \cN(v_0, r)} d(w, v^1) - 1 + | \cS_1(v) | - \sum\limits_{i = 2}^k | \cS_i(v) |,
\end{align}
where $(a)$ and $(c)$ are due to \eqref{eq:subtree_partition}, and $(b)$ follows since $v^1, v$ are neighbors and $v^1$ is closer to $\cS_1(v)$ than $v$, and $v$ is closer to $\cS_i(v)$ than $v^1$ for $1 \le i \le k$. If $v$ is not a centroid, we can apply \eqref{eq:not_centroid_condition} to \eqref{eq:neighbor_centrality_equivalence} to obtain 
$$
\sum\limits_{w \in \cN(v_0, r)} d(w,v)  \ge \sum\limits_{w \in \cN(v_0, r)} d(w, v^1) +2.
$$
In light of Proposition \ref{prop:unique_centroid}, this shows that if $v \neq v_0$, 
$$
\sum\limits_{w \in \cN(v_0, r)} d(w,v) > \min\limits_{u \in \cN_{v_0}(r)} \sum\limits_{w \in \cN_{v_0}(r) } d(w,u).
$$
The only remaining vertex, $v_0$, must therefore be the minimizer. 
\end{proof}

The main result for regular trees follows readily from the intermediate results we have established. 

\begin{lemma}
\label{lemma:tree_distance_bounds}
Let $G$ be a $k$-regular tree and let $v \in V_n$. Then for $n$ sufficiently large, 
\begin{equation}
\label{eq:distance_lower_bound}
\sum\limits_{w \in V_n} d(w, v) \ge \frac{n \log n}{k \log (k - 1)}
\end{equation}
and 
\begin{equation}
\label{eq:distance_upper_bound}
\sum\limits_{w \in V_n} d(w, v)^2 \le \frac{4 n \log^2 n}{\log^2 (k - 1)}.
\end{equation}
\end{lemma}

\begin{proof}
Noting that $\cN_{v_0}(r_n) \subseteq V_n$, we have 
\begin{align*}
\sum\limits_{w \in V_n} d(w, v_0) & \ge r_n | \partial \cN(r_n)| \\ & \stackrel{(a)}{=} r_n k(k-1)^{r_n - 1} \\
& \stackrel{(b)}{\ge} \frac{r_n k}{(k-2)^2} \left( \frac{k-2}{k}(n-1) + 1 \right) \\
& \stackrel{(c)}{\ge} \frac{n \log n}{k \log (k-1)}.
\end{align*}
Above, $(a)$ is due to \eqref{eq:trees_partialN}, $(b)$ uses the formula for $r_n$ in Corollary \ref{cor:trees_rn} and $(c)$ holds for $n$ sufficiently large, as it lower bounds the first-order term in the previous expression. Equation \eqref{eq:distance_lower_bound} now follows from an application of Lemma \ref{lemma:tree_distance_comparison}. 

Next, , we upper bound the squared sum of the distances. Since $V_n \subset \cN_{v_0}(r_n + 1)$, the diameter of $V_n$ is at most $2r_n + 2$. It follows that, for $n$ sufficiently large, 
$$
\sum\limits_{w \in V_n} d(w, v)^2 \le n (2r_n + 2)^2 \sim  \frac{4 n \log^2 n}{\log^2 (k - 1)},
$$
where we have used the asymptotic behavior of $r_n$ derived in Corollary \ref{cor:trees_rn}. 
\end{proof}

\subsection{Lattices}

We first prove an intermediate results. The following lemma is an analogue of Lemma \ref{lemma:tree_distance_comparison} for the case of lattices. 

\begin{lemma}
\label{lemma:lattice_distance_comparison}
Let $G$ be a $\ell$-dimensional lattice, let $v_0 \in V$ and let $r$ be a positive integer. Then for all $v \in V$, 
\begin{equation}
\label{eq:lattice_distance_comparison}
\sum\limits_{w \in \cN_{v_0}(r)} d(w,v) \ge \sum\limits_{w \in \cN_{v_0}(r)} d(w, v_0).
\end{equation}
\end{lemma}

\begin{proof}
Assume that vertices are labelled by their coordinates in $\R^d$. It follows that for $u, v \in V$, $d(u,v) = \norm{u - v}_1$. We can then write, for any $v \in V$, 
$$
\sum\limits_{w \in \cN_{v_0}(r) } d(w,v) = \sum\limits_{w \in \cN_{v_0}(r)} \norm{w - v}_1 = \sum\limits_{i = 1}^d \sum\limits_{w \in \cN_{v_0}(r)} | v_i - w_i |.
$$ 
The value of $v_i$ that minimizes $\sum_{w \in \cN_{v_0}(r)} | v_i - w_i |$ is the median of the collection $\{w_i \}_{w \in \cN_{v_0}(r)}$, which is $(v_0)_i$ (the $i$th component of the vector $v_0$) due to the symmetry of the set $\cN_{v_0}(r)$. As this argument holds for each $i$, \eqref{eq:lattice_distance_comparison} follows. 
\end{proof}

The following result contains the desired bounds for $\sum_{w \in V_n} d(w,v)$ and $\sum_{w \in V_n} d(w,v)^2$. 

\begin{lemma}
\label{lemma:geodesic_lattice}
Let $G$ be a $\ell$-dimensional lattice. There exist constants $c_3, c_4 > 0$ depending only on $d$ such that for all $v \in V_n$, 
\begin{align}
\label{eq:lattice_geodesic_lower_bound}
\sum\limits_{w \in V_n} d(w,v) & \ge c_3 n^{1 + \frac{1}{\ell}};  \\
\label{eq:lattice_geodesic_upper_bound}
 \sum\limits_{w \in V_n} d(w,v)^2 & \le c_4 n^{1 + \frac{2}{\ell}}.
\end{align}
\end{lemma}

\begin{proof}
\begin{align*}
\sum\limits_{w \in V_n} d(w, v_0) & \ge \sum\limits_{w \in \cN_{v_0}(r_n)} d(w, v_0) \\
& = \sum\limits_{k = 1}^{r_n} k | \partial \cN(k) | \\
& \stackrel{(a)}{\sim} \frac{2^\ell}{(\ell - 1)!} \sum\limits_{k = 1}^{r_n} k^\ell \\
& \sim \frac{\ell 2^\ell}{(\ell + 1)!} r_n^{\ell + 1} \\
& \stackrel{(b)}{\sim} \frac{\ell 2^\ell}{(\ell + 1)!} \left( \frac{\ell!}{2^\ell} n \right)^{1 + \frac{1}{\ell}}.
\end{align*}
Above, $(a)$ follows from the formula for $| \partial \cN(k) |$ in \eqref{eq:lattice_partialN} and $(b)$ is due to the asymptotics of $r_n$ given in Corollary \ref{cor:lattice_rn}. Equation \eqref{eq:lattice_geodesic_lower_bound} follows. 

Next, we upper bound the sum of the squared distances. Noting that the diameter of $V_n$ is at most $2(r_n + 1)$, we have
$$
\sum\limits_{w \in V_n} d(w, v)^2 \le 4(r_n + 1)^2 n \sim 4 \left( \frac{\ell!}{2^\ell} \right)^{\frac{2}{\ell}} n^{1 + \frac{2}{\ell}},
$$
where the asymptotics of the final expression are obtained from the asymptotics of $r_n$ found in Corollary \ref{cor:lattice_rn}. Equation \eqref{eq:lattice_geodesic_upper_bound} follows. 
\end{proof}

\bibliographystyle{abbrv}
\bibliography{citations}

\end{document}